\documentclass{amsart}

\usepackage{geometry, amssymb, verbatim, esint}
\usepackage{color}

\usepackage[colorlinks=true,urlcolor=blue, citecolor=red,linkcolor=blue,
linktocpage,pdfpagelabels, bookmarksnumbered,bookmarksopen]{hyperref}
\usepackage[hyperpageref]{backref}
\usepackage{tikz}
\usepackage{mathrsfs}
\usepackage{soul}
\newtheorem{teo}{Theorem}[section]
\newtheorem{lm}[teo]{Lemma}
\newtheorem{coro}[teo]{Corollary}
\newtheorem{prop}[teo]{Proposition}
\newtheorem*{main}{Main Theorem}
\theoremstyle{definition}
\newtheorem{definition}[teo]{Definition}
\newtheorem{exa}[teo]{Example}
\newtheorem{rem}[teo]{Remark}
\newtheorem*{ack}{Acknowledgments}

\numberwithin{equation}{section}

\title[Extremals for Poincar\'e-Sobolev constants]{Extremals for Poincar\'e-Sobolev sharp constants\\ in Steiner symmetric sets}
\date{\today}
\author[Brasco]{Lorenzo Brasco}
\address[L.\ Brasco]{Dipartimento di Matematica e Informatica
	\newline\indent
	Universit\`a degli Studi di Ferrara
	\newline\indent
	Via Machiavelli 35, 44121 Ferrara, Italy}
\email{lorenzo.brasco@unife.it}

\author[Briani]{Luca Briani}
\address[L.\ Briani]{School of Computation, Information and Technology,		\newline\indent
 	Technical University of Munich
 	\newline\indent
	Boltzmannstra\ss e 3, 85748 Garching bei M\"unchen, Germany}
\email{luca.briani@tum.de}

\author[Prinari]{Francesca Prinari}

\address[F. Prinari]{Dipartimento di Scienze Agrarie, Alimentari e Agro-ambientali
\newline\indent 
Universit\`a di Pisa
\newline\indent
Via del Borghetto 80, 56124 Pisa, Italy}
\email{francesca.prinari@unipi.it}

\begin{document}

\begin{abstract}
We prove existence of minimizers for the sharp Poincar\'e-Sobolev constant in general Steiner symmetric sets, in the subcritical and superhomogeneous regime. The sets considered are not necessarily bounded, thus the relevant embeddings may suffer from a lack of compactness. We prove existence by means of an elementary compactness method. We also prove an exponential decay at infinity for minimizers, showing that in the case of Steiner symmetric sets the relevant estimates only depend on the underlying geometry. Finally, we illustrate the optimality of the existence result, by means of some examples. 
\end{abstract}

\maketitle

\begin{center}
\begin{minipage}{10cm}
\small
\tableofcontents
\end{minipage}
\end{center}

\section{Introduction}

\subsection{A starting point: the two-dimensional strip}
Motivated by obtaining existence of solitary waves for a two-dimensional, inviscid and incompressible flow, 
in \cite[Theorem 7.5 \& Corollary 7.6]{AmTo}, Amick and Toland proved the following remarkable result: for every $q>2$ and every $\lambda>0$ the following boundary value problem for the {\it Lane-Emden equation}
\begin{equation}
\label{LEintro}
\left\{\begin{array}{rcll}
-\Delta u&=&\lambda\, |u|^{q-2}\,u,& \text{in}\ S:=\mathbb{R}\times(-1,1),\\
u&=&0,& \text{on}\ \partial S
\end{array}
\right.
\end{equation}
does admit a {\it positive} solution with finite energy, i.e. such that $u\in W^{1,2}_0(S)$. Concurrently with them, Bona, Bose and Turner obtained the same kind of result in \cite[Theorem 5.1]{BBT}: we refer the reader to their paper \cite{BBT} for a thorough explanation of the physical model leading to this problem\footnote{It is fair to recall that the equation considered in \cite{AmTo, BBT} is slightly more general, since they admit a broader class of nonlinear right-hand sides: for the scopes of this paper, we will confine ourselves to the case of the simple equation above.}.
\par
From the variational point of view, this result may appear a little bit surprising: 
\begin{itemize}
\item at first, because the embedding $W^{1,2}_0(S)\hookrightarrow L^q(S)$ is only continuous, but not compact, due to the translation invariance of $S$ in the direction $\mathbf{e}_1=(1,0)$. This lack of compactness prevents the application of standard methods to prove existence in Sobolev spaces. The two main ones would be: either to look for a critical point of the functional
\[
\mathfrak{F}_{q,\lambda}(\varphi):=\frac{1}{2}\,\int_S |\nabla \varphi|^2\,dx-\frac{\lambda}{q}\,\int_S |\varphi|^q\,dx,
\]
by using the Mountain Pass Theorem, for example (see \cite[Chapter II, Theorem 6.2]{St});
or to look for a global minimizer of the following constrained problem
\[
\lambda_{2,q}(S):=\inf_{\varphi\in W^{1,2}_0(S)} \left\{\int_S |\nabla \varphi|^2\,dx\, :\, \int_S |\varphi|^q\,dx=1\right\}.
\]
Indeed, observe that a minimizer $v$ of this problem (which can be taken to be positive, thanks to the fact that both the functional and the constraint are even), would solve
\[
-\Delta v=\lambda_{2,q}(S)\,v^{q-1},\qquad \text{in}\ S.
\]
Then, by taking the scaled function $u=(\lambda/\lambda_{2,q}(S))^{1/(2-q)}\,v$, we would get a positive solution of \eqref{LEintro}. This particular solution of the Lane-Emden equation is called {\it with minimal energy}, since it minimizes the free energy functional $\mathfrak{F}_{q,\lambda}$ among all nontrivial solutions (see for example the proof of \cite[Corollary 1.4]{BraLin});
\vskip.2cm
\item secondly, this existence result is {\it false} for $1<q\le 2$. Indeed, in the case $1<q<2$ we know that every positive solution $u\in W^{1,2}_0(S)$ to \eqref{LEintro} would automatically be a minimizer of the functional $\mathfrak{F}_{q,\lambda}$ above (see for example \cite[Proposition 3.5 \& Remark 3.6]{BPZ}). However, it is not difficult to see that this functional is unbounded from below, due to the scale invariance of $S$ along the $\mathbf{e}_1$ direction. Indeed, for every non-trivial $\varphi\in W^{1,2}_0(S)$ and every $\ell>0$, by defining $\varphi_\ell(x_1,x_2)=\varphi(\ell\,x_1,x_2)$, after a change of variable we would get
\[
\lim_{\ell\to 0^+}\mathfrak{F}_{q,\lambda}(\varphi_\ell)=\lim_{\ell\to 0^+}\left[\frac{\ell}{2}\,\int_S |\nabla \varphi|^2\,dx-\frac{\lambda}{q}\,\frac{1}{\ell}\,\int_S |\varphi|^q\,dx\right]=-\infty.
\]
The borderline case $q=2$ is quite peculiar, since \eqref{LEintro} boils down to the eigenvalue equation. In this case, the only possibility to have existence of a positive solution $u$ (which thus would be a positive eigenfunction) would be that $\lambda$ coincides with the first eigenvalue of the Dirichlet-Laplacian on $S$, i.e.
\[
\lambda=\inf_{\varphi\in W^{1,2}_0(S)} \left\{\int_S |\nabla \varphi|^2\,dx\, :\, \int_S |\varphi|^2\,dx=1\right\},
\]
with $u$ minimizer of this problem. However, it is well-known that on the strip $S$ this problem does not admit a solution. 
\end{itemize}
In light of the previous observations, at that time the existence result of \cite{AmTo} and \cite{BBT} greatly stimulated the interest towards related results for semilinear or quasilinear PDEs in unbounded domains, where standard compactness methods may fail. Without any attempt of completeness, it is mandatory to recall the papers by Esteban and Lions \cite{Es, EL2, EL}, appeared at the same time as \cite{AmTo, BBT} (see also \cite[Section V.2]{Lio}). There, the authors generalized this existence result by considering more general cylindrical sets of the form $\Omega=\omega\times\mathbb{R}^k\subseteq \mathbb{R}^N$, with $\omega\subseteq \mathbb{R}^{N-k}$ open bounded set and $1\le k\le N-1$. We also refer to \cite{DelF} and \cite{KST} for related existence results for the same equation in some unbounded sets, under suitable geometric structural assumptions.

\subsection{Main results}
It is now time to present the main scopes and outcomes of our paper. We need at first to settle some notation.
For $1<p<N$, we define the Sobolev critical exponent
\[
p^*=\frac{N\,p}{N-p},
\]
as it is customary. Then, for every exponent $q>p$ such that
\begin{equation}\label{pq} 
q \left\{\begin{array}{ll}
< p^*,& \mbox{ if } p<N,\\
<\infty, & \mbox{ if } p=N,\\
\le \infty,& \mbox{ if } p> N,
\end{array}
\right.
\end{equation}
we define 
\begin{equation}\label{lapq}
\lambda_{p,q}(\Omega)=\inf_{\varphi\in C^\infty_0(\Omega)} \left\{\int_\Omega |\nabla\varphi|^p\,dx\, :\, \|\varphi\|_{L^q(\Omega)}=1\right\}.
\end{equation}
Observe that this coincides with the sharp constant for the embedding
\begin{equation}
\label{embeddo}
\mathscr{D}^{1,p}_0(\Omega)\hookrightarrow L^q(\Omega),
\end{equation}
where $\mathscr{D}^{1,p}_0(\Omega)$ is the {\it homogeneous Sobolev space}, given by the completion of $C^\infty_0(\Omega)$ with respect to the norm
\[
C^\infty_0(\Omega)\ni\varphi\mapsto\|\nabla\varphi\|_{L^p(\Omega)}.
\]
By keeping in mind the existence result on the two-dimensional strip, in this paper we will focus on proving existence of minimizers for \eqref{lapq}
for sets which could be ``very large'' (so that the relevant Sobolev embedding may fail to be compact), but having some geometric properties which permit to retrieve a suitable form of compactness, at least for minimizing sequences. In turn, this will show existence of solutions with minimal energy to the following quasilinear Lane-Emden equation
\begin{equation}
\label{LEgenintro}
-\Delta_p u=\lambda\,u^{q-1},\qquad \text{in}\ \Omega,
\end{equation}
with the power $q$ being {\it superhomogeneous} (i.e. $q>p$) and {\it subcritical} (in the sense of Sobolev embeddings). For the case $p=2$, apart for the physical motivation described in \cite{BBT}, we recall that positive solutions of \eqref{LEgenintro} describe the extinction profiles of positive solutions to the so-called {\it fast diffusion equation} with Dirichlet boundary conditions (see the pioneering result \cite[Theorem 2]{BH} or the thorough discussion of \cite[Section 9]{BonFig}). 
\par
We will also be interested in deriving some properties of minimizers for \eqref{lapq} (hereafter also called {\it extremals}), notably their behaviour at infinity.
\vskip.2cm\noindent
More precisely, we will work with open {\it Steiner symmetric sets} $\Omega\subseteq\mathbb{R}^N$, i.e. sets 
such that:
\begin{itemize}
\item[(S1)] $\Omega$ is symmetric with respect to each hyperplane $\{x\, :\,\langle x,\mathbf{e}_i\rangle =0\}$;
\vskip.2cm
\item[(S2)] $\Omega$ is convex in every direction $\mathbf{e}_i$.
\end{itemize}
Here  $\{\mathbf{e}_1,\dots,\mathbf{e}_N\}$ stands for the canonical orthonormal basis of $\mathbb{R}^N$, as usual.
We will show that, for these sets, the existence result of \cite{AmTo, BBT} still holds, for every $1<p<\infty$ and every $q$ as above. Observe that our sets do not necessarily have a cylindrical structure. 
In this respect, the existence part of this paper only partially superposes with the aforementioned results by Esteban and Lions, for example. It should be also pointed out that our methods of proof will be {\it elementary} (for example, we will not appeal to any {\it concentration-compactness method}), as we will explain in the next subsections.
Our paper can find \cite{BBT} as its true ancestor, from the mathematical point of view.
\vskip.2cm\noindent
We summarize below the main findings of the present paper: this will result by collecting together Theorems \ref{teo:main}, \ref{teo:maininfty} and \ref{teo:decaysteiner} below.
\begin{main}
Let $1<p<\infty$ and assume that  $q>p$ satisfies \eqref{pq}.
Let $\Omega\subseteq\mathbb{R}^N$ be an open Steiner symmetric set such that  $\Omega\neq\mathbb R^N$. Then, the infimum defining $\lambda_{p,q}(\Omega)$ is attained by some Steiner symmetric function  $u\in W^{1,p}_0(\Omega)\setminus\{0\}$. Moreover, every Steiner symmetric extremal satisfies the following exponential decay estimate
\[
|u(x)|\le C\,e^{-a|x|}, \qquad\text{for}\ x\in\mathbb{R}^N,
\]
with $a=a(p,q,\lambda_{p,p}(\Omega))>0$ and $C=C(N,p,q,\Omega)>0$.
\end{main}
\vskip.2cm\noindent
Observe that, under the standing assumptions of the Main Theorem, we admit unbounded open sets which could be quite ``large'', as already anticipated. For example, these sets could be unbounded in every direction and have infinite volume.
Thus, in principle, it is not even clear that $\lambda_{p,q}(\Omega)>0$, i.e. that the continuous embedding \eqref{embeddo} holds. Actually, this is the case, as we show in Section \ref{sec:3}: more precisely, we prove at first that each open Steiner symmetric set $\Omega\not=\mathbb{R}^N$ has finite {\it inradius}. The latter is the quantity defined by
\[
r_\Omega=\sup\Big\{r>0\,:\, \text{there exists a ball}\ B_r(x)\subseteq\Omega\Big\}.
\]
It is well-known that the condition $r_\Omega<+\infty$ is only a necessary one, for the embedding \eqref{embeddo} to hold. We show that for a Steiner symmetric set {\it this is sufficient}, as well, through the following quantitative geometric lower bound
\begin{equation}\label{lowerboundintro}
\lambda_{p,q}(\Omega)\ge \frac{1}{C_{N,p,q}}\,\left(\frac{1}{r_\Omega}\right)^{N-p-\frac{p}{q}\,N},
\end{equation} 
which is interesting in itself.
\vskip.2cm\noindent
Before proceeding further with some comments on the proof of the Main Theorem, a list of remarks are in order.
\begin{rem}[The case $q\le p$]
For $q\le p$, the assumptions on $\Omega$ of the Main Theorem are not enough to ensure that $\lambda_{p,q}(\Omega)$ is attained. As a counter-example, we can take for simplicity $p=N=2$ and consider again the strip
\[
S=\mathbb{R}\times(-1,1).
\]
Then we recall that:
\begin{enumerate}
\item for $q<2$, in general we have that the embedding $\mathcal{D}^{1,2}_0(\Omega)\hookrightarrow L^q(\Omega)$ is continuous if and only if this is compact (see \cite[Theorem 15.6.2]{Maz}). Since $\mathcal{D}^{1,2}_0(S)\hookrightarrow L^q(S)$ can not be compact (due to the translation invariance of the set along the direction $\mathbf{e}_1$), we must have $\lambda_{2,q}(S)=0$;
\vskip.2cm
\item for $p=q=2$, we have that $\lambda_{2,2}(S)$ coincides with the bottom of the spectrum of the Dirichlet-Laplacian in $S$ and we have already observed that this is not attained.
%\[
%\lambda_{p,p}(\mathcal{O})=\inf_{\varphi\in C^{\infty}_0((-1,1))} \left\{\int_{-1}^1 |\varphi'|^p\,dt\, :\, \int_{-1}^1 |\varphi|^p\,dt=1\right\}>0.
%\]
%However, it is not difficult to see that $\lambda_{p,p}(\mathcal{O})$ can not be attained on $W^{1,p}_0(\mathcal{O})$. Indeed, one can argue by contradiction: suppose that $u\in W^{1,p}_0(\mathcal{O})$ attains the infimum $\lambda_{p,p}(\mathcal{O})$. By taking $0<\ell<1$ and considering the scaled function
%\[
%u_\ell(x)=u(\ell\,x_1,\dots,\ell\, x_{N-1},x_N),
%\]
%we would get (after a change of variable)
%\[
%\begin{split}
%\lambda_{p,p}(\mathcal{O})\le \frac{\displaystyle\int_{\mathcal{O}}|\nabla u_\ell|^p\,dx}{\displaystyle\int_{\mathcal{O}} |u_\ell|^p\,dx}=\frac{\displaystyle\int_{\mathcal{O}}\left(\ell^2\,\sum_{i=1}^{N-1}\left|\frac{\partial u}{\partial x_i}\right|^2+\left|\frac{\partial u}{\partial x_N}\right|^2\right)^\frac{p}{2}\,dx}{\displaystyle\int_{\mathcal{O}} |u|^p\,dx}< \frac{\displaystyle\int_{\mathcal{O}}|\nabla u_\ell|^p\,dx}{\displaystyle\int_{\mathcal{O}} |u_\ell|^p\,dx}=\lambda_{p,p}(\Omega).
%\end{split}
%\]
%This gives the desired contradiction. Observe that we used that $\ell<1$ and that 
%\[
%\int_{\mathcal{O}}\sum_{i=1}^{N-1}\left|\frac{\partial u}{\partial x_i}\right|^2\,dx>0,
%\] 
%otherwise $u$ would be a function of the variable $x_N$ only and could not belong to $L^p(\mathcal{O})$.
\end{enumerate}
\end{rem}

\begin{rem}[Breaking Steiner symmetry]
Our assumptions (S1) and (S2) on the open sets are, in a sense, necessary, whenever the embedding $\mathcal{D}^{1,p}_0(\Omega)\hookrightarrow L^q(\Omega)$ fails to be compact. We will show for example that there exist:
\begin{itemize}
\item {\it convex} sets, having only $N-1$ orthogonal hyperplanes of symmetry;
\vskip.2cm
\item sets with $N$ orthogonal hyperplanes of symmetry, but not convex along one of the coordinate directions;
\end{itemize}
such that $\lambda_{p,q}$ does not admit an extremal. We will also highlight that our existence result is quite unstable, with respect to small (smooth) deformations of the boundary: for example, while the strip $S=\mathbb{R}\times(-1,1)$ admits an extremal, by making a small symmetric {\it inward} deformation we get a new open set, having two orthogonal axis of symmetry, such that $\lambda_{p,q}$ is not attained.
\end{rem}

\begin{rem}[Uniqueness]
This is an interesting problem, which however we do not treat in this paper.
It is well-known that the uniqueness of positive solutions to the Lane-Emden \eqref{LEgenintro} is quite a delicate matter, in the case $q>p$. Even for solutions with minimal energy, i.e. positive extremals of $\lambda_{p,q}(\Omega)$, the known uniqueness results are very few. 
For example, we refer the interested reader to \cite[Theorem 1.1]{BraLin}, \cite[Theorem 5]{Da88}, \cite[Theorem 4.1]{DGP}, \cite[Theorem 1.1 \& Proposition 6.3]{HL} and \cite[Lemma 3 \& Theorem 1]{Lin},  for a positive answer in some particular cases, when $\Omega$ is {\it bounded}. However, extending these partial results to the case of unbounded sets as those considered in the Main Theorem does not seem to be straightforward. In the special case of the strip $S=\mathbb{R}\times(-1,1)$ and for $p=2$, uniqueness up to translation for positive solutions to \eqref{LEintro} seems to belong to the folklore on the subject: we refer to \cite[page 665]{Dan89}, which contains a very brief sketch of the proof.
\par
Finally, we recall that, in general, uniqueness fails already for some bounded sets, see for example \cite[Section 4.2]{BraFra}. 
\end{rem}

\subsection{Some comments on the proofs}

For the existence part, we will use the Direct Method in the Calculus of Variations. Here, we first point out that we do not prove that {\it every} minimizing sequence is sequentially precompact in $L^q(\Omega)$. We will rather construct a particular minimizing sequence which has this compactness property: this will be accomplished by considering the minimizers $u_n$ of a suitably ``perturbed'' problem, depending on a natural index $n\in\mathbb{N}$. 
\par
Such a problem consists in adding a ``vanishing confinement'' term, i.e. we consider  
\begin{equation}
\label{confino}
\inf_{u\in W^{1,p}_0(\Omega)} \left\{\int_\Omega |\nabla u|^p\,dx+\frac{1}{n+1}\,\int_\Omega V\,|u|^p\,dx\, :\, \|u\|_{L^q(\Omega)}=1\right\}.
\end{equation}
The precise form of the positive potential $V$ is not very important, but it must enjoy the following crucial properties:
\begin{itemize}
\item[{\it (V1)}] it has to be a {\it confining potential}, i.e. for every $t>0$ the sublevel set
$\{x\in\mathbb{R}^N\,:\, V(x)\le t\}$,
is bounded;
\vskip.2cm
\item[{\it (V2)}] for every non-negative function $u\in L^p(\Omega)$, we have 
\[
\int_\Omega V\,u^p\,dx\ge\int_{\Omega}V\,(\mathcal{S}_{\mathbf{e}_i}(u))^p\,dx,\qquad \text{for every}\ i\in\{1,\dots,N\}.
\]
where $\mathcal{S}_{\mathbf{e}_i}(u)$ stands for the {\it Steiner symmetrization of $u$ with respect to the direction $\mathbf{e}_i$}, see Section \ref{sec:3} below.
\end{itemize}
For example, it turns out that the simple choice $V(x)=|x|$ will do the job. The property {\it (V1)} guarantees that a minimizer $u_n\in W^{1,p}_0(\Omega)$ for \eqref{confino} actually exists; moreover, property {\it (V2)} assures that it can be taken to be Steiner symmetric, thanks to the so-called {\it P\'olya-Szeg\H{o} principle} for the Steiner symmetrization (see \cite{CF} or Section \ref{sec:3} below).
\par
In order to prove that $\{u_n\}_{n\in\mathbb{N}}$ is sequentially precompact in $L^q(\Omega)$ and get existence of an extremal for our original problem, the first crucial step is excluding that $\{u_n\}_{n\in\mathbb{N}}$ weakly converges to the null function. To this aim, we use the equation solved by $u_n$ and rely on a uniform (with respect to $n$) Harnack's inequality. This fact and the Steiner symmetry of each $u_n$ permit to conclude that $u_n$ is uniformly bounded from below on a fixed ball and thus to exclude the vanishing of the weak limit. 
\par
We then get the desired strong convergence by relying on the same idea as in the proof of \cite[Lemma 4.2]{CFL}, which deals with existence of an extremal for the sharp constant in a Gagliardo-Nirenberg-Sobolev inequality on the whole space. Here, we exploit the classical {\it Br\'ezis-Lieb Lemma} for the sequence $\{(u_n,\nabla u_n)\}_{n\in\mathbb{N}}$: as a technical point, we recall that in order to apply the Br\'ezis-Lieb Lemma, the weak convergence is not enough when $p\not=2$ (see \cite[Remarks, point (iii)]{BreLie}). Thus, we have to preliminary show that $\{(u_n,\nabla u_n)\}_{n\in\mathbb{N}}$ converges pointwise almost everywhere. For the functions themselves this is quite easy; for their gradients we can use a monotonicity trick, based on the ellipticity of the equation solved by each $u_n$. The last fact marks a difference with the case considered in \cite{CFL}, where $p=2$ and one can be dispensed with this additional pointwise convergence. As a byproduct of the proof, we actually get that $\{u_n\}_{n\in\mathbb{N}}$ converges (up to a subsequence) strongly in $W^{1,p}_0(\Omega)$.
\par
All these tools are quite robust and do not rely on higher regularity estimates (differently from the proofs of \cite{AmTo} and \cite{BBT}, for example). Thus, we can easily treat the case of a generic exponent $1<p<\infty$.
\vskip.2cm\noindent
We now want to make some comments on the decay estimate for extremals. We first remark that this is a consequence of a more general result, proved in Section \ref{sec:7}: this concerns non-negative weak {\it subsolutions} $u\in W^{1,p}_0(\Omega)$ of the Lane-Emden equation, i.e. 
\[
-\Delta_p u\le \lambda\,u^{q-1},\qquad \text{in}\ \Omega,
\]
in a general open set $\Omega\subseteq\mathbb{R}^N$, not necessarily Steiner symmetric.
\par
The result is obtained by using iterative methods {\it \`a la} De Giorgi-Moser. These are nonlinear methods in nature, thus here as well we can treat the case of a general exponent $1<p<\infty$. We point out in particular that we do not rely on ODE methods (like for example in the classical result \cite[Theorem 1]{BL}) or on the construction of suitable barrier functions (see for example \cite[Proposition 1]{CCW}).
The first step of the proof is to get an $L^\infty-L^q$ a priori estimate ``localized at infinity'', i.e. an estimate of the type
\[
\|u\|_{L^\infty(\Omega\setminus B_{\varrho+1})}\lesssim \|u\|_{L^q(\Omega\setminus B_{\varrho})},\qquad \mbox{for every}\ \varrho>0.
\] 
By combining this estimate with the fact that $u\in L^q(\Omega)$, we get that $u$ is arbitrarily small (in the $\sup$ norm) in the complement of a sufficiently large ball. More precisely, 
for every $\varepsilon>0$ there exists $R_\varepsilon>0$ such that 
\begin{equation}
\label{start}
\|u\|_{L^\infty(\Omega\setminus B_{R_\varepsilon})}<\varepsilon.
\end{equation}
Thanks to the fact that $q>p$, this in turn entails that
\[
-\Delta_p u\le \lambda\,\varepsilon^{q-p}\,u^{p-1},\qquad \text{in}\ \Omega\setminus B_{R_\varepsilon}.
\]
In other words, we get that $u$ is a subsolution ``at infinity'' of the eigenvalue equation for the $p-$Laplacian, with an eigenvalue which can be taken as small as we wish. The exponential decay now easily follows from this fact, provided we assume that $\Omega$ supports the Poincar\'e inequality
\[
c_\Omega\,\int_\Omega |\varphi|^p\,dx\le \int_\Omega |\nabla\varphi|^p\,dx,\qquad \text{for every}\ \varphi\in C^\infty_0(\Omega).
\]
As the reader may imagine, there is a subtle point here: in order to be a subsolution with a small eigenvalue, we need to choose a sufficiently large radius $R_\varepsilon$. This choice heavily depends on the (sub)solution $u$ itself and it affects the constant $C$ appearing in the final decay estimate. Thus, in general the resulting a priori estimate is a bit implicit.
\par
We will show in Theorem \ref{teo:decaysteiner} that, for Steiner symmetric sets and Steiner symmetric extremals, such an estimate only depends on the open set $\Omega$, i.e. the radius $R_\varepsilon$ at which the extremal ``starts to decay'' may be taken to depend {\it only on the underlying geometry} of $\Omega$, not on the extremal itself. More precisely, it depends on how the orthogonal sections of $\Omega$ ``shrink'' or ``persist'' at infinity (we refer the reader to Section \ref{sec:7} for a more precise discussion).
\par
The difficult point is to handle the case when $\Omega$ contains some unbounded tubes of fixed width (the prototypical case being the two-dimensional strip). In this case, we can get the claimed uniform result by carefully adapting an argument from \cite{BBT}. An explicit local Lipschitz estimate for extremals will be the key to get the result, by avoiding the stronger regularity used in \cite{BBT} for the case of the Laplacian (this could not be available in the case $p\not=2$ and for a general Steiner symmetric set).

\subsection{Plan of the paper}
In Sections \ref{sec:2} and \ref{sec:3} we introduce some notation and discuss some preliminary results, which will be used throughout the paper.  In particular, in Proposition \ref{prop:positivity}  we show the lower bound \eqref{lowerboundintro}.
In Section \ref{sec:4} we consider the minimizing sequence $\{u_n\}_{n\in\mathbb{N}}$ for the problem \eqref{lapq}, obtained by solving \eqref{confino}; we also discuss its main features. 
Then, Section \ref{sec:5} is devoted to the existence of extremals, by distinguishing the cases $q<\infty$ (see Theorem \ref{teo:main}) and $q=\infty$ (see Theorem \ref{teo:maininfty}). 
The optimality of the assumptions of these results are discussed  in Section \ref{sec:6}, by means of simple counterexamples. 
At last, in Section \ref{sec:7} we prove the decay properties of extremals: first, we show an estimate holding for every non-negative subsolution to the Lane-Emden equation, see Theorem \ref{teo:decay}; then, as discussed above, we refine such an estimate in the case of extremals of Steiner symmetric sets, see Theorem \ref{teo:decaysteiner}. 
\par
For completeness, in Appendix \ref{sec: app} we present some technical results concerning the case $q=\infty$, which are needed in the paper. 

\begin{ack}
The first author wish to thank Ryan Hynd, Simon Larson and Erik Lindgren for some useful discussions on the subject of this paper, during the thematic semester ``Geometric Aspects of Nonlinear Partial
Differential Equations'', held at the Institute Mittag-Leffler during Fall 2022. 
\par
Part of this paper has been written during the workshop ``CalVaFer -- Calculus of Variations in Ferrara'', held in Ferrara in March 2024, as well as during the ``XXXIV Convegno Nazionale di Calcolo delle Variazioni'', held in Riccione in February 2025.
\par
The second and third authors are members of the {\it Gruppo Nazionale per l'Analisi Matematica, la Probabilit\`a
e le loro Applicazioni} (GNAMPA) of the Istituto Nazionale di Alta Matematica (INdAM) whose support, through the project GNAMPA 2024  ``Propriet\`a geometriche e questioni di regolarit\`a in problemi variazionali locali e non local'' ({\tt CUP$\_$E53C23001670001}),  is gratefully acknowledged.  
\par
The research of L.~Briani was supported by the DFG through the Emmy Noether Programme (project number 509436910).  
\end{ack}

\section{Preliminaries}
\label{sec:2}

\subsection{Notation}
 For $r>0$ and $x_0\in\mathbb{R}^N$ we denote by  $B_r(x_0)$ the $N-$dimensional open ball centered at $x_0$, with radius $r>0$. 
When $x_0=0$, we will simply write $B_r$. We also denote $\omega_N=|B_1|$.
Occasionally, we will also need the notation
\[
Q_r(x_0)=\prod_{i=1}^N (x_{0,i}-r, x_{0,i}+r ) \qquad \hbox{ where } x_0=(x_{0,1},\cdots, x_{0,N}),
\]
for open $N-$dimensional hypercubes. As in the case of balls, we will simply write $Q_r$ when the center coincides with the origin.
\par
We define the {\it inradius} of an open set $\Omega\subseteq\mathbb{R}^N$ as
\[
r_\Omega=\sup\Big\{r>0\,:\, \text{there exists a ball}\ B_r(x)\subseteq\Omega\Big\}.
\]
%Finally,  	we denote by $C_0(\Omega)$ the completion of $C^\infty_0(\Omega)$ with respect to the sup norm.  
For an open set $\Omega\subseteq\mathbb{R}^N$ and an exponent $1<p<\infty$, we will use the distinguished notation
\[
\lambda_{p}(\Omega)=\inf_{u\in C^\infty_0(\Omega)} \left\{\int_\Omega |\nabla u|^p\,dx\, :\, \|u\|_{L^p(\Omega)}=1\right\}.
\] 
This is a particular instance of \eqref{lapq}, corresponding to the choice $q=p$.
%For $1<p<N$, we define the Sobolev critical exponent
%\[
%p^*=\frac{N\,p}{N-p},
%\]
%as it is customary. Then, more generally, for every exponent $q>p$ such that
%\begin{equation}\label{pq} 
%q \left\{\begin{array}{ll}
%< p^*,& \mbox{ if } p<N,\\
%<\infty, & \mbox{ if } p=N,\\
%\le \infty,& \mbox{ if } p> N,
%\end{array}
%\right.
%\end{equation}
%we define 
%\begin{equation}\label{lapq}
%\lambda_{p,q}(\Omega)=\inf_{u\in C^\infty_0(\Omega)} \left\{\int_\Omega |\nabla u|^p\,dx\, :\, \|u\|_{L^q(\Omega)}=1\right\}.
%\end{equation}
\par
We explicitly note that the infimum over $C^\infty_0(\Omega)$ in the definitions of both $\lambda_p(\Omega)$ and $\lambda_{p,q}(\Omega)$ can be equivalently performed on the whole $W^{1,p}_0(\Omega)$,
see for example \cite[Lemma 2.6]{BraPriZag2}. Here, by $W^{1,p}_0(\Omega)$ we intend the closure of $C^\infty_0(\Omega)$ in the standard Sobolev space $W^{1,p}(\Omega)$.

\subsection{Miscellaneous facts}
The following simple result expresses in a quantitative way the subadditivity of concave powers. It should be well-known, but we enclose the proof for the reader's convenience.
\begin{lm}
\label{lm:quantitativo}
Let $0<\alpha<1$, for every $A,B>0$ we have
\[
(A+B)^\alpha\le A^\alpha+B^\alpha-\frac{1-\alpha}{2^{\alpha+1}}\,\min\{A^\alpha,\,B^{\alpha}\}.
\]
\end{lm}
\begin{proof}
Without loss of generality, we can suppose that $B\le A$. We can also suppose that $B>0$, otherwise there is nothing to prove. If we set $t=B/A\in(0,1]$, we just need to prove that
\begin{equation}
\label{quantitativo}
(1+t)^\alpha\le 1+t^\alpha-\frac{1-\alpha}{2^{\alpha+1}}\,t^\alpha.
\end{equation}
To this aim, we introduce the function of one real variable
\[
h(t)=\frac{1+t^\alpha}{(1+t)^\alpha},\qquad\text{for every}\ t\in(0,1].
\]
We have
\[
h'(t)=\alpha\,\frac{t^{\alpha-1}-1}{(1+t)^{\alpha+1}}\ge\frac{\alpha}{2^{\alpha+1}}\,(t^{\alpha-1}-1),\qquad\text{for every}\ t\in(0,1].
\]
Thus, we get
\[
\begin{split}
h(t)=h(0)+\int_0^t h'(\tau)\,d\tau&\ge 1+\frac{\alpha}{2^{\alpha+1}}\,\int_0^t(\tau^{\alpha-1}-1)\,d\tau\\
&=1+\frac{1}{2^{\alpha+1}}\,(t^\alpha-\alpha\,t)\ge 1+\frac{1-\alpha}{2^{\alpha+1}}\,t^\alpha.
\end{split}
\]
By recalling the definition of $h$, this implies that
\[
1+t^\alpha\ge (1+t)^\alpha+(1+t)^\alpha\,\frac{1-\alpha}{2^{\alpha+1}}\,t^\alpha.
\]
By using that $(1+t)^\alpha\ge 1$, we get \eqref{quantitativo}, as desired.
\end{proof}
%\subsection{Standard estimates}

We recall the following particular family of {\it Gagliardo-Nirenberg interpolation inequalities}
	\begin{equation}
		\label{GNS}
		\|\psi\|_{L^r(\Omega)}\le G_{N,p,r}\,\|\nabla \psi\|_{L^p(\Omega)}^\vartheta\,\|\psi\|_{L^p(\Omega)}^{1-\vartheta},\qquad \mbox{ for every } \psi\in C^\infty_0(\Omega),
	\end{equation}
	which hold for every $1<p<\infty$ and $r>p$ satisfying \eqref{pq},
with $G_{N,p,r}>0$ only depending on $N,p$ and $r$. The exponent $\vartheta$ is simply dictated by scale invariance and given by 
	\[
	\vartheta=\frac{N}{p}-\frac{N}{r},
	\] 
see for example \cite[Theorem 3.8.1]{BraBook} or \cite[Theorem 12.83]{Leo}.
\par
In the next result, we recall that in the super-homogeneous regime $q>p$ we are interested in, all the Poincar\'e-Sobolev constants are equivalent.
\begin{prop}
\label{prop:positive}
Let $1<p<\infty$ and $q>p$ satisfying \eqref{pq}. Then there exist $C_2=C_2(N,p,q)>0$ and $C_3=C_3(N,p,q)>$ such that, for every open set $\Omega\subseteq \mathbb{R}^N$,  it holds
\[
C_2\,\Big(\lambda_p(\Omega)\Big)^{1-\frac{N}{p}+\frac{N}{q}}\le \lambda_{p,q}(\Omega)\le C_3\,\Big(\lambda_p(\Omega)\Big)^{1-\frac{N}{p}+\frac{N}{q}}.
\]
In particular, for every  $q$ satisfying \eqref{pq}
 we have
\[
\lambda_{p}(\Omega)>0 \qquad \Longleftrightarrow \qquad \lambda_{p,q}(\Omega)>0.
\]
\end{prop}

\begin{proof}
By using the  Gagliardo-Nirenberg inequality \eqref{GNS} with $r=q$, we have that 
\[
\| u\|_{L^q(\Omega)}^p\le (G_{N,p,q})^p\, \left(\int_\Omega |\nabla u|^p\,dx\right)^\vartheta\,\left(\int_\Omega |u|^p\,dx\right)^{1-\vartheta},\qquad \text{for every}\ u\in C^\infty_0(\Omega),
\]
 where the  exponent $\vartheta\in(0,1)$ is given by
\[
\vartheta=\frac{N}{p}-\frac{N}{q}.
\]
If we assume $\lambda_p(\Omega)>0$, by using the Poincar\'e inequality on the right-hand side above, we  get
\[
\|u\|_{L^q(\Omega)}^p\le (G_{N,p,q})^p\, \left(\int_\Omega |\nabla u|^p\,dx\right)\,\left(\frac{1}{\lambda_p(\Omega)}\right)^{1-\vartheta},\qquad \text{for every}\ u\in C^\infty_0(\Omega),
\] 
which in turn implies the  desired lower bound for $\lambda_{p,q}(\Omega)$.   
\par
The upper bound can be deduced by \cite[Theorem 15.4.1]{Maz} (for an alternative proof, see \cite[Corollary 5.9]{BraPriZag2} for $p>N$ and \cite[Theorem 6.1]{BozBra} for $1<p\le N$).
\end{proof}

\subsection{The Lane-Emden equation}

In the case $q<\infty$, whenever the infimum defining $\lambda_{p,q}(\Omega)$ is attained by some function $u\in W^{1,p}_0(\Omega)$, we have that $u$ satisfies the associated Euler-Lagrange equation 
\[
\int_{\Omega}\langle|\nabla u|^{p-2}\,\nabla u, \nabla \varphi\rangle\,dx= \lambda_{p,q}(\Omega)\int_\Omega |u|^{q-2}\,u\,\varphi\,dx,\qquad\text{for every}\ \varphi\in C^\infty_0(\Omega).
\]
In other words, $u$ is a weak solution of the so-called {\it Lane-Emden equation}
\[
-\Delta_p u=\lambda_{p,q}(\Omega)\,|u|^{q-2}\,u,\qquad \text{ in }\Omega,
\]
with homogeneous Dirichlet boundary condition.
In the case $p>N$ and  $q=\infty$, one can show that the Euler-Lagrange equation takes the following form 
\[
\int_{\Omega}\langle |\nabla u|^{p-2}\,\nabla u ,\nabla \varphi\rangle\,dx=  \lambda_{p,\infty}(\Omega)\,  |u(x_0)|^{p-2}\,u(x_0)\,\varphi(x_0),\quad  \text{for every}\ \varphi \in C^{\infty}_0(\Omega),
\]
where $x_0\in\Omega$ satisfies $|u(x_0)|=\|u\|_{L^{\infty}(\Omega)}$. We refer to Appendix \ref{sec: app} below for this fact. 
\par
In particular, for every $q$ as in \eqref{pq},  we have that every non-negative extremal $u$ of $\lambda_{p,q}(\Omega)$ is a weakly $p-$superharmonic function, that is,  $u$ satisfies
\[
\int_{\Omega}\langle |\nabla u|^{p-2}\,\nabla u ,\nabla \varphi\rangle\,dx\geq 0, \qquad \text{for every}\ \varphi\in C^\infty_0(\Omega)\ \text{such that}\ \varphi\geq 0.
\]
As a consequence, when $\Omega$ is a connected open set, every non-negative minimizer $u$ of $\lambda_{p,q}(\Omega)$ is actually positive in $\Omega$, thanks to the strong minimum principle.
 We also recall that any minimizer $u$ of $\lambda_{p,q}(\Omega)$, if it does exist,  has constant sign: 
for the case $q<\infty$ we refer to \cite[Lemma 2.1]{BraLin}, while for the case when  $p>N$ and $q=\infty$, we refer to Proposition \ref{prop:francesca}.
\par
Finally, we recall the following mild regularity result. The specific form of the a priori estimate will be important. 
\begin{lm}
\label{lm:Linfty}
Let $1<p<\infty$ and let $p\le q<\infty$ be an exponent satisfying \eqref{pq}.
Let $\Omega\subseteq\mathbb{R}^N$ be an open set such that $\lambda_p(\Omega)>0$ and
 let  $u\in W^{1,p}_0(\Omega)$ be a non-negative weak subsolution of
\[
-\Delta_p u=\lambda\,u^{q-1},\qquad \text{in}\ \Omega,
\]
for some $\lambda>0$. Then $u\in L^\infty(\Omega)$ and we have 
\[
\|u\|_{L^\infty(\Omega)}\le C_1\,\left(\lambda^\frac{N}{p\,q}\,\|u\|_{L^q(\Omega)}\right)^\frac{p\,q}{p\,q-(q-p)\,N},
\]
for some $C_1=C_1(N,p,q)>0$.
\end{lm}
\begin{proof}
The proof consists in reproducing verbatim that of \cite[Proposition 2.4]{BraLin}, thus it is omitted. 
For $q>p$, one just needs to preliminarily observe that $u\in L^q(\Omega)$, as well. 
Indeed, by Proposition \ref{prop:positive} and the assumption $\lambda_p(\Omega)>0$, the continuous embedding $W^{1,p}_0(\Omega)\hookrightarrow L^q(\Omega)$ holds.
\end{proof}

\begin{comment}
\begin{rem} \label{casep>N} In the sequel we will consider also  the particular case $N<p<\infty$ where the following equivalence holds for every open set $\Omega\subsetneq\mathbb{R}^N$ and for every $p\le q\le \infty$: 
 		\[ 
		\lambda_{p,q}(\Omega)>0 \qquad \Longleftrightarrow \qquad r_{\Omega}<+\infty.
		 \]
		 For a proof,  see \cite[Corollary 5.9]{BPZ}.
\end{rem}
\end{comment}

\section{Steiner symmetrization} 
\label{sec:3}

In this section we recall the main features of the classical Steiner symmetrization, referring the reader to \cite{CF}, \cite[Section 2.2]{He}, \cite{Kw} and \cite[Chapter 3]{LiLo} for more details.
\subsection{Sets}
For every direction $\mathbf{e}\in \mathbb{S}^{N-1}$, we denote by $\Pi_{\mathbf{e}}:\mathbb{R}^{N}\to\mathbb{R}^{N}$ the orthogonal projection onto the hyperplane $\langle\mathbf{e}\rangle^\bot:=\{x\in\mathbb{R}^N:\,  \langle x, \mathbf{e}\rangle=0\}$, that is 
\[
\Pi_{\mathbf{e}}(x)=x-\langle x, \mathbf{e}\rangle\,\mathbf{e},\qquad \mbox{ for every } x\in\mathbb{R}^N.
\]
For every direction $\mathbf{e}\in\mathbb{S}^{N-1}$, we also define
\[
\mathcal{R}_\mathbf{e}(x)=x-2\,\langle x,\mathbf{e}\rangle\,\mathbf{e},\qquad \text{for every}\ x\in\mathbb{R}^N,
\]
the {\it reflection with respect to the hyperplane} $\langle \mathbf{e}\rangle^\bot$. 
\par
Let $E\subseteq\mathbb{R}^N$ be a set and $\mathbf{e}\in\mathbb{S}^{N-1}$ a given direction.
For every $x\in\Pi_\mathbf{e}(E)$, we set
\[
E_{\mathbf{e}}(x)=\{y\in E\,:\, \exists\, t\in\mathbb{R}\ \text{such that}\ y=x+t\,\mathbf{e}\}.
\]
This is nothing but the intersection between $E$ and the line passing through $x$, with direction $\mathbf{e}$.
The {\it Steiner symmetrization of $E$ with respect to the direction $\mathbf{e}$} is the set $\mathcal{S}_{\mathbf{e}}(E)$ given by
\[
\mathcal{S}_{\mathbf{e}}(E)=\left\{x+t\,\mathbf{e}\, :\, x\in \Pi_\mathbf{e}(E),\, |t|<\frac{\mathcal{H}^{1}(E_{\mathbf{e}}(x))}{2} \right\}.
\]
We recall that if the set $E\subseteq\mathbb{R}^N$ is Lebesgue measurable, then $\mathcal{S}_{\mathbf{e}}(E)$ has the same property and there holds  
\[
|E|=|\mathcal{S}_{\mathbf{e}}(E)|,
\]
for a proof see \cite[Theorem 2.3]{EvGa}. Moreover, if $E$ is an open set, then $\mathcal{S}_{\mathbf{e}}(E)$ is open, as well. In the following, we say that:
\begin{itemize}
\item[$\bullet$] $E$ is {\it symmetric with respect to the hyperplane} $\langle\mathbf{e}\rangle^{\perp}$ if 
\[
\mathcal{R}_\mathbf{e}(E)=E.
\]
In particular, for every $x\in\Pi_\mathbf{e}(E)$, if $x+t\,\mathbf{e}\in \Omega$ for some $t\in\mathbb{R}$, then we have $x-t\,\mathbf{e}\in\Omega$, as well;
\vskip.2cm
\item[$\bullet$] $E$ is {\it convex in the direction} $\textbf{e}$ if 
\[
x+t_1\,\textbf{e}\in \Omega,\ x+t_2\,\textbf{e}\in E\Longrightarrow x+(s\,t_1+(1-s)\,t_2)\,\textbf{e}\in E,\quad \begin{array}{c}\text{for every}\ x\in \Pi_\textbf{e}(E),\\ t_1,t_2\in\mathbb{R}, s\in(0,1).\end{array}
\]
\end{itemize} 
Then, the following characterization holds: 
\[
\mathcal{S}_{\mathbf{e}}(E)=E \qquad  \Longleftrightarrow \qquad  E\ \text{is symmetric with respect to}\ \langle\textbf{e}\rangle^{\perp}\ \text{and convex in the direction}\ \textbf{e}.
\]
In the sequel, we will consider open sets with an orthogonal system of symmetries, according to the following definition.

\begin{definition}[The class $\mathfrak{S}^N$]
\label{S}
We say that an open set $\Omega \subseteq \mathbb{R}^N$ is {\it Steiner symmetric} if 
\vskip.2cm
\begin{itemize}
\item[(S1)] $\Omega$ is  symmetric with respect to $\langle\mathbf{e}_j\rangle^{\perp}$, for every  $j \in\{1,\cdots,N\}$;
\vskip.2cm
\item[(S2)] $\Omega$ is convex in the direction $\mathbf{e}_j$, for every $j\in\{1,\cdots,N\}$.
\end{itemize}  
\vskip.2cm
We will indicate by $\mathfrak{S}^N$ the collection of all Steiner symmetric open subsets of $\mathbb{R}^N$.
\end{definition}
\begin{rem}
\label{rem:S1S2}
In particular, every $\Omega\in\mathfrak{S}^N$ has the following property: if $\Sigma\subseteq\Omega$ is a convex set, then we have 
\[
\left(\Sigma\cup \mathcal{R}_{\mathbf{e}_j}(\Sigma)\right)^{\rm ch}\subseteq\Omega,\qquad \text{for every}\, j\in\{1,\dots,N\},
\]
that is $\Omega$ contains the {\it convex hull} of $\Sigma\cup \mathcal{R}_{\mathbf{e}_j}(\Sigma)$, as well.
\end{rem}
The next result contains some basic geometric properties of Steiner symmetric sets.
\begin{lm}\label{lemmainradSN}
Let $\Omega\in\mathfrak{S}^N$ with $\Omega\neq\mathbb R^N$. Then we have 
\[
r_\Omega<+\infty\qquad \text{and}\qquad B_{r_\Omega}\subseteq\Omega.
\]
Moreover, the following measure density condition holds
\begin{equation}
\label{mde}
\inf\left\{\frac{|B_r(x)\setminus \Omega|}{|B_r(x)|}\, :\, x\in\partial\Omega,\ r>0\right\}\ge \left(\frac{1}{2}\right)^N.
\end{equation}
\end{lm}
\begin{proof}
Let   $\{r_n\}_{n\in\mathbb{N}}\subseteq (0,+\infty)$ be an increasing sequence converging to $r_\Omega$. Accordingly, there exists a sequence of points $\{x_n\}_{n\in\mathbb{N}}\subseteq \Omega$ such that
\[
B_{r_n}(x_n)\subseteq\Omega,\qquad \text{for every}\ n\in\mathbb{N}.
\]
By appealing to Remark \ref{rem:S1S2} with $\Sigma=B_{r_n}(x_n)$ and $j=1$, we get
\[
\Big(B_{r_n}(x_n)\cup \mathcal{R}_{\mathbf{e}_1}(B_{r_n}(x_{n})\Big)^{\rm ch}\subseteq \Omega,\qquad \text{for every}\ n\in\mathbb{N}.
\]
By observing that 
\[
\Big(B_{r_n}(x_n)\cup \mathcal{R}_{\mathbf{e}_1}(B_{r_n}(x_{n})\Big)^{\rm ch}\supseteq B_{r_n}(\Pi_{\mathbf{e}_1}(x_n)),
\]
we get
\[
B_{r_n}(\Pi_{\mathbf{e}_1}(x_n))\subseteq \Omega,\qquad \text{for every}\ n\in\mathbb{N}.
\]
We can repeat the above argument with the ball $B_{r_n}(\Pi_{\mathbf{e}_1}(x_n))$ and $\mathbf{e}_2$, in place of $B_{r_n}(x_n)$ and $\mathbf{e}_1$, by also getting that 
\[
B_{r_n}(\Pi_{\mathbf{e}_2}\circ \Pi_{\mathbf{e}_1}(x_n))\subseteq \Omega,\qquad \text{for every}\ n\in\mathbb{N}.
\]
By applying this scheme recursively with respect to every direction $\mathbf{e}_i$ for $i\in\{1,\dots, N\}$, we finally obtain that 
\[
B_{r_n}=B_{r_n}(\Pi_{\mathbf{e}_N}\circ\dots\circ \Pi_{\mathbf{e}_1}(x_n))\subseteq \Omega,\qquad \text{for every}\ n\in\mathbb{N}.
\]
By recalling that $\{r_n\}_{n\in\mathbb{N}}$ converges to $r_\Omega$ and that $\Omega\not=\mathbb{R}^N$, the previous result guarantees 
that $r_{\Omega}<\infty$ and $B_{r_\Omega}\subseteq \Omega$, as claimed.
\vskip.2cm\noindent
We now come to the proof of the measure density estimate \eqref{mde}. Let $x=(x_1,\cdots,x_N)\in \partial\Omega$, thanks to the symmetries of $\Omega$ we have
\[
|B_r(\overline{x})\setminus \Omega|=|B_r(x)\setminus\Omega|,
\]
where $\overline{x}=(|x_1|,\dots,|x_N|)\in\partial\Omega$. Thus, we can suppose without loss of generality that $x_i\ge 0$, for every $i\in\{1,\dots,N\}$.
We claim that 
\begin{equation}\label{ottanti}
\mathcal{C}_x:=\prod_{i=1}^N[x_i,+\infty) \subseteq \mathbb{R}^N\setminus\Omega.
\end{equation}
Observe that $\mathcal{C}_x$ is nothing but a {\it hyperoctant} in $\mathbb{R}^N$, with vertex at $x$.
The argument is similar to the one given in the first part of the proof.
By contradiction, we suppose that there exists \[
z=(z_1,\cdots,z_N)\in \Omega\cap\mathcal{C}_x.
\] 
By Remark \ref{rem:S1S2}, we get that $\Omega$ must contain the segment joining $z$ and $\mathcal{R}_{\mathbf{e}_1}(z)$. In other words, we have
\[
(t,z_2,\dots,z_N)\in \Omega,\qquad \text{for every}\ -z_1\le t\le z_1.
\]
In particular, since $z\in \mathcal{C}_x$, we have $z_1\ge x_1$ and thus with the choice $t=x_1$, we get
\[
(x_1,z_2,\dots,z_N)\in\Omega.
\] 
We can now iterate and repeat this argument in the other coordinate directions. This would give us $x=(x_1,x_2,\dots,x_N)\in\Omega$, which is a contradiction. 
Therefore \eqref{ottanti} holds true. As a consequence, for every $r>0$ we have
\[
|B_r(x)\setminus \Omega|\ge \left|B_r(x)\cap \mathcal{C}_x\right|= 2^{-N}\,|B_r(x)|,
\]
that proves the desired estimate. 
\end{proof}

Thanks to the geometric properties encoded in Lemma \ref{lemmainradSN}, we deduce the following expedient result. This is interesting in itself.
\begin{prop}
\label{prop:positivity}
Let $1<p<\infty$ and $q\ge p$ satisfying \eqref{pq}. Let $\Omega\in\mathfrak{S}^N$ be such that $\Omega\neq\mathbb R^N$. Then we have $\lambda_{p,q}(\Omega)>0$. More precisely, the following lower bound holds
\[
c\,\left(\frac{1}{r_\Omega}\right)^{p-N+N\,\frac{p}{q}}\le\lambda_{p,q}(\Omega),
\] 
for some $c=c(N,p,q)>0$.
\end{prop}
\begin{proof}
By applying Lemma \ref{lemmainradSN}, we know that $r_\Omega<+\infty$. For $p>N$, it is then sufficient to appeal to \cite[Theorem 15.4.1]{Maz} (see also \cite[Corollary 5.9]{BraPriZag2}). 
For $1<p\le N$, we can simply evoke  \cite[Corollary 5.6]{BozBra2}, thanks to the measure density condition \eqref{mde}. 
\end{proof}
In Section \ref{sec:7} we will need the following further definitions, aimed at describing the ``local geometry at infinity'' of a Steiner symmetric set. Let $\Omega\in\mathfrak{S}^N$, for every $i\in\{1,\dots,N\}$ we define
\[
\Omega_{i,t}=\{x\in\Omega\,:\,\langle x,\mathbf{e}_i\rangle =t\},\qquad\text{for}\ t\in\mathbb{R}.
\]
These are the sections of $\Omega$ which are orthogonal to $\mathbf{e}_i$. Observe that these are open subsets of the hyperplane $\{x\,:\,\langle x,\mathbf{e}_i\rangle =t\}\simeq\mathbb{R}^{N-1}$, with respect to the induced topology.
Accordingly, for every $i\in\{1,\dots,N\}$ we define
\begin{equation}
\label{inradius_sect}\
r_{\Omega_{i,t}}=\sup\Big\{r>0\,:\, \exists\ \text{a ball}\ B_r(x_0)\subseteq\mathbb{R}^N\ \text{with }x_0\in\Omega_{i,t}\  \text{s.\,t.}\ (B_r(x_0))_{i,t}\subseteq\Omega_{i,t}\Big\}
,\quad \text{for}\ t\in\mathbb{R}.
\end{equation}
The quantity $r_{\Omega_{i,t}}$ is nothing but the inradius of the section $\Omega_{i,t}$. 
Thanks to Steiner symmetry, it is not difficult to see that it enjoys the following properties.
\begin{lm}
\label{lm:luca}
Let $\Omega\in\mathfrak{S}^N$ be an open set such that $\Omega\not=\mathbb{R}^N$. Then for every $i\in\{1,\dots,N\}$
\begin{equation}
\label{eq:rit}
r_{\Omega_{i,t}}=r_{\Omega_{i,-t}},\ \text{for every}\ t\in\mathbb{R} \qquad\text{and}\qquad r_{\Omega_{i,t}}\le r_{\Omega_{i,s}},\ \text{for every}\ 0\le s<t.
\end{equation}
Moreover, 
there exists a constant $\tau_\Omega>0$ such that
\begin{equation}
\label{eq:tauomega}
\max_{i\in\{1,\dots, N\}}r_{\Omega_{i,t}}<+\infty,\quad \text{for every}\ t\ge \tau_\Omega.
\end{equation}
\end{lm}
\begin{proof}
The properties \eqref{eq:rit} follow quite easily from (S1) and (S2).
\par
In order to prove \eqref{eq:tauomega}, we observe that if $\Omega_{i,t}=\{x\in\mathbb{R}^N\,:\,\langle x,\mathbf{e}_i\rangle =t\}$ for every $t\in\mathbb{R}$, then $\Omega$ would coincide with the whole space $\mathbb{R}^N$. Thus, there must exist $t_i\in\mathbb{R}$ such that $\Omega_{i,t}\not=\{x\in\mathbb{R}^N\, :\,\langle x,\mathbf{e}_i\rangle =t\}$. With a slight abuse of notation, we observe that $\Omega_{i,t}\in\mathfrak{S}^{N-1}$.
By applying Lemma \ref{lemmainradSN} in dimension $N-1$, we get that $\Omega_{i,t}$ has a finite inradius and a maximal ball is centered at the point $t_i\,\mathbf{e}_i$.
In light of the first property in \eqref{eq:rit}, we can suppose that $t_i\ge 0$, while by the second property in \eqref{eq:rit} we get that 
\[
r_{\Omega_{i,t}}\le r_{\Omega_{i,t_i}},\ \text{for every}\ 0\le t_i<t.
\]
By repeating this argument for every direction, we get the conclusion.
\end{proof}

\subsection{Functions}
Let $1\leq p<\infty$, let $E\subseteq \mathbb R^N$ a measurable set and let  $u\in L^p(E)$ be a non-negative
 function. We extend $u$ by setting $u=0$ in $\mathbb{R}^N\setminus E$. 
We can define the {\it Steiner symmetrization of $u$ with respect to the direction} $\mathbf{e}\in\mathbb{S}^{N-1}$ as  the non-negative function $\mathcal{S}_{\mathbf{e}}(u)$ given  by 
\[
\mathcal{S}_{\mathbf{e}}(u)(y)=\int_0^{+\infty} 1_{\mathcal{S}_{\mathbf{e}}(E^+_u(t))} (y)\, dt,
\]
 where, for every $t\geq 0$, we indicated
 \[
 E^+_u(t)=\{x\in E\, :\, u(x)>t\},
 \] 
 and 
 $1_A$ stands for  the characteristic function  of a set $A$.
  We note that: 
\begin{itemize}
\item $\mathcal{S}_{\mathbf{e}}(u)=0$ almost everywhere in $\mathbb{R}^N\setminus \mathcal{S}_{\mathbf{e}}(E)$;
\vskip.2cm
\item by exploiting the definition above and the layer cake representation, it is easy to show that 
 for almost every $t\in \mathbb R$, it holds 
\begin{equation}\label{equimeas}
\mathcal{S}_{\mathbf{e}}(E)^+_{\mathcal{S}_\mathbf{e}(u)}(t)=\{y\in \mathcal{S}_{\mathbf{e}}(E)\, :\, \mathcal{S}_\mathbf{e}(u)(y)>t\}=\mathcal{S}_{\mathbf{e}}\left(E_u^+(t)\right).
\end{equation}
Hence, $\mathcal{S}_{\mathbf{e}}(u)$ is a non-negative measurable function on $\mathbb{R}^N$ which is equimeasurable with $u$. 
In particular, this implies that  $\mathcal{S}_{\mathbf{e}}(u)\in L^p(\mathbb{R}^N)$ with   
\begin{equation}\label{PZ1}
\| \mathcal{S}_{\mathbf{e}}(u)\|_{L^p(\mathbb{R}^N)}=\|u \|_{L^p(\mathbb{R}^N)};
\end{equation}
\item   if $v$ is a Lipschitz function compactly supported in an open set  $\Omega$, then $\mathcal{S}_{\mathbf{e}}(v)$ is a Lipschitz function with compact support in $\mathcal{S}_{\mathbf{e}}(\Omega)$, as well. 
\end{itemize} 
Finally, we recall that,  if $u\in W^{1,p}(\mathbb{R}^N)$ with $1< p<\infty$, then  the celebrated P\`olya-Sz\H{e}go inequality holds true 
\begin{equation}
\label{PZ2}
\int_{\mathbb{R}^N}|\nabla \mathcal{S}_{\mathbf{e}}(u)|^p\,dx\leq\int_{\mathbb{R}^N}|\nabla u|^p\,dx.
\end{equation}
For a proof, we refer, for example, to  \cite[Proposition 3.1]{Bu}. 
\par
The following statement is well-known, but we have not been able to detect a proof in the literature. Due to its pivotal importance, we eventually decided to include it.
\begin{lm}
\label{lm:zeropres}
Let $1<p<\infty$ and let $\Omega\subseteq\mathbb{R}^N$ be an open set. For every $\mathbf{e}\in \mathbb{S}^{N-1}$ and every $u\in W^{1,p}_0(\Omega)$, we have $\mathcal{S}_{\mathbf{e}}(u)\in W^{1,p}_0(\mathcal{S}_\mathbf{e}(\Omega))$.
\end{lm}
\begin{proof}
By definition, there exists a sequence $\{u_n\}_{n\in\mathbb{N}}\subseteq C_{0}^\infty(\Omega)$ 
\[
\lim_{n\to\infty} \|u_n-u\|_{W^{1,p}(\Omega)}=0.
\] 
We extend all the functions to be $0$ outside $\Omega$. 
The {\it contractivity property} of Steiner symmetrization (see \cite[Propositions 1-2]{CT}) gives   
\[
\|\mathcal{S}_{\mathbf{e}}(u_n)-\mathcal{S}_{\mathbf{e}}(u)\|_{L^p(\mathbb{R}^N)}\le \|u_n-u\|_{L^p(\mathbb{R}^N)}, \qquad \text{for every}\ n\in \mathbb N.
\]
Hence $\mathcal{S}_{\mathbf{e}}(u_n)$ converges to $\mathcal{S}_\mathbf{e}(u)$ strongly in $L^p(\mathbb{R}^N)$. Moreover, \eqref{PZ2} gives that $\mathcal{S}_\mathbf{e}(u)\in W^{1,p}(\mathbb{R}^N)$. Still by \eqref{PZ2} and the boundedness of $\{\nabla u_n\}_{n\in \mathbb N}$, we also get that $\{\nabla \mathcal{S}_{\mathbf{e}}(u_n)\}_{n\in \mathbb N}$ is bounded in $L^p(\mathbb{R}^N)$.
Thus, we easily obtain that $\{\nabla \mathcal{S}_\mathbf{e}(u_n)\}_{n\in \mathbb N}$ weakly converges in $L^p(\mathbb{R}^N)$ to $\nabla \mathcal{S}_\mathbf{e}(u)$. Indeed, by the definition of weak gradient and the strong convergence of $\{\mathcal{S}_\mathbf{e}(u_n)\}_{n\in\mathbb{N}}$ in $L^p(\mathbb{R}^N)$, we have for every vector field $\phi\in C^\infty_0(\mathbb{R}^N;\mathbb{R}^N)$ 
\[
\begin{split}
\int_{\mathbb{R}^N} \langle\nabla \mathcal{S}_\mathbf{e}(u),\phi\rangle\,dx=-\int_{\mathbb{R}^N} \mathcal{S}_\mathbf{e}(u)\,\mathrm{div}\phi\,dx&=-\lim_{n\to\infty}\int_{\mathbb{R}^N} \mathcal{S}_\mathbf{e}(u_n)\,\mathrm{div}\phi\,dx\\
&=\lim_{n\to\infty}\int_{\mathbb{R}^N} \langle \nabla \mathcal{S}_\mathbf{e}(u_n),\phi\rangle\,dx.
\end{split}
\]
Note that each function $\mathcal{S}_\mathbf{e}(u_n)$ is a Lipschitz function with compact support in $\mathcal{S}_{\mathbf{e}}(\Omega)$, thus we have in particular that  $\{\mathcal{S}_{\mathbf{e}}(u_n)\}_{n\in\mathbb{N}}\subseteq W_0^{1,p}(\mathcal{S}_{\mathbf{e}}(\Omega)) $. Hence, we can conclude that the function $\mathcal{S}_{\mathbf{e}}(u)$ is the weak limit of a sequence belonging to $W_0^{1,p}(\mathcal{S}_{\mathbf{e}}(\Omega)) $. The latter being weakly closed, we get the conclusion.
\end{proof}
In the next section we will work with functions $u\in L^p(\Omega)$ satisfying the condition
\[
\int_\Omega |x|\,|u|^p\,dx<+\infty.
\] 
We will show that the previous integral does not increase under the Steiner symmetrization. At first, we prove the following preliminary lemma. We also refer to \cite[Lemma 2.11]{Vol} for a stronger result in the case $\alpha=2$.
\begin{lm}\label{lm:simV}
Let $A\subseteq\mathbb{R}^N$ be a measurable set.   For every $\alpha>0$ and every $\mathbf{e}\in \mathbb{S}^{N-1}$, we have 
\[
\int_A |x|^\alpha\,dx\ge \int_{\mathcal{S}_{\mathbf{e}}(A)}|x|^\alpha\,dx.
\]
\end{lm}
\begin{proof}
Without loss of generality, we can assume  that  
\[
\int_A |x|^\alpha\,dx<+\infty.
\]
Hence, by using the Markov-Chebyshev inequality, we get that 
\[
|\{x\in A:\ |x|>1\}|\le \int_A |x|^{\alpha}\,dx<+\infty,
\] 
which implies
$|A|<+\infty.$
Notice that  for every $\tau>0$ we have that 
\[
\{x\in A\,:\, |x|^\alpha>\tau \}=A\setminus \overline{B_{\tau^{1/\alpha}}}.
\]
Moreover, since
\[
\mathcal{S}_{\mathbf{e}}(A\cap B_{\tau^{1/\alpha}})\subseteq \mathcal{S}_{\mathbf{e}}(A)\cap \mathcal{S}_{\mathbf{e}}(B_{\tau^{1/\alpha}})=\mathcal{S}_{\mathbf{e}}(A)\cap B_{\tau^{1/\alpha}},
\]
it holds
\[
|A\cap B_{\tau^{1/\alpha}}|=|\mathcal{S}_{\mathbf{e}}(A\cap B_{\tau^{1/\alpha}})|\le |\mathcal{S}_{\mathbf{e}}(A)\cap B_{\tau^{1/\alpha}}|.
\]
To obtain the claimed result we can combine these facts and Cavalieri Formula as follows:
\[
\begin{split}
\int_A |x|^\alpha\,dx =\int_0^{+\infty} |\{x\in A: |x|^\alpha>\tau\}|\,d\tau&=\int_0^\infty |A\setminus \overline{B_{\tau^{1/\alpha}}}|\,d\tau\\
&=\int_0^{+\infty} \Big[|A|-|A\cap B_{\tau^{1/\alpha}}|\Big]\,d\tau\\
&\ge \int_0^{+\infty} \Big[|\mathcal{S}_{\mathbf{e}}(A)|-|\mathcal{S}_{\mathbf{e}}(A)\cap B_{\tau^{1/\alpha}}|\Big]\,d\tau\\
&=\int_0^{+\infty} |\{x\in \mathcal{S}_{\mathbf{e}}(A)\,:\, |x|^\alpha>\tau\}|\,d\tau=\int_{\mathcal{S}_{\mathbf{e}}(A)} |x|^\alpha\,dx.
\end{split}
\]
This concludes the proof. \end{proof}

\begin{prop}\label{pr:simV}
Let $\Omega\subseteq\mathbb{R}^N$ be an open set and $1\le p<\infty$. For every $\alpha>0$, $\mathbf{e}\in \mathbb{S}^{N-1}$ and for every non-negative  function $u\in L^p(\Omega)$ it holds
\begin{equation}
\label{scende}
\int_\Omega |x|^\alpha\,u^p\,dx\ge
\int_{\mathcal{S}_{\mathbf{e}}(\Omega)} |x|^\alpha\,(\mathcal{S}_{\mathbf{e}}(u))^p\,dx. 
\end{equation}
\end{prop}
\begin{proof}
%\st{ First, we notice that, for every $t>0$, the set $\{x\in\Omega:\ u(x)>t\}$ has finite measure, since it holds}
% \[
%|\{x\in\Omega:\ u(x)>t\}|\le t^{-p}\int_\Omega u^p\,dx.
%\]
Without loss of generality, we can assume  that  
\[
\int_\Omega u^p\,|x|^\alpha\,dx<+\infty.
\]
By applying Lemma \ref{lm:simV} together with Cavalieri formula for the measure $d\mu=|x|^\alpha\,dx$, we obtain
\[
\begin{split}
\int_\Omega |x|^\alpha\,u^p\,dx &=p\int_0^{+\infty} t^{p-1} \left(\int_{\Omega^+_u(t)} |x|^\alpha \,dx\right) \,dt\ge p\int_0^{+\infty} t^{p-1} \left(\int_{\mathcal{S}_{\mathbf{e}}(\Omega^+_u(t))} |x|^\alpha \,dx\right) \,dt.
\end{split}
\]
Thanks to \eqref{equimeas}, the previous inequality implies that 
\[
\begin{split}
\int_\Omega |x|^\alpha\,u^p\,dx%\ge p\int_0^\infty t^{p-1} \left(\int_{\{x\in\Omega:\ u(x)>t\}^{\star_{\mathbf{e}}}} |x|^\alpha \,dx\right) \,dt\\
\geq \,  p\int_0^{+\infty} t^{p-1} \left(\int_{\mathcal{S}_{\mathbf{e}}(\Omega)^+_{\mathcal{S}_\mathbf{e}(u)}(t)} |x|^\alpha \,dx\right) \,dt=\int_{\mathcal{S}_\mathbf{e}(\Omega)}|x|^\alpha\,(\mathcal{S}_{\mathbf{e}}(u))^p\,dx,
\end{split}
\]
where the last identity again follows from Cavalieri formula. 
This concludes the proof. 
\end{proof}

We conclude this section with the following definition.
\begin{definition}[Steiner symmetric function]
 For $\Omega\in\mathfrak{S}^N$, we say that a non-negative function $u\in L^p(\Omega)$ is {\it Steiner symmetric} if 
 \[
 u\equiv \mathcal{S}_{\mathbf{e}_j}(u),\quad \text{in}\ \Omega,\qquad \text{for every}\ j\in\{1,\dots, N\}.
 \]
\end{definition}

We notice that if $u\in L^p(\Omega)$ is Steiner symmetric, then it is non-increasing along every coordinate direction. In particular, if $u\in L^p(\Omega)\cap L^\infty(\Omega)$, we have 
\[
\|u\|_{L^{\infty}(\Omega)}=\|u\|_{L^{\infty}(B_r)},\qquad \text{for every}\ r>0.
\]
Finally, we note that  if $u\in L^p(\Omega)$ is non-negative, then the function $\mathcal{S}_{\mathbf{e}_N}\circ \mathcal{S}_{\mathbf{e}_{N-1}} \circ \cdots  \circ \mathcal{S}_{\mathbf{e}_1} (u)$ is of course  Steiner symmetric.

\section{A minimizing sequence}
\label{sec:4}

We will use the following notation
\[
W^{1,p}_0(\Omega;|x|):=\left\{u\in W^{1,p}_0(\Omega)\, :\, \int_\Omega |x|\,|u|^p\,dx<+\infty\right\}.
\] 
We start with a result holding for a class of open sets under minimal assumptions, not necessarily Steiner symmetric. This shows that $\lambda_{p,q}(\Omega)$ can be approximated by a problem containing a ``vanishing confinement'' term, thus giving a suitable gain of compactness. For completeness, we also include the case $q=p$, even if we will not use it in the sequel.
\begin{lm}
\label{lm:approssimazione}
Let $1<p<\infty$ and $q\ge p$ satisfying \eqref{pq}. Let $\Omega\subseteq\mathbb{R}^N$ be an open set such that $\lambda_p(\Omega)>0$.
For every $n\in\mathbb{N}$, we define 
\[
\lambda_{p,q}(\Omega;V_n):=\inf_{u\in W^{1,p}_0(\Omega;|x|)}\Big\{\mathcal{G}_{p,n}(u)\,:\, \|u\|_{L^q(\Omega)}=1\Big\},
\]
where 
\[
\mathcal{G}_{p,n}(u)=\int_{\Omega}|\nabla u|^{p}\,dx +\int_{\Omega}V_n\,|u|^{p}\,dx\qquad \text{and}\qquad V_n(x)=\frac{|x|}{n+1}.
\]
Then, we have
\begin{equation}
\label{eq:limiteutile0}
\lim_{n\to\infty}\lambda_{p,q}(\Omega;V_n)=\lambda_{p,q}(\Omega).
\end{equation}
Moreover, the value $\lambda_{p,q}(\Omega;V_n)$ is attained by some function $u_n\in W^{1,p}_0(\Omega;|x|)$ with $\|u_n\|_{L^q(\Omega)}=1$ and we have
\begin{equation}
\label{eq:limiteutile}
\lim_{n\to\infty}\int_{\Omega}|\nabla u_n|^p\, dx=\lambda_{p,q}(\Omega).
\end{equation}
\end{lm}
\begin{proof}
Since we have $V_{n+1}\le V_n$, we get
\[
\mathcal{G}_{p,n}(u)\ge \mathcal{G}_{p,n+1}(u),\qquad \text{for every}\ u\in W^{1,p}_0(\Omega;|x|).
\]
Thus, it is not difficult to see that $n\mapsto \lambda_{p,q}(\Omega;V_n)$ is monotone decreasing and
\[
\lim_{n\to\infty}\lambda_{p,q}(\Omega;V_n)\ge \lambda_{p,q}(\Omega).
\]
 In order to prove that this limit coincides with $\lambda_{p,q}(\Omega)$, for every $\varepsilon>0$ and every $u\in C^\infty_0(\Omega)$ such that $\|u\|_{L^q(\Omega)}=1$,
 we can choose $n_\varepsilon\in \mathbb{N}$ such that
\[
\frac{1}{n_\varepsilon+1}\,\int_{\Omega}|x|\,|u|^{p}\,dx \le \varepsilon.
\]
Accordingly, we get for every $n\ge n_\varepsilon$
\[
\lambda_{p,q}(\Omega;V_n)\le \lambda_{p,q}(\Omega;V_{n_\varepsilon})\le \int_\Omega |\nabla u|^p\,dx+\frac{1}{n_\varepsilon+1}\,\int_{\Omega}|x|\,|u|^{p}\,dx \le \int_\Omega |\nabla u|^p\,dx+\varepsilon.
\]
By arbitrariness of both $\varepsilon>0$ and $u$, we conclude that
\[
\lim_{n\to\infty}\lambda_{p,q}(\Omega;V_n)\le \lambda_{p,q}(\Omega),
\]
as well.
\par
To show that $\lambda_{p,q}(\Omega;V_n)$ is attained, we can assume that $W^{1,p}_0(\Omega)\hookrightarrow L^q(\Omega)$ is not compact, otherwise the proof would be straightforward.
We take $\{u_{m,n}\}_{m\in\mathbb{N}}$ to be a minimizing sequence  such that 
\[
 \mathcal{G}_{p,n}(u_{m,n})\le \lambda_{p,q}(\Omega; V_n)+\frac{1}{m+1} \qquad \|u_{m,n}\|_{L^q(\Omega)}=1,\qquad \text{for every}\ m\in \mathbb{N}.
\]
 Since $\lambda_p(\Omega)>0$,  there exists a subsequence (not relabeled) $\{u_{m,n}\}_{m\in\mathbb{N}}$ weakly  converging to some  function $u_n\in W^{1,p}_0(\Omega)$. Moreover, for every $R>0$ and for every $m,n\in\mathbb{N}$, we have
\[
\begin{split}
\int_{\Omega\setminus B_{R}(0)}|u_{m,n}|^{p}\,dx&\le \frac{1}{R}\,\int_{\Omega\setminus B_{R}(0)}|x|\,|u_{m,n}|^{p}\,dx\\
&\le \frac{n+1}{R}\,\mathcal{G}_{p,n}(u_{m,n})\le \frac{n+1}{R}\,\left( \lambda_{p,q}(\Omega;V_n)+\frac{1}{m+1}\right).
\end{split}
\]
Thus, we get
\[
\lim_{R\to+\infty}\int_{\Omega\setminus B_{R}(0)}|u_{m,n}|^{p}\,dx=0,
\]
uniformly in $m$.  Hence, the Riesz-Fr\'echet-Kolmogorov theorem ensures that $\{u_{m,n}\}_{m\in\mathbb{N}}$ converges to $u_n$ strongly in $L^{p}(\Omega)$ as $m$ goes to $\infty$. 
In particular, by observing that the potential $V_n$ is locally bounded, for every $k\in\mathbb{N}$ we have
\[
\begin{split}
\liminf_{m\to\infty}\int_{\Omega}V_n\,|u_{m,n}|^{p}\,dx&\ge \liminf_{m\to\infty}\int_{\Omega\cap B_k}V_n\,|u_{m,n}|^{p}\,dx=\int_{\Omega\cap B_k}V_n\,|u_{n}|^{p}\,dx.
\end{split}
\] 
By arbitrariness of $k$, the Monotone Convergence Theorem gives that
\[
\liminf_{m\to\infty}\int_{\Omega}V_n\,|u_{m,n}|^{p}\,dx\ge \int_{\Omega}V_n\,|u_{n}|^{p}\,dx.
\]
By combining this fact and the lower semicontinuity of $L^p$ norms with respect to the weak convergence, we get that
\begin{equation}\label{eq:liminf}
 \lambda_{p,q}(\Omega;V_n)=\lim_{m\to\infty}\mathcal{G}_{p,n}(u_{m,n})\ge  \int_{\Omega}|\nabla u_n|^{p}\,dx +\int_{\Omega}V_n\,|u_n|^{p}\,dx,
\end{equation}
 which  implies that $u_n\in W^{1,p}_0(\Omega;|x|)$ and 
 \[ 
 \mathcal{G}_{p,n}(u_n)\leq  \lambda_{p,q}(\Omega;V_n).
 \]
In the case $q=p$, this is enough to conclude that $u_n$ is a minimizer. For $q>p$, we still need to verify that $u_n$ has unit $L^q(\Omega)$ norm. An application of the interpolation inequality \eqref{GNS} gives 
\[
\|u_{m,n}-u_n\|_{L^q(\Omega)}\le G_{N,p,q}\, \|\nabla u_{m,n}-\nabla u_n\|_{L^p(\Omega)}^\theta\,\|u_{m,n}-u_n\|_{L^p(\Omega)}^{1-\theta}.
\]
%with  
%	\[
%	\theta=\displaystyle \frac{N\,(q-p)}{p\,q}.
%	\] 
Thus, the sequence $\{u_{m,n}\}_{m\in\mathbb{N}}$  converges to $u_{n}$ also in $L^{q}(\Omega)$. In particular $\|u_n\|_{L^q(\Omega)}=1$ and $u_n$ is feasible for the minimization problem which defines $\lambda_{p,q}(\Omega;V_n)$. In light of \eqref{eq:liminf}, we can finally infer that $u_n$ is a minimizer.
\par
In order to prove \eqref{eq:limiteutile}, we can simply observe that
\[
\lambda_{p,q}(\Omega)\le\int_{\Omega} |\nabla u_n|^p\,dx\le \int_{\Omega}|\nabla u_n|^{p}\,dx +\int_{\Omega}V_n\,|u_n|^{p}\,dx=\lambda_{p,q}(\Omega;V_n).
\]
By taking the limit as $n$ goes to $\infty$ and using \eqref{eq:limiteutile0}, we conclude.
\end{proof}
\begin{rem}
\label{rem:stimaunife}
 With the previous notation, let us suppose that $B_r\subseteq\Omega$. If we take $v\in W^{1,p}_0(B_r)$ to be an extremal function for $\lambda_{p,q}(B_r)$, by extending it equal to zero outside $B_r$, we get for $p<q<\infty$ 
\[
\begin{split}
\lambda_{p,q}(\Omega;V_n)&\le \int_{B_r} |\nabla v|^p\,dx+\int_{B_r} \frac{|x|}{n+1}\,|v|^p\,dx\\
&\le \lambda_{p,q}(B_r)+\frac{1}{n+1}\, \left(\int_{B_r}|x|^\frac{q}{q-p}\,dx\right)^\frac{q-p}{q}\\
&=r^{N-p-\frac{p}{q}\,N}\,\lambda_{p,q}(B_1)+\frac{1}{n+1}\, \left(\dfrac{N\,\omega_N}{N+\frac{q}{q-p}}\right)^\frac{q-p}{q}\,r^{N+1-\frac{p}{q}\,N}.
\end{split}
\]
%\,\left(\int_{B_r(x_0)} |x|^\frac{q}{q-p}\,dx\right)^\frac{q-p}{q}.
%{\color{magenta} where $\rho(x)=|x|$ and the space $L^{\frac{q-p}{q}} (B_{r}(x_0))$ stands for  $L^{1}(B_{r}(x_0))$ when $q=\infty$}.   
In the second estimate we used H\"older's inequality and the fact that $v$ has unit $L^q(B_r)$ norm. A similar estimate holds also in the case $q=\infty$, by simply using that $|v|\le\|v\|_{L^\infty(B_r)}=1$.
\end{rem}
We then specialize the previous problem to the case of a Steiner symmetric set and give some mild regularity properties of the minimizers found in the previous result.
\begin{prop}
\label{prop:fine}
Let $1<p<\infty$ and let $p<q<\infty$ satisfy \eqref{pq}. Let $\Omega\in\mathfrak{S}^N$ be an open set, with $\Omega\not=\mathbb{R}^N$. With the notation of Lemma \ref{lm:approssimazione}, the sequence $\{u_n\}_{n\in\mathbb{N}}$ can be chosen so that it has the following properties:
\begin{itemize}
\item[{\it (i)}] $u_n\ge 0$ and $u_n$ is Steiner symmetric;
\vskip.2cm
\item[{\it (ii)}] there exists $M_1=M_1(N,p,q,r_\Omega)>0$ such that
\[
\|u_n\|_{W^{1,p}(\Omega)}\le M_1,\qquad \text{for every}\ n\in\mathbb{N};
\]
\item[{\it (iii)}] $u_n\in L^\infty(\Omega)$ and there exists $M_2=M_2(N,p,q,r_\Omega)>0$ such that
\[
\|u_n\|_{L^\infty(\Omega)}\le M_2,\qquad \text{for every}\ n\in\mathbb{N};
\]
\item[{\it (iv)}] %when  $q<\infty$,  
there exists $\mathfrak{m}=\mathfrak{m}(N,p,q,r_\Omega)>0$ 
such that 
\[
\inf_{B_{r_\Omega/(3\sqrt{N})}} u_n\ge \mathfrak{m},\qquad \text{for every}\ n\in\mathbb{N}.
\]
\end{itemize}
\end{prop}
\begin{proof}
We first observe that $\lambda_p(\Omega)>0$, thanks to Proposition \ref{prop:positivity}. Thus, the results of Lemma \ref{lm:approssimazione} actually hold.
The property $(i)$ is quite straightforward, it is sufficient to observe that the minimization problem 
\[
\lambda_{p,q}(\Omega;V_n):=\inf_{u\in W^{1,p}_0(\Omega;|x|)}\Big\{\mathcal{G}_{p,n}(u)\,:\, \|u\|_{L^q(\Omega)}=1\Big\},
\]
is invariant by substituting $u$ with $|u|$. Moreover, by repeatedly applying  Lemma \ref{lm:zeropres},  Proposition \ref{pr:simV} and  \eqref{PZ1},   for every $u\ge 0$ admissible in the above  minimization problem, we have that  its Steiner symmetrization 
$$
\widetilde{u}:=\mathcal{S}_{\mathbf{e}_N}\circ \mathcal{S}_{\mathbf{e}_{N-1}} \circ \cdots  \circ \mathcal{S}_{\mathbf{e}_1} (u)
$$
 is still admissible. Moreover, by taking into account \eqref{scende} and  \eqref{PZ2},   we have
\[
\mathcal{G}_{p,n}(\widetilde{u})\le \mathcal{G}_{p,n}(u).
\]
 Thus, by minimality, we can suppose that each $u_n$ has the claimed properties.
\par
In order to bound uniformly the $W^{1,p}(\Omega)$ norm, we first observe that
\[
u\mapsto \|\nabla u\|_{L^p(\Omega)},
\]
is an equivalent norm on $W^{1,p}_0(\Omega)$, thanks to the fact that $\lambda_p(\Omega)>0$. More precisely, by Proposition \ref{prop:positivity}, we have that 
\[
\|u\|_{W^{1,p}(\Omega)}\leq C(N,p,r_{\Omega})\, \|\nabla u\|_{L^p(\Omega)}, 	\qquad \text{for every}\ u\in W^{1,p}_0(\Omega).
\]
Now, we observe that, by minimality of $u_n$ we have 
\begin{equation}
\label{gradienti_limitati}
\|\nabla u_n\|^p_{L^p(\Omega)}\le \int_{\Omega}|\nabla u_n|^{p}\,dx +\int_{\Omega}V_n\,|u_n|^{p}\,dx=\lambda_{p,q}(\Omega;V_n).
\end{equation}
By Lemma \ref{lemmainradSN}, we know that $r_\Omega<+\infty$ and $B_{r_\Omega}\subseteq\Omega$.
We can apply the estimate of Remark \ref{rem:stimaunife} with the ball $B_{r_\Omega}$ and get
\begin{equation}
\label{uniforma}
\lambda_{p,q}(\Omega;V_n)\le r_\Omega^{N-p-\frac{p}{q}\,N}\,\lambda_{p,q}(B_1)+\frac{1}{n+1}\, \left(\dfrac{N\,\omega_N}{N+\frac{q}{q-p}}\right)^\frac{q-p}{q}\,r_\Omega^{N+1-\frac{p}{q}\,N}.
%\frac{\lambda_{p,q}(B_1)}{r_\Omega^{p-N+\frac{p}{q}\,N}}+\frac{1}{n+1} \| \rho \|_{L^{\frac{q-p}{q}} (B_{r_{\Omega}})}.
%\left(\int_{B_{r_\Omega}} |x|^\frac{q}{q-p}\,dx\right)^\frac{q-p}{q},\qquad \text{for every}\ n\in\mathbb{N}.
\end{equation}
By combining this estimate with \eqref{gradienti_limitati}, we get a uniform upper bound on $\|\nabla u_n\|_{L^p(\Omega)}$, depending only on $N,p,q$ and $r_\Omega$. %By using the Poincar\'e inequality for $\Omega$ and Proposition \ref{prop:positivity}, we can obtain a similar uniform upper bound on $\|u_n\|_{L^p(\Omega)}$, as well. 
This concludes the proof of $(ii)$.
\par
We now consider point $(iii)$. The optimality condition for our minimization problem is given by the Euler-Lagrange equation 
\begin{equation}
\label{ELpepe}
-\Delta_p u_n+V_n\,u_n^{p-1}=\lambda_{p,q}(\Omega;V_n)\,u_n^{q-1},\qquad \text{in}\ \Omega,
\end{equation}
in weak form. In particular, by using that $V_n\,u_n^{p-1}\ge 0$, we get that each $u_n$ is a non-negative weak subsolution of
\[
-\Delta_p u_n\le \lambda_{p,q}(\Omega;V_n)\,u_n^{q-1},\qquad \text{in}\ \Omega.
\]
By appealing to Lemma \ref{lm:Linfty}, we get that $u_n\in L^\infty(\Omega)$, with the estimate
\[
\|u_n\|_{L^\infty(\Omega)}\le C_1\,\Big(\lambda_{p,q}(\Omega;V_n)\Big)^\frac{N}{p\,q-(q-p)\,N}.
\]
Observe that we also used that $u_n$ has unit $L^q(\Omega)$ norm. On account of \eqref{uniforma}, 
by defining 
\[
M_2:=C_1\,\left(r_\Omega^{N-p-\frac{p}{q}\,N}\,\lambda_{p,q}(B_1)+\left(\dfrac{N\,\omega_N}{N+\frac{q}{q-p}}\right)^\frac{q-p}{q}\,r_\Omega^{N+1-\frac{p}{q}\,N}\right)^\frac{N}{p\,q-(q-p)\,N},
\]
we thus get the claimed uniform estimate in $(iii)$, by observing that $M_2$ only depends on $N,p,q$ and $r_\Omega$.
\par
We finally come to property $(iv)$. %First we consider the case $q<\infty$.
 We note that $\|u_n\|_{L^\infty(\Omega)}$ is uniformly bounded from below, as well. Indeed, point $(ii)$ combined with a simple interpolation, gives
\[ 
1=\|u_n\|_{L^q(\Omega)}\le \|u_n\|_{L^\infty(\Omega)}^\frac{q-p}{q}\,\|u_n\|_{L^p(\Omega)}^\frac{p}{q}\leq  \|u_n\|_{L^\infty(\Omega)}^\frac{q-p}{q}\,M_1^\frac{p}{q}, 
\]
which implies
%\[
%\begin{split}
%1=\|u_n\|_{L^q(\Omega)}&\le \|u_n\|_{L^\infty(\Omega)}^\frac{q-p}{q}\,\|u_n\|_{L^p(\Omega)}^\frac{p}{q}\\
%&\le \|u_n\|_{L^\infty(\Omega)}^\frac{q-p}{q}\,\left(\frac{1}{\lambda_p(\Omega)}\right)^\frac{1}{q}\,\|\nabla u_n\|_{L^p(\Omega)}^p\\
%&\le \|u_n\|_{L^\infty(\Omega)}^\frac{q-p}{q}\,\left(\frac{1}{\lambda_p(\Omega)}\right)^\frac{1}{q}\,\lambda_{p,q}(\Omega;V_n).
%\end{split}
%\]
%
%By using again \eqref{uniforma} to bound the last term, we can easily obtain
\[
\|u_n\|_{L^\infty(\Omega)}\ge M_1^{-\frac{p}{q-p}},\qquad \text{for every}\ n\in\mathbb{N},
\]
and this lower bound only depends on $N,p,q$ and $r_\Omega$. We also observe that 
\[
\|u_n\|_{L^\infty(\Omega)}=\|u_n\|_{L^\infty(B_r)},\qquad \text{for every}\ r>0,
\] 
thanks to the Steiner symmetry of $u_n$. We now wish to appeal to Harnack's inequality, in order to conclude. To this aim, we observe that from \eqref{ELpepe} we get in particular that $u_n$ is a non-negative weak solution of
\[
-\mathrm{div}\mathbf{A}(\nabla u)=\mathbf{B}_n(x,u),\qquad \text{in}\ B_{r_\Omega}\subseteq\Omega,
\]
where 
\[
\mathbf{A}(\xi)=|\xi|^{p-2}\,\xi,\qquad \mathbf{B}_n(x,t)=\Big(\lambda_{p,q}(\Omega_n;V_n)\,u_n(x)^{q-p}-V_n(x)\Big)\,|t|^{p-2}\,t.
\]
We notice that for $x\in B_{r_\Omega}$ we have
\[
\begin{split}
|\mathbf{B}_n(x,t)|&\le |V_n(x)|\,|t|^{p-1}+|\lambda_{p,q}(\Omega_n;V_n)\,u_n(x)^{q-p}|\,|t|^{p-1}\\
&\le \left[r_\Omega+M_2^{q-p}\,\left(r_\Omega^{N-p-\frac{p}{q}\,N}\,\lambda_{p,q}(B_1)+ \left(\dfrac{N\,\omega_N}{N+\frac{q}{q-p}}\right)^\frac{q-p}{q}\,r_\Omega^{N+1-\frac{p}{q}\,N}\right)\right]\,|t|^{p-1}\\
&=: b\,|t|^{p-1},
\end{split}
\]
with $b=b(N,p,q,r_\Omega)>0$. Observe that we used \eqref{uniforma} and the uniform upper bound of point $(iii)$.
Accordingly, we can apply Harnack's inequality of \cite[Theorem 1.1]{Tru} in the cube $Q_{r_\Omega/\sqrt{N}}\subseteq B_{r_\Omega}$ so to get
\[
\sup_{Q_{r_\Omega/(3\sqrt{N})}} u_n\le H\,\inf_{Q_{r_\Omega/(3\sqrt{N})}} u_n,
\]
where the constant $H$ depends only on $N,p,q$ and  $r_\Omega>0$, in light of the discussion above. By observing that $B_{r_\Omega/(3\sqrt{N})}\subseteq Q_{r_\Omega/(3\sqrt{N})}$, we finally get
\[
H\,\inf_{B_{r_\Omega/(3\sqrt{N})}} u_n\ge \sup_{B_{r_\Omega/(3\sqrt{N})}} u_n=\|u_n\|_{L^\infty(\Omega)}\ge M_1^{-\frac{p}{q-p}},
\]
which gives the desired uniform lower bound, by setting $\mathfrak{m}=M_1^{-\frac{p}{q-p}}/H>0$.
%{\color{magenta} When $q=\infty$, we have that  \[u_n(0)=\|u_n\|_{L^\infty(\Omega)}=1 ,\qquad \text{for every}\ n\in\mathbb{N}, \]
%thanks to the symmetries of each function $u_n$.
%By applying  Morrey's inequality, we have that there exists a constant $C=C(N,p)$ such that  
%\[ u_n(x)\geq u_n(0)-C |x|=1-C|x| ,\qquad \text{for every}\ n\in\mathbb{N}.  \]
%Hence, for $\overline r=\frac  {1} {2C}$ we easily obtain that
%\[
%\inf_{B_{\overline{r}}} u_n\ge \frac 1 2 ,\qquad \text{for every}\ n\in\mathbb{N}.
%\]}
\end{proof}

\section{Proof of the Main Theorem: existence}
\label{sec:5}
In this section, we are going to prove the existence part of the Main Theorem stated at the beginning.
We start with the case $q<\infty$.
\begin{teo}[Case $q<\infty$]
\label{teo:main}
Let $1<p<\infty$ and assume that  $q>p$ satisfies \eqref{pq}, with $q<\infty$.
Let $\Omega\in\mathfrak{S}^N$ be such that  $\Omega\neq\mathbb R^N$. Then, the infimum defining $\lambda_{p,q}(\Omega)$ is attained by some Steiner symmetric function  $u\in W^{1,p}_0(\Omega)\setminus\{0\}$.
\end{teo}

\begin{proof}
First of all, by applying Proposition \ref{prop:positivity}, we can assure that  $\lambda_{p,q}(\Omega)>0$. We wish to apply the Direct Method in the Calculus of Variations: unfortunately, under our standing assumptions, in general the embedding 
\[
W^{1,p}_0(\Omega)\hookrightarrow L^q(\Omega),
\]
{\it is not} compact. However, we will show that it is possible to choose a particular minimizing sequence, which is precompact in $L^q(\Omega)$.
\par
More precisely, we want to use the sequence $\{u_n\}_{n\in\mathbb{N}}\subseteq W^{1,p}_0(\Omega;|x|)$ obtained in Proposition \ref{prop:fine}. Observe that this is indeed a minimizing sequence for our problem, thanks to \eqref{eq:limiteutile}.
In particular, this is a bounded sequence in $W^{1,p}_0(\Omega)$. 
Thus, it converges weakly to a function $u\in W^{1,p}_0(\Omega)$, up to a subsequence. It is not difficult to see that $u\not =0$: indeed, by testing the weak convergence in $L^p(\Omega)$ against the characteristic function of the ball $B_{r_\Omega/(3\sqrt{N})}$, we get
\[
\int_{B_{r_\Omega/(3\sqrt{N})}}u\,dx=\lim_{n\to\infty} \int_{B_{r_\Omega/(3\sqrt{N})}}u_n\,dx\ge \mathfrak{m}\,\left|B_{r_\Omega/(3\sqrt{N})}\right|>0,
\]
in light of point $(iv)$ of Proposition \ref{prop:fine}.
%we also have $\{u_n\}_{n\in\mathbb{N}}\subseteq W^{1,p}(B_{r_\Omega})$ and the latter is compactly embedded into $L^p(\Omega)$. This shows that the same subsequence as above converges strongly in $L^p(B_{r_\Omega})$, as well. By possibly passing to a further subsequence, we can also infer convergence almost everywhere in $B_{r_\Omega}$. In light of Proposition \ref{prop:fine}, we obtain
%\[
%0<\mathfrak{m}\le \lim_{n\to\infty}\inf_{B_{r_\Omega/(3\sqrt{N})}} u_n\le \lim_{n\to\infty}u_n(x)=u(x),\qquad \text{for a.\,e.}\ x\in B_{r_\Omega/(3\sqrt{N})}.
%\]
This shows that $u\not\equiv 0$, as claimed. In order to show that $u$ is an extremal for $\lambda_{p,q}(\Omega)$ and conclude, it is sufficient to show that 
\begin{equation}
\label{volio}
\lim_{n\to\infty} \|u_n-u\|_{L^q(\Omega)}=0,
\end{equation}
possibly up to a subsequence.
\par
In what follows, when needed, we can always think of our functions $\{u_n\}_{n\in\mathbb{N}}$ and $u$ as defined on the whole $\mathbb{R}^N$, by extending them to be $0$ on the complement of $\Omega$. Thus, we also have $\{u_n\}_{n\in\mathbb{N}}\subseteq W^{1,p}(\mathbb{R}^N)$ and in particular $\{u_n\}_{n\in\mathbb{N}}\subseteq W^{1,p}(B_R)$, for every $R>0$.
Recall that for every $R>0$, the embedding $W^{1,p}(B_R)\hookrightarrow L^p(\Omega)$ is compact. By using this fact and the weak convergence in $W^{1,p}_0(\Omega)$ exposed above, a standard argument (see for example \cite[Lemma 3.8.7 \& Remark 3.9.5]{BraBook}) also gives that 
\begin{equation}
\label{convpalle}
\lim_{n\to\infty} \|u_n-u\|_{L^p(B_R)}=\lim_{n\to\infty} \|u_n-u\|_{L^p(\Omega\cap B_R)}=0,\qquad \text{for every}\ R>0.
\end{equation} 
We are going to show that the same holds for the sequence of the gradients $\{\nabla u_n\}_{n\in\mathbb{N}}$.
We use again  the Euler-Lagrange equation \eqref{ELpepe}.
%\[
%-\Delta_p u_n+V_n\,u_n^{p-1}=\lambda_{p,q}(\Omega;V_n)\,u_n^{q-1},\qquad \text{in}\ \Omega.
%\]
Let $R>0$ and $\zeta\in C^{\infty}_0(\mathbb{R}^N)$  a cut-off function such that
\[
0\le \zeta\le 1,\quad \zeta=1 \text{ on } B_{R},\quad \zeta=0\ \text{on}\ \mathbb{R}^{N}\setminus B_{2R}, \quad \|\nabla \zeta\|_{\infty}\le \frac{C}{R}.
\]
By testing the weak formulation of \eqref{ELpepe} with the admissible function $\zeta\, (u_n-u)\in W^{1,p}_0(\Omega)$, we obtain
\[
\begin{split}
\int_\Omega \langle|\nabla u_n|^{p-2}\nabla u_n, \nabla u_n-\nabla u\rangle\,\zeta\,dx &=\lambda_{p,q}(\Omega; V_n)\,\int_{\Omega}u_n^{q-1}\,(u_{n}-u)\,\zeta\,dx\\
&-\frac{1}{n+1}\int_{\Omega}|x|\,u_n^{p-1}\,(u_n-u)\,\zeta\,dx\\
&-\int_\Omega \langle|\nabla u_n|^{p-2}\nabla u_n, \nabla \zeta\rangle\,(u_n-u) \,dx .
\end{split}
\]
It is not difficult to see that the integrals in the right-hand side converges to $0$, as $n$ goes to $\infty$. Indeed, by Proposition \ref{prop:fine} and the properties of $\zeta$, we have
\[
\lambda_{p,q}(\Omega; V_n)\,\left|\int_{\Omega}u_n^{q-1}\,(u_{n}-u)\,\zeta\,dx\right|\le \lambda_{p,q}(\Omega;V_n)\,M_2^{q-1}\,|B_{2R}|^\frac{p-1}{p}\,\left(\int_{B_{2R}} |u_n-u|^p\,dx\right)^\frac{1}{p},
\]
\[
\left|\frac{1}{n+1}\int_{\Omega}|x|\,u_n^{p-1}\,(u_n-u)\,\zeta\,dx\right|\le 2\,R\,M_2^{p-1}\,|B_{2R}|^\frac{p-1}{p}\,\left(\int_{B_{2R}} |u_n-u|^p\,dx\right)^\frac{1}{p},
\]
and
\[
\left|\int_\Omega \langle|\nabla u_n|^{p-2}\nabla u_n, \nabla \zeta\rangle\,(u_n-u) \,dx\right|\le \frac{C}{R}\,M_1^{p-1}\,\left(\int_{B_{2R}} |u_n-u|^p\,dx\right)^\frac{1}{p}.
\]
If we use \eqref{convpalle}, we get the claim.
In particular, by taking the limit as $n$ goes to $\infty$, we deduce that 
\[
\lim_{n\to\infty}\int_{\Omega} \langle|\nabla u_n|^{p-2}\,\nabla u_n, \nabla u_n-\nabla u\rangle\,\zeta\,dx=0.
\]
On the other hand, the previously inferred weak convergence of $\nabla u_n$ to $\nabla u$ gives also
\[
\lim_{n\to\infty}\int_{\Omega} \langle|\nabla u|^{p-2}\,\nabla u, \nabla u_n-\nabla u\rangle\,\zeta\,dx=0.
\]
By substracting these two identities we deduce that
\[
\lim_{n\to\infty}\int_{\Omega} \langle|\nabla u_n|^{p-2}\,\nabla u_n-|\nabla u|^{p-2}\nabla u, \nabla u_n-\nabla u\rangle\,\zeta\,dx=0.
\]
We can now proceed in a standard way, by exploiting the monotonicity properties of the $p-$Laplacian (see for example the proof of \cite[Lemma B.1]{BPZ}), so to obtain that
%Thanks to the elementary inequality (see \cite[ Section 10, equation (I)]{Lind}) 
%\[
%(|z|^{p-2}z-|w|^{p-2}w)\cdot (z-w)\ge 2^{2-p}|z-w|^{p}, \quad \text{for every }z,w\in\mathbb{R}^{d},
%\]
%and by exploiting the fact that $\zeta$ is constantly equal to $1$ on $B_R$
%we deduce that
\begin{equation}
\label{convpallegrad}
\lim_{n\to\infty}\|\nabla u_n-\nabla u\|_{L^{p}(B_R)}=\lim_{n\to\infty}\|\nabla u_n-\nabla u\|_{L^{p}(\Omega\cap B_R)}=0,\qquad \text{for every}\ R>0.
\end{equation}
Before proceeding further, we observe that from \eqref{convpalle} and \eqref{convpallegrad}, by using a standard diagonal argument, we can infer existence of a subsequence (not relabelled) such that we also have
\[
\lim_{n\to\infty} u_n(x)=u(x),\qquad \lim_{n\to\infty} \nabla u_n(x)=\nabla u(x),\qquad \text{for a.\,e.}\ x\in\Omega.
\]
This in turn permits to apply the so-called {\it Brezis-Lieb Lemma}  (see \cite[Theorem 1]{BreLie}). We thus obtain
\[
\lim_{n\to\infty}\left(\int_\Omega |\nabla u_n|^{p}\,dx-\int_{\Omega}|\nabla u_n-\nabla u|^{p}\,dx\right)=\int_{\Omega}|\nabla u|^{p}\,dx.
\]
Observe that from \eqref{eq:limiteutile}, we know that
\[
\lim_{n\to\infty}\int_\Omega |\nabla u_n|^{p}\,dx=\lambda_{p,q}(\Omega).
\]
Thus, by joining the last two informations,
we can conclude that
 \begin{equation}
 \label{BL3}
\lim_{n\to\infty}\int_{\Omega}|\nabla u_n-\nabla u|^{p}\,dx=\lambda_{p,q}(\Omega)-\int_{\Omega}|\nabla u|^{p}\,dx.
\end{equation}
Analogously, from the almost everywhere convergence of $u_n$ to $u$, we have
\begin{equation}
\label{BL}
\lim_{n\to\infty}\left(\int_\Omega |u_n|^{q}\,dx-\int_{\Omega}|u_n-u|^{q}\,dx\right)=\int_{\Omega}|u|^{q}\,dx.
\end{equation}
In particular, there exists $n_0\in\mathbb{N}$ such that
\[
A_n:=\int_\Omega |u_n|^{q}\,dx-\int_{\Omega}|u_n-u|^{q}\,dx\ge \frac{1}{2}\,\int_\Omega |u|^q\,dx>0,\qquad \text{for every}\ n\ge n_0.
\]
Observe that we used that $u\not\equiv 0$, by the first part of the proof.
By taking into account that each $u_n$ has unit  $L^q(\Omega)$ norm, we obtain, for every $n\ge n_0$, 
\[
\begin{split}
\lambda_{p,q}(\Omega)=\lambda_{p,q}(\Omega)\,\|u_n\|_{L^q(\Omega)}^p&\le \lambda_{p,q}(\Omega)\,A_n^\frac{p}{q}+\lambda_{p,q}(\Omega)\,\|u_n-u\|_{L^{q}(\Omega)}^p\\
&-c\,\min\left\{A_n^\frac{p}{q},\|u_n-u\|_{L^{q}(\Omega)}^p\right\},
%&\le \lambda_{p,q}(\Omega)\, \left(\int_{\Omega} \Big||u_n|^q-|u_n-u|^q\Big|\, dx+\int_{\Omega}|u_n-u|^{q} dx\right)^\frac{p}{q}\\
%&\le \lambda_{p,q}(\Omega)\, \int_{\Omega} (|u_n|^q-|u_n-u|^q)\, dx+\|\nabla u_n-\nabla u\|_{L^p(\Omega)}^q.
\end{split}
\]
with $c=c(p/q)>0$, thanks to the enforced sub-additivity of the power $p/q<1$, given by Lemma \ref{lm:quantitativo}. 
By using the definition of $\lambda_{p,q}(\Omega)$, we then get, for every $n\ge n_0$, 
\[
\begin{split}
\lambda_{p,q}(\Omega)&\le \lambda_{p,q}(\Omega)\,A_n^\frac{p}{q}+\int_{\Omega}|\nabla u_n-\nabla u|^{p}\,dx-c\,\min\left\{A_n^\frac{p}{q},\|u_n-u\|_{L^{q}(\Omega)}^p\right\}.
\end{split}
\]
We take the limit as $n$ goes to $\infty$: in light of \eqref{BL3} and \eqref{BL}, this yields
\[
\begin{split}
0&\le \lambda_{p,q}(\Omega)\,\left(\int_\Omega |u|^q\,dx\right)^\frac{p}{q}-\int_\Omega |\nabla u|^p\,dx-c\,\lim_{n\to\infty}\min\left\{A_n^\frac{p}{q},\|u_n-u\|_{L^{q}(\Omega)}^p\right\},
\end{split}
\]
that is
\[
\int_\Omega |\nabla u|^p\,dx+ c\,\lim_{n\to\infty}\min\left\{A_n^\frac{p}{q},\|u_n-u\|_{L^{q}(\Omega)}^p\right\}\le \lambda_{p,q}(\Omega)\,\left(\int_\Omega |u|^q\,dx\right)^\frac{p}{q}.
\]
On the other hand, by definition of $\lambda_{p,q}(\Omega)$, we have
\[
\lambda_{p,q}(\Omega)\,\left(\int_\Omega |u|^q\,dx\right)^\frac{p}{q}\le \int_\Omega |\nabla u|^p\,dx.
\]
By joining the last two equations, we get in particular that
\[
\lim_{n\to\infty}\min\left\{A_n^\frac{p}{q},\|u_n-u\|_{L^{q}(\Omega)}^p\right\}=0.
\]
%and
%\begin{equation}
%\label{opt}
%\int_{\Omega} |\nabla u|^{p}\,dx= \lambda_{p,q}(\Omega)\left(\int_{\Omega}|u|^{q}\,dx\right)^{\frac{p}{q}}.
%\end{equation}
By recalling that $A_n$ is uniformly bounded from below, the previous fact implies that $\{u_n\}_{n\in\mathbb{N}}$ converges strongly in $L^q(\Omega)$ to $u$, i.e. we prove \eqref{volio}. This proves that $u$ is an extremal for $\lambda_{p,q}(\Omega)$.
\par
Once we get that $u$ is an extremal, we can go back to \eqref{BL3} and obtain that 
\[
\lim_{n\to\infty} \|\nabla u_n-\nabla u\|_{L^p(\Omega)}=0.
\]
On account of the assumption $\lambda_p(\Omega)>0$, this in turn implies that $u_n$ converges to $u$ strongly in $L^p(\Omega)$, as well. Hence, the contractivity of the Steiner symmetrization
%continuity of the Steiner symmetrization with respect to the strong $W^{1,p}$ convergence, see \cite[Theorem 1]{Bu}, 
and the Steiner symmetry of each $u_n$ gives
\[
\begin{split}
0=\lim_{n\to\infty}\|u_n-u\|_{L^p(\Omega)}&\ge \limsup_{n\to\infty} \|\mathcal{S}_{\mathbf{e}_i}(u_n)-\mathcal{S}_{\mathbf{e}_i}(u)\|_{L^p(\Omega)}\\
&=\limsup_{n\to\infty} \|u_n-\mathcal{S}_{\mathbf{e}_i}(u)\|_{L^p(\Omega)}\ge 0,\qquad \text{for every}\ i\in\{1,\dots,N\}.
\end{split}
\]
%\BBB\marginpar{Forse piu' chiaro cosi?}
%\[
%0=\lim_{n\to\infty} \|u_n-u\|_{W^{1,p}(\Omega)}=\lim_{n\to\infty} \|u_n-\mathcal{S}_i(u)\|_{W^{1,p}(\Omega)},\qquad \text{for every}\ i\in\{1,\dots,N\}.
%\]
%\EEE
By uniqueness of the limit, we obtain that $u=\mathcal{S}_{\mathbf{e}_i}(u)$ for every $i\in\{1,\dots,N\}$, as desired.
 This eventually concludes the proof.
\end{proof}
By a limit argument, we get the existence of an optimal function at the endpoint $q=\infty$, as well.
\begin{teo}[Case $q=\infty$]
\label{teo:maininfty}
Let $N<p<\infty$  and let $\Omega\in\mathfrak{S}^N$ be such that $\Omega\not= \mathbb R^N$.
For every $p<q<\infty$, let $u_q\in W^{1,p}_0(\Omega)$ be a Steiner symmetric function 
such that  
\[
\|u_q\|_{L^q(\Omega)}=1,\qquad \lambda_{p,q}(\Omega)=\int_\Omega |\nabla u_q|^p\,dx.
\]
Then, there exists a Steiner symmetric function $u_\infty\in W^{1,p}_0(\Omega)$ such that
\begin{equation}\label{limiteqinfty}
\lim_{q\to \infty} \|u_q-u_\infty\|_{W^{1,p}(\Omega)}=0.
\end{equation}
Moreover, the function $u_\infty$ is the unique Steiner symmetric extremal of $\lambda_{p,\infty}(\Omega)$. Finally, we have 
\[
\lambda_{p,q}(\Omega)\,(u_{q})^{q-1}\ \stackrel{q\to\infty}{\rightharpoonup}\ \lambda_{p,\infty}(\Omega)\,\delta_0,\qquad \mbox{in}\ \mathscr{D}'(\Omega).
\]
\end{teo}
\begin{proof} 
First of all, we note that the  existence of  $u_q\in W^{1,p}_0(\Omega)$ as in the statement  is ensured by Theorem \ref{teo:main}. We also observe that $\lambda_p(\Omega)>0$ thanks to Proposition \ref{prop:positivity}. Thus, by Poincar\'e inequality we have that
\[
u\mapsto \|\nabla u\|_{L^p(\Omega)},
\]
is an equivalent norm on $W^{1,p}_0(\Omega)$.
%We consider  every function $u_q$ extended equal to $0$ outside $\Omega$. Thanks  to  \eqref{GNS} applied with $r=\infty$, we have that  $ W^{1,p}_0(\Omega)\hookrightarrow C^{0,\alpha_p}(\overline{\Omega})$ %(since every sequence in $C_{0}^\infty(\Omega)$ converging to a function $u$ in $W^{1,p}(\Omega)$, converges to $u$ also in the uniform convergence). Moreover, for the same argument, we have that $ W^{1,p}_0(\Omega)\hookrightarrow C_0(\Omega).$ 
%Hence, in order to show that $\{u_q\}_{q>p}$ is precompact in $C_0(\Omega)$, it is sufficient to show that it is precompact in $ W^{1,p}_0(\Omega)$.
By Sobolev embeddings in the case $p>N$, we have that 
\[
W^{1,p}_0(\Omega)\hookrightarrow C^{0,\alpha_p}(\overline{\Omega}),\qquad \text{with}\ \alpha_p=1-\frac{N}{p}.
\]
 More precisely, by Morrey's inequalities, for every $q>p$ we have that
 %  thanks to Lemma \ref{lemmainradSN},  Remark \ref{casep>N} and Proposition \ref{prop:positive}, we have that $\lambda_{p,q}(\Omega)>0$ for every $p<q<\infty$. Hence, the  existence of $u_q$ as in the statement  is ensured by Theorem \ref{main}. We consider  every function $u_q$ extended equal to $0$ outside $\Omega$. Thanks to  Morrey's inequality, for every $q>p$,  we have that 
\begin{equation}\label{eq:morrey1}
			|u_q(x)-u_q(y)|\le \mathcal{M}_1\,\|\nabla u_q\|_{L^p(\Omega)}\, |x-y|^{\alpha_p}, \qquad \text{for every}\ x,y\in\mathbb{R}^N,
		\end{equation}
		and 
		\begin{equation}
		\label{eq:morrey2}
			|u_q(x)|\le \mathcal{M}_2\, \|\nabla u_q\|_{L^p(\Omega)}^\frac{N}{p}\,\|u_q\|_{L^p(\Omega)}^{1-\frac N p}, \qquad \mbox{ for every } x\in\mathbb{R}^N. 
		\end{equation}
On account of the monotonicity of $\lambda_{p,q}(\Omega)$ with respect to the set inclusion and the fact that $\Omega$ has finite inradius (thanks to Lemma \ref{lemmainradSN}), we have that
\[
\int_\Omega |\nabla u_q|^p\,dx=\lambda_{p,q}(\Omega)\le \frac{\lambda_{p,q}(B_1)}{r_\Omega^{p+\frac{p}{q}\,N-N}},\qquad \text{for every}\ p< q< \infty.
\]
By using that $q\mapsto \lambda_{p,q}(B_1)$ is continuous on every compact interval (see \cite[Theorem 1]{AFI} and \cite{Er}) and that by
\cite[Corollary 6.2]{BraPriZag2} 
	\begin{equation}
	\label{boundgrad}
\lim_{q\to\infty}\lambda_{p,q}(B_1)=\lambda_{p,\infty}(B_1)<+\infty,
	\end{equation}
we thus obtain that the family  $\{u_q\}_{q>p}$ is bounded in $W^{1,p}_0(\Omega)$. 
By using this fact in conjunction with  \eqref{eq:morrey1} and \eqref{eq:morrey2}, we get that the family $\{u_q\}_{q>p}$ is equicontinuous and bounded in $C^0_{\rm b}(\overline\Omega)$, the space of continuous and bounded functions over $\overline\Omega$.  
%Moreover, since $\lambda_{p}(\Omega)>0$, by using  the Poincaré inequality, we obtain that 
%\begin{equation}\label{boundnormep}\sup_{q>q_0}\|u_{q}\|_{L^{p}(\Omega)}\leq \left(\frac 1  {\lambda_{p}(\Omega)} \right)^{1/p} \sup_{q>q_0} \|\nabla u_q\|_{L^p(\Omega)}=M.\end{equation}
 %Then   the  family $\{u_q\}_{q\geq q_0}$ is also  bounded  in $W^{1,p}_0(\Omega)$. 
 By applying the Ascoli-Arzel\`a Theorem on the compact set $\overline{B_{r_\Omega}}\subseteq\overline\Omega$ and the reflexivity of the space  $W^{1,p}_0(\Omega)$, we have that there exists a sequence $\{u_{q_n}\}_{n\in\mathbb{N}}\subseteq \{u_q\}_{q>p}$ and a function $u_{\infty}\in W^{1,p}_0(\Omega)\cap C^0(\overline{B_{r_\Omega}})$  such that  $\{u_{q_n}\}_{n\in\mathbb{N}}$  weakly converges to $u_{\infty}$ in $W^{1,p}_0(\Omega)$ and  uniformly on the compact set $\overline{B_{r_\Omega}}$.
 \par
By the properties of $u_q$, we have
\[
\int_\Omega |\nabla u|^p\,dx=\lambda_{p,q}(\Omega)=\lambda_{p,q}(\Omega)\,\int_\Omega |u|^q\,dx\le \lambda_{p,q}(\Omega)\,\|u_q\|_{L^\infty(\Omega)}^{q-p}\,\int_\Omega |u_q|^p\,dx.
\]
By using the Poincar\'e inequality on the left-hand side and then simplifying the $L^p$ norm of $u_q$, we obtain with simple manipulations
\[
\left(\frac{\lambda_{p}(\Omega)}{\lambda_{p,q}(\Omega)}\right)^\frac{1}{q-p}\le \|u_q\|_{L^\infty(\Omega)}.
\]
%By interpolation, for every $q>p$,  we have that 
%\[
%1=\|u_{q}\|_{L^{q}(\Omega)}\leq \|u_{q}\|^\frac{p}{q}_{L^{p}(\Omega)}  \|u_{q}\|^{1-\frac{p}{q}}_{L^\infty(\Omega)}\leq M^\frac{p}{q}\,\|u_{q}\|^{1-\frac{p}{q}}_{L^\infty(\Omega)},  
%\]
%which  gives  
%\[
%\|u_{q}\|_{L^\infty(\Omega)}\geq M^{\frac p {q-p} }.
%\]
 By the Steiner symmetry, we know that $u_q$ is maximal at the origin. In particular, we get that 
 \[
u_{q_n}(0)\geq  \left(\frac{\lambda_{p}(\Omega)}{\lambda_{p,q_n}(\Omega)}\right)^\frac{1}{q_n-p} \qquad \text{for every}\ n\in\mathbb{N}.
\] 
By passing to the limit as $n$ goes to $\infty$, we obtain 
\[
u_{\infty}(0)=\lim_{n\to \infty}
u_{q_n}(0)\geq  1.
\]
Hence $u_{\infty}\not\equiv0$ and 
\begin{equation}\label{boundnorma}
\|u_\infty\|_{L^{\infty}(\Omega)}\ge 1.
\end{equation}
On the other hand, by using the definition of $\lambda_{p,\infty}(\Omega)$, the weak convergence in $W^{1,p}(\Omega)$ of the sequence $\{u_{q_n}\}_{n\in\mathbb{N}}$ and \eqref{boundgrad} we get
\begin{equation}\label{boundgrad2}
\lambda_{p,\infty}(\Omega)\,\|u_\infty\|^{p}_{L^{\infty}(\Omega)}\le \int_{\Omega}|\nabla u_\infty|^p\,dx\le \lim_{n\to \infty}\int_{\Omega}|\nabla u_{q_n}|^p\,dx= \lim_{n\to \infty}\lambda_{p,q_n}(\Omega)=\lambda_{p,\infty}(\Omega).
\end{equation}
By combining \eqref{boundnorma} and \eqref{boundgrad2}, we deduce
\[
\|u_\infty\|_{L^{\infty}(\Omega)}= 1 \qquad\text{and}\qquad \int_{\Omega}|\nabla u_{\infty}|^p\,dx=\lambda_{p,\infty}(\Omega),
\]
showing at first that $u_\infty$ is an extremal for $\lambda_{p,\infty}(\Omega)$. 
\par
We can now improve the convergence of the sequence $\{u_{q_n}\}_{n\in\mathbb{N}}$. Indeed, the minimality of $u_\infty$ implies that all the inequalities in \eqref{boundgrad2} are actually equalities. In particular, we get 
\[
\lim_{n\to\infty}\|\nabla u_{q_n}\|_{L^{p}(\Omega)}= \|\nabla u_{\infty}\|_{L^{p}(\Omega)}.
\]
By Clarkson's inequalities, this in turn yields that 
\[
\lim_{n\to \infty}\|\nabla u_{q_n}-\nabla u_{\infty}\|_{L^{p}(\Omega)}=0,
\]
which gives the claimed compactness in $W^{1,p}_0(\Omega)$. 
\par
We also observe that, by the same argument used at the end of the proof of Theorem \ref{teo:main}, we have that $u_{\infty}$ is Steiner symmetric. We wish to prove that it is the {\it unique} Steiner symmetric extremal.
By Lemma \ref{lm:maxpoint}, we have that it solves
\[
\int_\Omega \langle |\nabla u_{\infty}|^{p-2}\,\nabla u_{\infty},\nabla \varphi\rangle\,dx=\lambda_{p,\infty}(\Omega)\,\varphi(0),\qquad \text{for every}\ \varphi\in W^{1,p}_0(\Omega),
\]
and $x=0$ is the unique maximum point of $u_{\infty}$ (the fact that it concides with the origin follows from the Steiner symmetry). 
Let us assume that there exists another $v\not=u_{\infty}$ which is a Steiner symmetric  extremal of $\lambda_{p,\infty}(\Omega)$. Then 
\[
\left\|\frac{u_{\infty}+v}{2}\right\|_{L^\infty(\Omega)}= \frac{u_{\infty}(0)+v(0)}{2}=1,
\] 
i.e. $(u_{\infty}+v)/2$ would be an admissible test function for $\lambda_{p,\infty}(\Omega)$. By the strict convexity of the $p-$Dirichlet integral we would get a contradiction. Thus, we obtain that $u_{\infty}$ is the unique Steiner symmetric  extremal of $\lambda_{p,\infty}(\Omega)$: accordingly,  
every convergent sequence  $\{u_{q_n}\}_{n\in\mathbb{N}}$  converges to the same limit $u_{\infty}$.
This gives the convergence of the whole family, i.e. we have \eqref{limiteqinfty}.
\par
Finally, by using \eqref{limiteqinfty}, we obtain for every $\varphi\in C^\infty_0(\Omega)$
 \[
\begin{split}
\lim_{q\to\infty} \lambda_{p,q}(\Omega)\,\int_\Omega u_{q}^{q-1}\,\varphi\,dx&=\lim_{n\to\infty}\int_\Omega \langle |\nabla u_{q}|^{p-2}\,\nabla u_{q},\nabla \varphi\rangle\,dx\\
&=\int_\Omega \langle |\nabla u|^{p-2}\,\nabla u,\nabla \varphi\rangle\,dx=\lambda_{p,\infty}(\Omega)\,\varphi(0),
\end{split} 
\]
which gives that
\[
\lambda_{p,q}(\Omega)\,(u_{q})^{q-1}\ \stackrel{q\to\infty}{\rightharpoonup}\ \lambda_{p,\infty}(\Omega)\,\delta_0,\qquad \mbox{in}\ \mathscr{D}'(\Omega).
\]
This concludes the proof.
\end{proof}

\section{Examples and counter-examples}
\label{sec:6}

We start by showing a couple of examples to which our main result applies.
\begin{exa}\label{ex.strip}
Let $1\le k\le N-1$, a noteworthy application of Theorem \ref{teo:main} is to open sets of the type 
\[
S=\mathbb{R}^{N-k}\times \omega,
\]
where $\omega\subseteq\mathbb{R}^k$ is an open set such that $\omega\in \mathfrak{S}^k$ and $\omega\neq\mathbb{R}^k$. Clearly $S\in\mathfrak{S}^N$ with $S\neq\mathbb{R}^N$:
  by virtue of Theorem \ref{teo:main} and Theorem \ref{teo:maininfty}, we deduce that $\lambda_{p,q}(S)$ is attained, when  $1<p<\infty$ and   $q>p$ satisfies \eqref{pq}.
In particular, for $q\neq\infty$ satisfying  \eqref{pq}, Theorem \ref{teo:main} ensures the existence of a positive solution to the boundary value problem
\[
-\Delta_p u=\lambda\,u^{q-1}\quad \text{in}\ S,\qquad u\in W^{1,p}_0(S),
\]
for every $\lambda>0$. This applies for example to the slab
\[
S=\mathbb{R}^{N-1}\times(-1,1),
\]
thus generalizing the result obtained in \cite{AmTo}, where existence is shown in the case  $N=2$, $k=1$ and  $p=2< q<\infty$. Notice that, in this example,  the assumption $p<q$ is crucial, existence failing when $p=q$. 

\end{exa}

\begin{exa}[Infinite cross]
Another example of application is given by the set
\[
C=\left((-1,1)\times\mathbb{R}\right)\cup \left( \mathbb{R}\times(-1,1)\right)\subseteq\mathbb{R}^2.
\]
This is a two-dimensional Steiner symmetric open set: accordingly, there exists a minimizer for $\lambda_{p,q}(C)$, for every $1<p<\infty$ and $q>p$ satisfying \eqref{pq}. As above, the superhomogeneous Lane-Emden equation
\[
-\Delta_p u=\lambda\,u^{q-1}\quad \text{in}\ C,
\]
admits a positive solution $u\in W^{1,p}_0(\Omega)$, for every $\lambda>0$. Here the situation is slightly different with respect to Example \ref{ex.strip}: indeed, observe that for $p=q=2$ we still have existence of an extremal, see for example \cite[Example 3.2]{BBO}.
\end{exa}

%\subsection{Optimality of the assumptions}

We now wish to show the optimality of the assumptions (S1) and (S2) in Definition \ref{S}, in the case of unbounded sets for which the embedding $W^{1,p}_0(\Omega)\hookrightarrow L^p(\Omega)$ fails to be compact.
Namely, we give two simple counterexamples, aiming at showing that if we drop one of the assumptions, then the conclusion of Theorem \ref{teo:main} may fail to hold.

\begin{exa}[Sets satisfying (S2), but not satisfying (S1)]\label{Halfstrip} 
We consider a set of the form
\[
S=\omega\times \mathbb{R},
\]
where the open set $\omega\subsetneq\mathbb{R}^{N-1}$ is such that $\omega\in\mathfrak{S}^{N-1}$. 
%satisfy
%\[
%\lambda_{p}(S)=\lambda_{p}((-1,1))>0\qquad \text{and}\qquad\lambda_{p,q}(S)>0.
%\]
We also set
\[
S^+=\omega\times (0,+\infty),
\]
which satisfies the second assumption (S2) of Definition \ref{S}, but it does not satisfy the first one (S1). We claim that $\lambda_{p,q}(S^+)$ {\it is not} attained in $W^{1,p}_0(S^+)$, when $p<q$ satisfy \eqref{pq}.
%As already noticed in Example \ref{ex.strip}, when  $1<p<\infty$ and $q$ satisfies \eqref{pq},  we have that  $\lambda_{p,q}(S)>0$. Moreover, 
To this aim, we start by noticing that 
\begin{equation}
\label{uguaslab}
\lambda_{p,q}(S)=\lambda_{p,q}(S^+).
\end{equation}
Indeed, the fact that $\lambda_{p,q}(S)\le \lambda_{p,q}(S^+)$ simply follows by the inclusion $S^+\subseteq S$ and the monotonicity of $\lambda_{p,q}$. On the other hand, for every $\varepsilon>0$ there exists $\phi_\varepsilon\in C^\infty_0(S)$  such that
\[
\lambda_{p,q}(S)+\varepsilon\ge \int_S |\nabla \phi_\varepsilon|^p\,dx,\qquad \| \phi_\varepsilon\|_{L^q(S)}=1.
\]
Since $\phi_\varepsilon$ is compactly supported, its support is contained in $\omega\times (-L,+\infty)$, for $L>0$ large enough. Therefore, by translation invariance of the sharp Poincar\'e-Sobolev inequality, we have
\[
\lambda_{p,q}(S)+\varepsilon\ge \lambda_{p,q}(\omega\times (-L,+\infty))=\lambda_{p,q}(S^+).
\]
By the arbitrariness of $\varepsilon>0$, we deduce that $\lambda_{p,q}(S)\ge \lambda_{p,q}(S^+)$, as well.
%In the case $q=\infty$, it is sufficient to pass to the limit  in \eqref{uguaslab}  as $q\to \infty$ and to apply \cite[Corollary 6.2]{BPZ}, to get $\lambda_{p,\infty}(S^+)=\lambda_{p,\infty}(S).$
\par 
By exploiting \eqref{uguaslab}, it is not difficult to obtain that  $\lambda_{p,q}(S^+)$ does not admit an extremal. It is sufficient to reproduce the same argument of \cite[Remark 5.5]{B2}: any positive extremal $u$ for $\lambda_{p,q}(S_+)$ would be an extremal for $\lambda_{p,q}(S)$, as well, thanks to \eqref{uguaslab}. In particular, $u\not\equiv0$ would be a $p-$superharmonic function in $S$, identically vanishing on the set $\omega\times(-\infty,0)$. This would violate the minimum principle.
\end{exa}

\begin{exa}[Sets satisfying (S1), but not satisfying (S2)]
\label{exa:pinchedslab}
By keeping the same notation as in Example \ref{Halfstrip}, we first observe that for every open set $\Omega\subseteq\mathbb{R}^N$  such that
\begin{equation}
\label{bulla}
y_0+S^+\subseteq\Omega\subsetneq S\qquad \text{and}\qquad |S\setminus \Omega|>0,
\end{equation}
the constant $\lambda_{p,q}(\Omega)$ can not be attained.
Indeed, by monotonicity we have
\[
\lambda_{p,q}(S^+)\le \lambda_{p,q}(\Omega)\le \lambda_{p,q}(S),
\]
and thus $\lambda_{p,q}(\Omega)=\lambda_{p,q}(S)$, thanks to \eqref{uguaslab}. By exploiting the same argument based on the minimum principle as in Example \ref{Halfstrip}, we get that $\lambda_{p,q}(\Omega)$ can not admit an extremal.
%Suppose by contradiction that $u$ is a positive minimizer for $\lambda_{p,q}(\Omega)$. Then, the function $v$ obtained by extending $u$ equal to zero in $S\setminus \Omega$ is a positive minimizer for the set $S$. By the minimum principle this gives a contradiction, since $v$ identically vanishes on the ball $B_r(x)$.
\par
As an explicit example of this situation, we can take the following {\it pinched slab}
\[
\Omega^-_\varepsilon=\left\{(x',x_N)\in\mathbb{R}^{N-1}\times \mathbb{R}:\ |x_N|< 1-\varepsilon\,\varphi\left(\frac{|x'|}{\varepsilon}\right)\right\},\qquad 0<\varepsilon<1,
\]
where $\varphi\in C^\infty_0((-1,1))\setminus\{0\}$ is non-negative even function, monotone decreasing on $[0,1)$ and  with $\varphi(0)=1$. Observe that the set $\Omega_\varepsilon$ satisfies the first assumption (S1) of Definition \ref{S}, but it fails to satisfy (S2): indeed, it is not convex in the directions $\{\mathbf{e}_1,\dots,\mathbf{e}_{N-1}\}$, because of the vertical inward pinching. It is not difficult to see that $\Omega^-_\varepsilon$ satisfies \eqref{bulla}, thus the constant $\lambda_{p,q}(\Omega^-_\varepsilon)$ is not attained.
\par
On the contrary, it is worth observing that the set
\[
\Omega^+_\varepsilon=\left\{(x',x_N)\in\mathbb{R}^{N-1}\times \mathbb{R}:\ |x_N|< 1+\varepsilon\,\varphi\left(\frac{|x'|}{\varepsilon}\right)\right\},\qquad 0<\varepsilon<1,
\]
satisfies both (S1) and (S2), thus $\lambda_{p,q}(\Omega^+_\varepsilon)$ is attained, for $q>p$ satisfying \eqref{pq}.
\begin{figure}
\includegraphics[scale=.3]{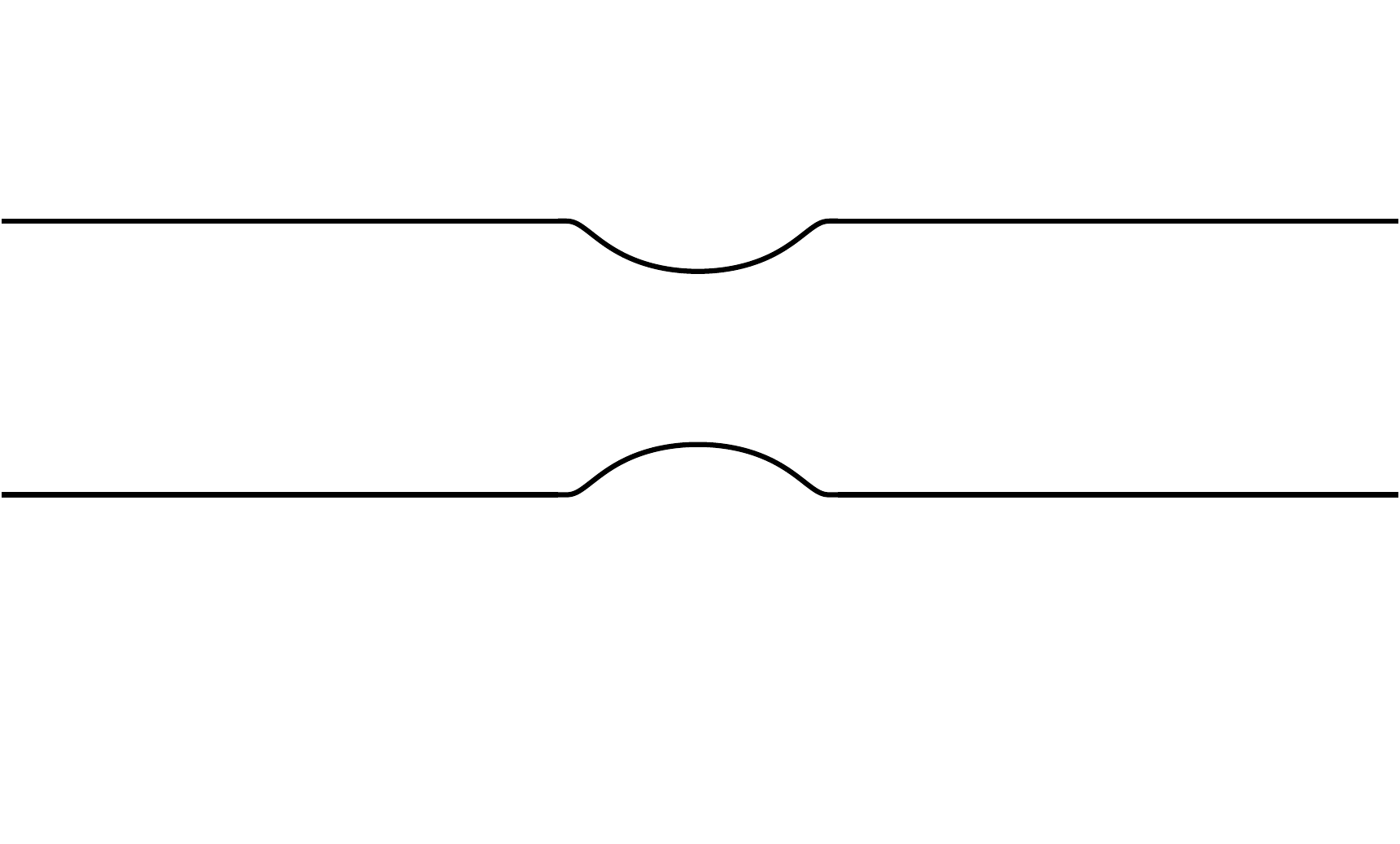}
\caption{The pinched slab $\Omega^-_\varepsilon$ of Example \ref{exa:pinchedslab}, with inward pinching. The constant $\lambda_{p,q}(\Omega^-_\varepsilon)$ is not attained.}
\end{figure}
\begin{figure}
\includegraphics[scale=.3]{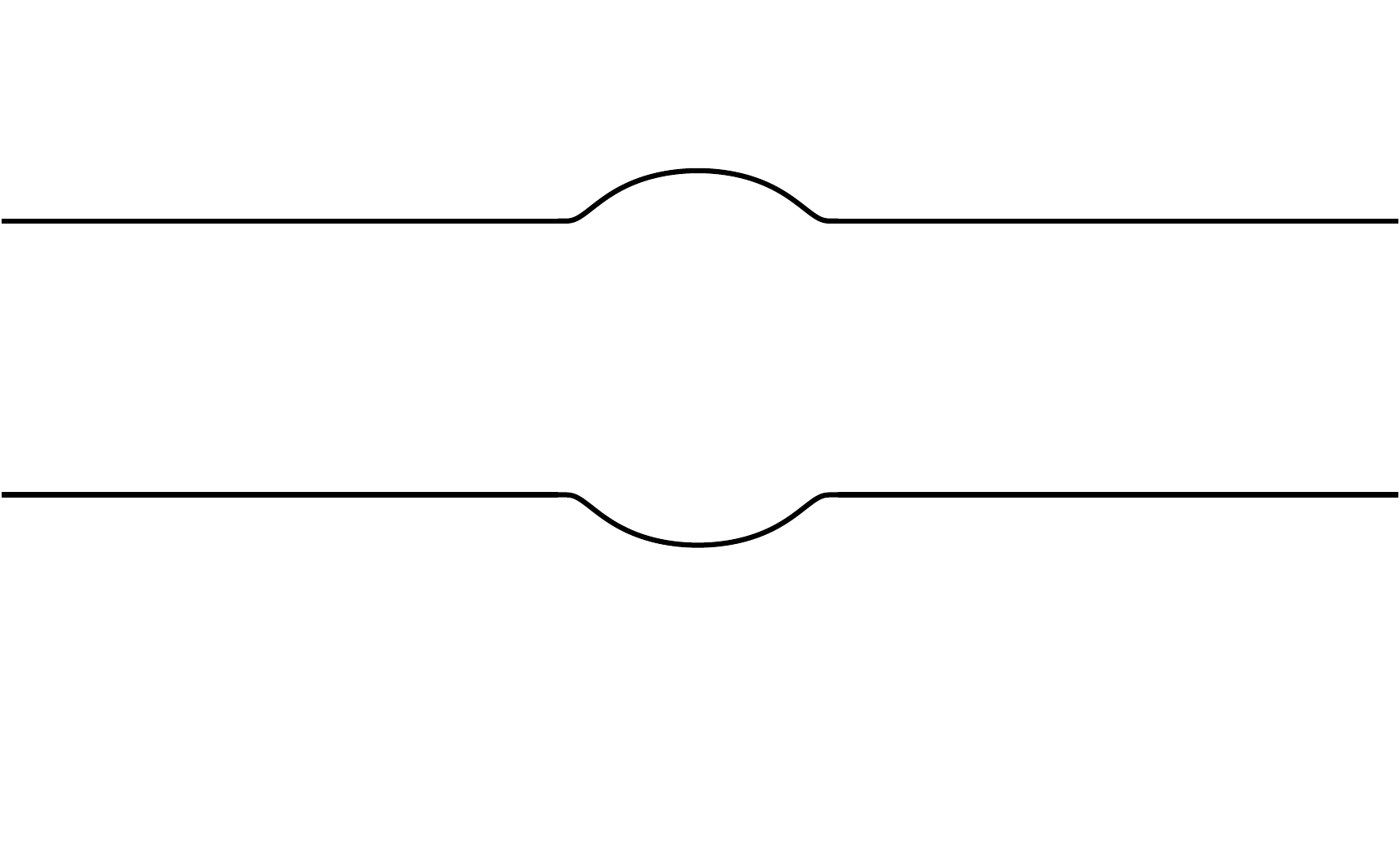}
\caption{The pinched slab $\Omega^+_\varepsilon$ of Example \ref{exa:pinchedslab}, with outward pinching. This is a Steiner symmetric set and the constant $\lambda_{p,q}(\Omega^+_\varepsilon)$ is attained, thanks to Theorems \ref{teo:main} and \ref{teo:maininfty}.}
\end{figure}
\end{exa}

\section{Proof of the Main Theorem: decay at infinity}
\label{sec:7}

In this section, we prove that the positive extremals of $\lambda_{p,q}(\Omega)$ exponentially decay to zero at infinity.  
This will be the consequence of a more general result, valid for non-negative subsolutions of the Lane-Emden equation, in generic open sets.
\subsection{Subsolutions of the Lane-Emden equation}
The following result is a classical $L^\infty-L^q$ estimate, but ``localized at infinity''. It implies a weak decay at infinity in uniform norm.
% upper bound for the $L^\infty$-norm of $u$ outside a ball  $B_{\varrho+1}$ in terms of the $L^p$-norm of $u$ outside  $B_{\varrho}$.
\begin{lm}
\label{lm:LinftyLp_localized}
Let $1<p<\infty$ and  $p\le q<\infty$ satisfying \eqref{pq}. 
Let $\Omega\subseteq\mathbb{R}^N$ be an open set such that $
\lambda_p(\Omega)>0. $
 %$p\le q<p^*$ 
Let $u\in W^{1,p}_0(\Omega)$ be a non-negative weak subsolution of
\[
-\Delta_p u=\lambda\,u^{q-1},\qquad \text{in}\ \Omega,
\]
for some $\lambda>0$. Then for every $\varrho>0$, we have 
\begin{equation}
\label{stimafuoripalla}
\|u\|_{L^\infty(\Omega\setminus B_{\varrho+1})}\le C_4\,\left(1+\lambda\,\|u\|_{L^\infty(\Omega)}^{q-p}\right)^\frac{N}{p\,q}\,\|u\|_{L^q(\Omega\setminus B_{\varrho})},
\end{equation}
for some $C_4=C_4(N,p,q)>0$.
\end{lm}
\begin{proof}
By using Lemma \ref{lm:Linfty}, we already know that $u\in L^\infty(\Omega)$. Thus, we get that $u$ verifies in particular
\begin{equation}
\label{subsolution}
\int_\Omega \langle |\nabla u|^{p-2}\,\nabla u,\nabla \varphi\rangle\,dx\le \lambda\,\|u\|_{L^\infty(\Omega\setminus B_r)}^{q-p}\,\int_\Omega u^{p-1}\,\varphi\,dx,
\end{equation}
for every $\varphi\in W^{1,p}_0(\Omega\setminus\overline{B_r})$ such that $\varphi\ge 0$ in $\Omega$. We will get the desired $L^\infty-L^q$ estimate by using a standard Moser iteration.
\par
We fix $\varrho>0$ as in the statement and, for every $\varrho\le r<R$, we take a radially symmetric Lipschitz cut-off function $\eta$ such that
\[
0\le \eta\le 1,\qquad \eta\equiv 0 \ \text{on}\ B_r,\qquad \eta\equiv 1\ \text{on}\ \mathbb{R}^N\setminus B_R,\qquad \|\nabla \eta\|_{L^\infty(\mathbb{R}^N)}=\frac{1}{R-r}.
\]
For every $\beta\ge 1$ we insert in \eqref{subsolution} the test function
\[
\varphi=\eta^p\,u^\beta,
\]
which is feasible. With simple algebraic manipulations, we get
\begin{equation}
\label{moser1}
\begin{split}
\beta\,\left(\frac{p}{p+\beta-1}\right)^p\,\int_\Omega\left |\nabla \left(u^\frac{p+\beta-1}{p}\right)\right|^p\,\eta^p\,dx&\le p\,\int_\Omega |\nabla u|^{p-1}\,|\nabla \eta|\,\eta^{p-1}\,u^\beta\,dx\\
&+\lambda\,\|u\|_{L^\infty(\Omega\setminus B_r)}^{q-p}\,\int_\Omega u^{\beta+p-1}\,\eta^p\,dx.
\end{split}
\end{equation}
By using Young's inequality, we get for every $\delta>0$
\begin{equation}
\label{conbeta}
\begin{split}
p\,\int_\Omega |\nabla u|^{p-1}\,|\nabla \eta|\,\eta^{p-1}\,u^\beta\,dx&\le (p-1)\,\delta\,\int_\Omega |\nabla u|^p\,u^{\beta-1}\,\eta^p+\delta^{1-p}\,\int_\Omega |\nabla \eta|^p\,u^{\beta+p-1}\,dx\\
&=(p-1)\,\left(\frac{p}{\beta+p-1}\right)^p\,\delta\,\int_\Omega \left|\nabla \left(u^\frac{\beta+p-1}{p}\right)\right|^p\,\eta^p\\
&+\delta^{1-p}\,\int_\Omega |\nabla \eta|^p\,u^{\beta+p-1}\,dx.
\end{split}
\end{equation}
We choose $\delta=\beta/(2\,(p-1))$, thus from \eqref{moser1} and \eqref{conbeta} we get 
\[
\begin{split}
\frac{\beta}{2}\,\left(\frac{p}{p+\beta-1}\right)^p\,\int_\Omega\left |\nabla \left(u^\frac{p+\beta-1}{p}\right)\right|^p\,\eta^p\,dx&\le \left(\frac{2\,(p-1)}{\beta}\right)^{p-1}\,\int_\Omega |\nabla \eta|^p\,u^{\beta+p-1}\,dx\\
&+\lambda\,\|u\|_{L^\infty(\Omega\setminus B_r)}^{q-p}\,\int_\Omega u^{\beta+p-1}\,\eta^p\,dx.
\end{split}
\]
This can be rewritten as
\begin{equation}
\label{moser3}
\begin{split}
\int_\Omega\left |\nabla \left(u^\frac{p+\beta-1}{p}\right)\right|^p\,\eta^p\,dx&\le \frac{2}{\beta}\,\left(\frac{p+\beta-1}{p}\right)^p\,\left(\frac{2\,(p-1)}{\beta}\right)^{p-1}\,\int_\Omega |\nabla \eta|^p\,u^{\beta+p-1}\,dx\\
&+\lambda\,\|u\|_{L^\infty(\Omega\setminus B_r)}^{q-p}\,\frac{2}{\beta}\,\left(\frac{p+\beta-1}{p}\right)^p\,\int_\Omega u^{\beta+p-1}\,\eta^p\,dx.
\end{split}
\end{equation}
We add on both sides the quantity
\[
\int_\Omega |\nabla \eta|^p\,u^{\beta+p-1}\,dx,
\]
and then use that
\[
\int_\Omega\left |\nabla \left(u^\frac{p+\beta-1}{p}\right)\right|^p\,\eta^p\,dx+\int_\Omega |\nabla \eta|^p\,u^{\beta+p-1}\,dx\ge \frac{1}{2^{p-1}}\,\int_\Omega \left|\nabla\left(\eta\,u^\frac{\beta+p-1}{p}\right)\right|^p\,dx.
\]
On account of \eqref{moser3} we obtain
\[
\begin{split}
\frac{1}{2^{p-1}}\,\int_\Omega \left|\nabla\left(\eta\,u^\frac{\beta+p-1}{p}\right)\right|^p\,dx&\le \left[1+\frac{2}{\beta}\,\left(\frac{p+\beta-1}{p}\right)^p\,\left(\frac{2\,(p-1)}{\beta}\right)^{p-1}\right]\,\int_\Omega |\nabla \eta|^p\,u^{\beta+p-1}\,dx\\
&+\lambda\,\|u\|_{L^\infty(\Omega\setminus B_r)}^{q-p}\,\frac{2}{\beta}\,\left(\frac{p+\beta-1}{p}\right)^p\,\int_\Omega u^{\beta+p-1}\,\eta^p\,dx.
\end{split}
\]
Before proceeding further, we observe that
\[
1\le \frac{1}{\beta}\,\left(\frac{p+\beta-1}{p}\right)^p\le \beta^{p-1}\qquad\text{and}\qquad \frac{2\,(p-1)}{\beta}\le 2\,p.
\]
Thus, from the estimate above we also get
\begin{equation}
\label{moser3b}
\begin{split}
\int_\Omega \left|\nabla\left(\eta\,u^\frac{\beta+p-1}{p}\right)\right|^p\,dx&\le 4^{p}\,\beta^{p-1}\,p^{p-1}\,\int_\Omega |\nabla \eta|^p\,u^{\beta+p-1}\,dx\\
&+\lambda\,\|u\|_{L^\infty(\Omega\setminus B_r)}^{q-p}\,2^p\,\beta^{p-1}\,\int_\Omega u^{\beta+p-1}\,\eta^p\,dx\\
\end{split}
\end{equation}
By recalling the properties of $\eta$, with simple manipulations we get
\begin{equation}
\label{moser4}
\begin{split}
\int_\Omega \left|\nabla\left(\eta\,u^\frac{\beta+p-1}{p}\right)\right|^p\,dx
%&\le 4^{p}\,\beta^{p-1}\,p^{p-1}\,\frac{1}{(R-r)^p}\,\int_{\mathbb{R}^N\setminus B_r}\,u^{\beta+p-1}\,dx\\
%&+\lambda\,\|u\|_{L^\infty(\Omega\setminus B_r)}^{q-p}\,2^p\,\beta^{p-1}\,\int_{\mathbb{R}^N\setminus B_r} u^{\beta+p-1}\,dx\\
&\le 4^p\,\beta^{p-1}\,p^{p-1}\,\left[\frac{1}{(R-r)^p}+\lambda\,\|u\|_{L^\infty(\Omega\setminus B_r)}^{q-p}\right]\,\int_{\Omega\setminus B_r} u^{\beta+p-1}\,dx.
\end{split}
\end{equation}
We now need to distinguish three cases, depending on whether $p < N$, $p = N$ or $p > N$.
\vskip.2cm
\noindent  
{\it Case $p<N$}:  we use the Sobolev inequality on the left-hand side of \eqref{moser4}, while simply estimating
\[
\|u\|_{L^\infty(\Omega\setminus B_r)}^{q-p}\le \|u\|_{L^\infty(\Omega)}^{q-p},
\] 
on the right-hand side. We thus end up with
\begin{equation}
\label{moser5}
\begin{split}
T_{N,p}\,\left(\int_{\Omega\setminus B_{R}} u^{\frac{\beta+p-1}{p}p^*}\,dx\right)^\frac{p}{p^*}&\le 4^p\,\beta^{p-1}\,p^{p-1}\,\left[\frac{1}{(R-r)^p}+\lambda\,\|u\|_{L^\infty(\Omega)}^{q-p}\right]\,\int_{\Omega\setminus B_r} u^{\beta+p-1}\,dx.
\end{split}
\end{equation}
For ease of notation, we introduce the parameter
\[
\vartheta=\frac{\beta+p-1}{p},
\]
and observe that $\beta\le p\,\vartheta$. Thus, from \eqref{moser5} we get
\begin{equation}
\label{moser6}
\begin{split}
\left(\int_{\Omega\setminus B_{R}} u^{\vartheta\,p^*}\,dx\right)^\frac{1}{p^*\vartheta}
&\le \left(\frac{4^p\, p^{2\,(p-1)}}{T_{N,p}}\right)^\frac{1}{p\,\vartheta}\,\left(\vartheta^\frac{1}{\vartheta}\right)^\frac{p-1}{p}\\\
&\times\left[\frac{1}{(R-r)^p}+\lambda\,\|u\|_{L^\infty(\Omega)}^{q-p}\right]^\frac{1}{p\,\vartheta}\,\left(\int_{\Omega\setminus B_r} u^{\vartheta\,p}\,dx\right)^\frac{1}{p\,\vartheta}.
\end{split}
\end{equation}
We want to use \eqref{moser6} with the following choices
\begin{equation}\label{rR}
r=r_i=\varrho+\left(1-\frac{1}{2^i}\right),\qquad R=r_{i+1}=\varrho+\left(1-\frac{1}{2^{i+1}}\right),
\end{equation}
and
\[
\vartheta_0=\frac{q}{p},\qquad  \vartheta_{i+1}=\frac{p^*}{p}\,\vartheta_i=\left(\frac{N}{N-p}\right)^{i+1}\,\frac{q}{p},\quad \text{for every}\ i\in\mathbb{N}.
\]
Starting from $i=0$ and iterating \eqref{moser6} infinitely many times, we now get the claimed conclusion,  by noticing that the sequence $\{\vartheta_i\}_{i\in\mathbb{N}}$ diverges to $+\infty$, together with the fact that
%\[
%\lim_{n\to \infty} \|u\|_{L^n(\mathbb{R}^N\setminus B_{\rho+1-\frac{1}{2^n}}(0))}= \|u\|_{L^\infty(\mathbb{R}^N\setminus B_{\rho+1}(0))},
%\]
\[
\sum_{i=0}^{\infty} \frac{1}{p\,\vartheta_i}=\frac{1}{q}\,\sum_{i=0}^{\infty} \left(\frac{N-p}{N}\right)^i=\frac{N}{p\,q},
\]
and 
\[
\lim_{n\to \infty} \prod_{i=0}^n \vartheta_i^\frac{1}{\vartheta_i}=\exp\left(\sum_{i=0}^{\infty}  \frac{\log \vartheta_i} {\vartheta_i}\right)<+\infty.
\]
\noindent  {\it Case $p=N$}:  in this case, on  the left-hand side of \eqref{moser4}  we can apply the Gagliardo-Nirenberg inequality \ref{GNS} with $p=N$ and $r=2\,N$
\[ 
L_{N}\,\left(\int_{\Omega} |\varphi|^{2\,N}\, dx\right)^{\frac{1}{2}}\leq \left(\int_{\Omega} |\nabla \varphi|^{N}\, dx\right)^{\frac {1}{2}}\, \left(\int_{\Omega} |\varphi|^{N}\, dx\right)^\frac{1}{2},\qquad \text{for every}\ \varphi\in W^{1,N}_0(\Omega).
\]
By introducing the parameter
\[
\vartheta=\frac{\beta+N-1}{N},
\]  we get
\[ \begin{split} L_{N}\left(\int_{\Omega\setminus B_R} u^{2\vartheta N}\, dx\right)^{\frac 1 2}&\leq \left\{
4^p\,\beta^{N-1}\,N^{N-1}\,\left(\frac{1}{(R-r)^N}+\lambda\,\|u\|_{L^\infty(\Omega)}^{q-N}\right) \right\}^{\frac 1 2 }\\\
&\times \left(\int_{\Omega\setminus B_r} u^{\vartheta N}\,dx \right).
\end{split}
\]
The proof proceeds as in the previous case, with the choices \eqref{rR} and 
\[
\vartheta_0=\frac{q}{N},\qquad  \vartheta_{i+1}=2 \,\vartheta_i=2^{i+1}\,\frac{q}{N},\qquad \text{for every}\ i\in\mathbb{N}.
\]
\noindent  {\it Case $p>N$}:  in this case, we apply the second Morrey inequality 
 \[ 
 \|\varphi\|_{L^\infty(\Omega)}\le C_{N,p}\,\left(\int_{\Omega} |\nabla \varphi|^p\,dx\right)^\frac{N}{p^2}\,\left(\int_{\Omega} |\varphi|^p\,dx\right)^{\frac{1}{p}-\frac{N}{p^2}},\qquad \text{for every}\, \varphi\in W^{1,p}_0(\Omega),
\] 
to the function $\varphi= \eta\,u^{q/p}$. %In particular, we have
%\[
%\begin{split}
%\int_{\Omega} |\nabla (u\,\eta)|^p\,dx&\le 2^{p-1}\,\int_\Omega |\nabla u|^p\,\eta^p\,dx+2^{p-1}\,\int_\Omega |\nabla \eta|^p\,u^p\,dx\\
%&\le 2^{p-1}\,p\,\int_\Omega \langle |\nabla u|^{p-2}\,\nabla u,\nabla \eta\rangle\,\eta^{p-1}\,u\,dx+\lambda\,\int_\Omega u^q\,\eta^p\,dx
%\end{split}
%\]
Hence,  we combine the resulting estimate with the inequality \eqref{moser4} for $\beta=q-p+1$,  $r=\varrho$ and $R=\varrho+1$. This implies the desired conclusion
\[ 
\|u \|_{L^{\infty}(\Omega\setminus B_{\varrho+1}(0))} \leq \big(C_{N,p}\big)^\frac{p}{q}\, \left(4^p\,(q-p+1)^{p-1}\,p^{p-1}\right)^\frac{N}{p\,q}\,\left(1+\lambda\,\|u\|_{L^\infty(\Omega)}^{q-p}\right)^ {\frac{N}{p\,q}} \|u\|_{L^q(\Omega\setminus B_{\varrho}(0))}.
\]
The proof is over.  
\end{proof}
\begin{rem}
\label{rem:fuorisemispazio}
It is not difficult to see that, with exactly the same proof, one can also get 
\begin{equation}
\label{stimafuorisemispazio}
\|u\|_{L^\infty(\Omega\setminus \{x\in\mathbb{R}^N\, :\, \langle x,\mathbf{e}\rangle>t+1\})}\le C_4\,\left(1+\lambda\,\|u\|_{L^\infty(\Omega)}^{q-p}\right)^\frac{N}{p\,q}\,\|u\|_{L^q(\Omega\setminus \{x\in\mathbb{R}^N\, :\, \langle x,\mathbf{e}\rangle>t\})},
\end{equation}
which holds for every $t>0$ and every $\mathbf{e}\in\mathbb{S}^{N-1}$. The only difference is in the choice of the cut-off function: 
once fixed $t>0$, for every $t\le \tau<T$ it is sufficient to take a Lipschitz cut-off function $\eta$ such that
\[
\begin{split}
&0\le \eta\le 1,\quad \eta\equiv 0 \ \text{on}\ \{x\in\mathbb{R}^N\, :\, \langle x,\mathbf{e}\rangle\le \tau\},\quad \eta\equiv 1\ \text{on}\ \{x\in\mathbb{R}^N\, :\, \langle x,\mathbf{e}\rangle\ge T\},\\
&\|\nabla \eta\|_{L^\infty(\mathbb{R}^N)}=\frac{1}{T-\tau}.
\end{split}
\]
We leave the details to the reader.
\end{rem}
We can now proceed similarly as in the proof of \cite[Theorem 5.1]{BBO} to get the claimed exponential decay, at first in $L^q$ norm and then in the sup norm. Here we crucially need the restriction $q>p$.
\begin{teo}
\label{teo:decay}   
Let $1<p<\infty$ and  $p<q<\infty$ satisfying \eqref{pq}. 
Let $\Omega\subseteq\mathbb{R}^N$ be an open set such that $
\lambda_p(\Omega)>0$.
Let $u\in W^{1,p}_0(\Omega)$ be a non-negative weak subsolution of
\[
-\Delta_p u=\lambda\,u^{q-1},\qquad \text{in}\ \Omega,
\]
for some $\lambda>0$. Then, there exist $a=a(p,q,\lambda_p(\Omega))>0$ and a constant $C_5>0$, depending on $N, p,q,\lambda,\lambda_p(\Omega)$ and $u$, such that 
%\begin{equation}
%\label{decayp}
%\|u\|_{L^q(\Omega\setminus B_R)}\le C_5\, \|u\|_{L^q(\Omega)}\, e^{-aR}, \qquad \text{for every } R>0,
%	\end{equation}
%and
\begin{equation}\label{decayinfty}
|u(x)|\le C_5\, \|u\|_{L^q(\Omega)}\, e^{-a|x|}, \qquad\text{for a.e.}\ x\in\mathbb{R}^N.
	\end{equation}
\end{teo}
\begin{proof} 
We claim that it is sufficient to prove the existence of 
 $a=a(p,q,\lambda_p(\Omega))>0$ and a constant $C_6>0$, depending on $N, p,q,\lambda,\lambda_p(\Omega)$ and $u$,  such that
\begin{equation}
\label{decayp}
\|u\|_{L^q(\Omega\setminus B_R)}\le C_6\, \|u\|_{L^q(\Omega)}\, e^{-aR}, \qquad \text{for every } R>0.
\end{equation}
Indeed, assume that \eqref{decayp} holds. For a.e. $x\in B_1$, by Lemma \ref{lm:Linfty}  we have 
\[
\begin{split}
|u(x)|\le \|u\|_{L^\infty(\Omega)}&\le C_1\,\left(\lambda^\frac{N}{p\,q}\,\|u\|_{L^q(\Omega)}\right)^\frac{p\,q}{p\,q-(q-p)\,N}\\
&\le \left(C_1\,e^a\,\lambda^\frac{N}{p\,q-(q-p)\,N}\,\|u\|_{L^q(\Omega)}^\frac{(q-p)\,N}{p\,q-(q-p)\,N}\right)\,\|u\|_{L^q(\Omega)}\,e^{-a |x|}.
\end{split}
\]
Moreover, for every $n\in \mathbb N$ and for a.e. $x\in B_{n+2}\setminus  B_{n+1}$, from Lemma \ref{lm:LinftyLp_localized} and \eqref{decayp}, we would  get  
\[
\begin{split}
|u(x)|\le \|u\|_{L^\infty(\Omega\setminus B_{n+1})}&\le C_4\,\left(1+\lambda\,\|u\|_{L^\infty(\Omega)}^{q-p}\right)^\frac{N}{p\,q}\,\|u\|_{L^q(\Omega\setminus B_{n})}\\
&\le \left(C_6\,C_4\,\left(1+\lambda\,\|u\|_{L^\infty(\Omega)}^{q-p}\right)^\frac{N}{p\,q}\,\right)\,\|u\|_{L^q(\Omega)}\, e^{-ar}\\
&\le \left(C_6\,C_4\,e^{2a}\,\left(1+\lambda\,\|u\|_{L^\infty(\Omega)}^{q-p}\right)^\frac{N}{p\,q}\,\right)\,\|u\|_{L^q(\Omega)}\, e^{-a|x|}.
\end{split}
\]
Then, \eqref{decayinfty} would follow by combining these estimates, with the constant $C_5$ equal to the maximum of the two constants above.
\par
In order to show \eqref{decayp}, we first observe that for every  $\varepsilon>0$ there exists a radius $r_\varepsilon>0$ such that
\begin{equation}
\label{piccola}
\|u\|_{L^\infty(\Omega\setminus B_R)}\le\varepsilon,\qquad \text{for every}\ R\ge r_\varepsilon.
\end{equation}
This follows from \eqref{stimafuoripalla} and the fact that $u\in L^q(\Omega)$.
For every $R\ge r_\varepsilon$, we take  a radially symmetric cut-off Lipschitz function $\eta$ such that
\[
0\le \eta\le 1,\qquad \eta\equiv 0 \ \text{on}\ B_R,\qquad \eta\equiv 1\ \text{on}\ \mathbb{R}^N\setminus B_{R+1},\qquad \|\nabla \eta\|_{L^\infty(\mathbb{R}^N)}\leq 1.
\]
From \eqref{moser3b} with the choice $\beta=q-p+1$, we thus get
\[
\begin{split}
\int_\Omega\left |\nabla \left(u^\frac{q}{p}\,\eta\right)\right|^p\,dx&\le 4^{p}\,(p\,(q-p+1))^{p-1}\,\int_\Omega |\nabla \eta|^p\,u^q\,dx\\
&+ 2^p\,(q-p+1)^{p-1}\,\lambda\,\|u\|_{L^\infty(\Omega\setminus B_R)}^{q-p}\,\int_\Omega u^q\,\eta^p\,dx.
\end{split}
\] 
%and, by applying the estimate \eqref{conbeta} with $\beta=q-p+1$ and $0<\delta<1/(p-1)$, we easily get 
%
%\begin{equation}
%\label{moser4}
%\begin{split}
%\int_\Omega\left |\nabla u\right|^p\,\eta^p\,dx&\le 2^p \,\left(p-1)\right)^{p-1}\,\int_\Omega |\nabla \eta|^p\,u^p\,dx\\
%&+2 \lambda_{p,q}(\Omega)\,\|u\|_{L^\infty(\Omega)}^{q-p}\,\int_\Omega u^p\eta^p\,dx.
%\end{split}
%\end{equation}
By using that $u^{q/p}\,\eta$  is a feasible test function for $\lambda_{p}(\Omega)$ and \eqref{piccola}, we get in particular
\[
\begin{split}
\lambda_{p}(\Omega)\, \int_{\Omega} u^q\,\eta^p\, dx&\leq C_7\,\int_\Omega |\nabla \eta|^p\,u^q\,dx+ C_7\,\lambda\,\varepsilon^{q-p}\,\int_\Omega u^q\,\eta^p\,dx,
\end{split}
\]
where $C_7=C_7(p,q)=4^{p}\,(p\,(q-p+1))^{p-1}$.
Thanks to the properties of $\eta$, this entails that
\[
\begin{split}
\lambda_{p}(\Omega)\, \int_{\Omega\setminus B_{R+1}} u^q\, dx &\leq C_7\,\int_{\Omega\cap (B_{R+1}\setminus B_R)}u^q\,dx+ C_7\,\lambda\,\varepsilon^{q-p}\,\int_{\Omega\setminus B_R} u^q\,dx\\
&=C_7\,\left(1+\lambda\,\varepsilon^{q-p}\right)\,\int_{\Omega\cap (B_{R+1}\setminus B_R)}u^q\,dx\\
&+ C_7\,\lambda\,\varepsilon^{q-p}\,\int_{\Omega\setminus B_{R+1}} u^q\,dx.
\end{split}
\]
%By recalling \eqref{piccola}, we have for $R\ge r_\varepsilon$
%\[
%\int_{\Omega\setminus B_{R+1}} u^q\,dx=\left(\int_{\Omega\setminus B_{R+1}} u^q\,dx\right)^{1-\frac{p}{q}}\,\left(\int_{\Omega\setminus B_{R+1}} u^q\,dx\right)^\frac{p}{q}<\varepsilon^{q-p}\,\left(\int_{\Omega\setminus B_{R+1}} u^q\,dx\right)^\frac{p}{q},
%\] 
%thanks to the fact that $p/q<1$. Thus, from \eqref{ridotta}  we obtain
%\[
%\begin{split}
%\varepsilon^{p-q}\,\lambda_{p,q}(\Omega)\, \int_{\Omega\setminus B_{R+1}} u^q\, dx 
%&\le C\,\left(1+\lambda\,\|u\|_{L^\infty(\Omega)}^{q-p}\right)\,\int_{\Omega\cap (B_{R+1}\setminus B_R)}u^q\,dx\\
%&+ C\,\lambda\,\|u\|_{L^\infty(\Omega)}^{q-p}\,\int_{\Omega\setminus B_{R+1}} u^q\,dx.
%\end{split}
%\]
This can be rearranged into
\[
\left(\lambda_p(\Omega)-C_7\lambda\,\varepsilon^{q-p}\right)\, \int_{\Omega\setminus B_{R+1}} u^q\, dx \le C_7\,\left(1+\lambda\,\varepsilon^{q-p}\right)\,\int_{\Omega\cap (B_{R+1}\setminus B_R)}u^q\,dx.
\]
It is finally time for declaring our choice of $\varepsilon>0$: we take
\[
\varepsilon= \left(\frac{\lambda_{p}(\Omega)}{2\,C_7\,\lambda\,}\right)^\frac{1}{q-p}.
\]  
Then, by setting $r_0:=r_{\varepsilon}$, the previous inequality implies that  
\[
\int_{\Omega \setminus B_{R+1}} u^q\, dx \leq K\,\int_{\Omega \cap(B_{R+1}\setminus B_{R})} u^q\,  dx,\qquad \text{for every}\ R\ge r_0,
\]
where we set
\[
K:=1+\frac{2\,C_7}{\lambda_p(\Omega)}.
\]
If we introduce the notation
\[
A(R)=\int_{\Omega \setminus B_{R}} u^q\, dx,
\]
the previous estimate can be rewritten as
\[
A({R+1})\leq K  \left( A(R) - A(R+1)\right), \qquad   \text{for every } R\geq r_0.
\]
This can be rearranged into the following recursive inequality
\[
A(R+1)\le \frac{K}{K+1}\,A(R),\qquad   \text{for every } R\geq r_0.
\]
By taking into account that the map $r\mapsto A(r)$ is non-increasing, this fact implies that 
\[
A({R+1})\leq \left(\frac{K }{K+1} \right)^{R-r_0} A(r_0)  \quad   \text{for every } R\geq r_0,
\]
with $A(r_0)\leq \|u\|^q_{L^q(\Omega)}$.
By setting 
\[
a= \log\left(\frac{K+1}{K} \right)^\frac{1}{q}\qquad \text{and}\qquad C_5=  \left(\frac{K }{K+1} \right)^{-\frac{{r_0+1}}{q}},
\]
we get \eqref{decayp} for $R\geq r_0+1$.
On the other hand, for $0<R<r_0+1$ it is sufficient to note that
\[
\|u\|_{L^q(\Omega\setminus B_{R})}\leq  \|u\|_{L^q(\Omega)}\leq  e^{a(r_0+1)}\, \|u\|_{L^q(\Omega)}\, e^{-aR}.
\]
Thus, by defining
\[
C_6=\max\left\{e^{a(r_0+1)},\,\left(\frac{K }{K+1} \right)^{-\frac{{r_0}+1}{q}} \right\},
\]
we finally get the claimed estimate \eqref{decayp}.
\end{proof}
\begin{rem}
\label{rem:costanti}
While the exponent $a$ in the \eqref{decayinfty} is independent of the particular subsolution $u$, we remark that the constant $C_5$ {\it does depend} on $u$ itself. Let us make this precise: an inspection of the previous proof reveals that $C_5$ depends on $u$ through
\begin{itemize}
\item its $L^q(\Omega)$ norm; 
\vskip.2cm
\item the radius $r_0$ such that 
\[
\|u\|_{L^\infty(\Omega\setminus B_{r_0})}\le \left(\frac{\lambda_{p}(\Omega)}{2\,C_7\,\lambda\,}\right)^\frac{1}{q-p},
\]
where $C_7>0$ only depends on $p,q$.
\end{itemize}
More precisely, the previous proof shows that for non-negative subsolutions such that
\[
\|u\|_{L^q(\Omega)}\le M\qquad \text{and}\qquad r_0\le \overline{R},
\] 
then we can obtain \eqref{decayinfty} with the constant $C_5$ depending only on $N,p,q,\lambda,\lambda_p(\Omega), M$ and $\overline{R}$.
\end{rem}

\subsection{Extremals}
\begin{coro}
\label{coro:decayextremals}  
Let $1<p<\infty$ and  $p<q$ satisfying \eqref{pq}. 
Let $\Omega\subseteq\mathbb{R}^N$ be an open set such that $
\lambda_p(\Omega)>0$. Assume that there exists a non-negative minimizer $u\in W^{1,p}_0(\Omega)$  of the following problem
\[
	\lambda_{p,q}(\Omega)=\min_{u\in W^{1,p}_0(\Omega)}\left\{\int_\Omega |\nabla u|^p\,dx\, :\, \|u\|_{L^q(\Omega)}=1\right\}.
\]
Then, $u$ satisfies \eqref{decayinfty}.
\end{coro}
\begin{proof}
We  need to distinguish two cases, depending on whether $q < \infty$  or $q = \infty$.
\vskip.2cm\noindent  {\it Case $q<\infty$}. 
 By minimality, we have that $u$ is a non-negative  weak solution of
\[
-\Delta_p u=\lambda_{p,q}(\Omega)\,u^{q-1},\qquad \text{in}\ \Omega.
\]
The conclusion then follows directly from Theorem \ref{teo:decay}.
\vskip.2cm\noindent
\textit{Case $q=\infty$}. This case requires to be discussed only  when $p>N$ and is simpler. Indeed, by applying Lemma \ref{lm:maxpoint},  we have that  
\[
\int_\Omega \langle |\nabla u|^{p-2}\,\nabla u,\nabla \varphi\rangle\,dx=\lambda_{p,\infty}(\Omega)\,|u(x_0)|^{p-2}\,u(x_0)\,\varphi(x_0),\qquad \mbox{ for every } \varphi\in W^{1,p}_0(\Omega),
\]
where $x_0\in\Omega$ is the unique point such that $|u(x_0)|=\|u\|_{L^\infty(\Omega)}=1$. In particular if $ R>r_0:=|x_0|$,  we have that $u$ is a $p-$harmonic function in $\Omega\setminus \overline{B_{R}}$, namely
\[
\int_\Omega \langle |\nabla u|^{p-2}\,\nabla u,\nabla \varphi\rangle\,dx=0,\qquad \text{for every}\ \varphi\in C^{\infty}_0(\Omega\setminus \overline{B_{R}}).
\]
We insert the test function $\varphi=u\,\eta$,
where $\eta \in C^{0,1}(\mathbb{R}^N)$ is the same radially symmetric cut-off function as in the proof of Theorem \ref{teo:decay}, i.e.
\[
0\le \eta\le 1,\qquad \eta\equiv 0 \ \text{on}\ B_R,\qquad \eta\equiv 1\ \text{on}\ \mathbb{R}^N\setminus B_{R+1},\qquad \|\nabla \eta\|_{L^\infty(\mathbb{R}^N)}\leq 1.
\] 
By proceeding as in the proof of \eqref{moser4}, we get
\[
\begin{split}
\int_\Omega \left|\nabla\left(\eta\,u\right)\right|^p\,dx
&\le 4^p\,p^{p-1}\,\int_{\Omega\cap (B_{R+1}\setminus B_R)} u^p\,dx.
\end{split}
\]
By using that $u\,\eta$ is a feasible test function for $\lambda_{p,\infty}(\Omega)$, we get
\begin{equation}
\label{cammafa}
\lambda_{p,\infty}(\Omega)\,\|u\|^p_{L^\infty(\Omega\setminus B_{R+1})}\le 4^p\,p^{p-1}\,\int_{\Omega\cap (B_{R+1}\setminus B_R)} u^p\,dx.
\end{equation}
This gives a $L^\infty-L^p$ decay estimate analogous to that of Lemma \ref{lm:LinftyLp_localized}.
On the other hand, by using that $u\,\eta$ is a feasible test function for $\lambda_{p}(\Omega)$, we get
\[
\lambda_{p}(\Omega)\,\int_{\Omega\setminus B_{R+1}} u^p\,dx\le 4^p\,p^{p-1}\,\int_{\Omega\cap (B_{R+1}\setminus B_R)} u^p\,dx.
\]
%\eqref{moser3}, we get\[
%\int_\Omega |\nabla u|^{p}\eta^p\, dx\le 4^p(p-1)^{p-1}\int_{\Omega\cap B_{R+1}\setminus B_{R}}  u^p\,dx.
%\]
%By summing  on both sides the quantity  $$ \int_{\Omega} u ^p |\nabla \eta |^p dx$$ and by using that $u\,\eta$ is a feasible test function for $\lambda_{p}(\Omega)$, we get
%\[  \int_{\Omega \setminus \Omega_{R+1}} u^p\, dx \leq \frac{4^{p}   \left(2^p(p-1)^{p-1} +1\right) }{ 2 \lambda_{p}(\Omega) }\left(\int_{\Omega_{R+1}\setminus \Omega_{R}} u^p  dx\right), \quad \text{for every } R>\|x_0\|,\]
%where, as above,  $\Omega_r=\Omega \cap  B_{r}$  for every  $r>0$.
We can use this iterative estimate and proceed as in the proof of Theorem \ref{teo:decay}, so to get
\begin{equation}
\label{simmate}
\|u\|_{L^p(\Omega\setminus B_R)}\le C'_7\, \|u\|_{L^p(\Omega)}\, e^{-aR}, \qquad \text{for every}\ R>0,
\end{equation}
for suitable $a,C'_7>0$. Finally, the claimed exponential decay in uniform norm can be obtained by combining \eqref{simmate} and \eqref{cammafa}. We leave the details to the reader.
\end{proof}
\begin{rem}[Decay estimate for $q=\infty$]
\label{rem:puntomax}
From the previous proof,  in the case of a non-negative extremal $u$ for $\lambda_{p,\infty}(\Omega)$, we see that the constant $C_5$ depends on $u$ itself through the radius $r_0=|x_0|$. Here $x_0$ is the unique maximum point of $u$. Thus, if we can prove that 
\[
|x_0|\le \Lambda,
\] 
then we can obtain \eqref{decayinfty} with the constant $C_5$ depending only on $N,p,q,\lambda_{p,\infty}(\Omega), \lambda_p(\Omega)$ and $\Lambda$.
\end{rem}
In the case of Steiner symmetric extremals, the quality of the previous decay estimate can be improved: the estimate {\it does not} depend on the extremal. This is the content of the following
\begin{teo}[Decay of extremals: Steiner case]
\label{teo:decaysteiner}  
Let $1<p<\infty$ and $p<q$ satisfying \eqref{pq}. 
Let $\Omega\in\mathfrak{S}^N$ be an open set such that $\Omega\not=\mathbb{R}^N$. Then every Steiner symmetric extremal of $\lambda_{p,q}(\Omega)$ satisfies \eqref{decayinfty}, with a constant $C_5$ depending only on $N,p,q$ and $\Omega$.
\end{teo}
\begin{proof}
Thanks to Theorem \ref{teo:main}, we get that there exists at least a positive Steiner symmetric extremal. Let us take $u$ a positive Steiner symmetric extremal. We preliminary notice that we can assume that $u$ is continuous on $\overline{\Omega}$: for $p>N$, this simply follows from the Sobolev--type embedding
\[
W^{1,p}_0(\Omega)\hookrightarrow C^{0,\alpha_p}(\overline\Omega),
\]
previously recalled; for $1<p\le N$, this follows from classical regularity results for solutions of elliptic PDEs, thanks to the measure density condition of Lemma \ref{lemmainradSN} (see for example \cite[Theorem 4.2]{Tru}, \cite[Theorem 2.7]{GaZi} and \cite[Theorem 4.9]{MaZi}).
\par
We already know by Corollary \ref{coro:decayextremals} that $u$ decays exponentially fast at infinity. Since $u$ is Steiner symmetric, we know that it has a maximum point at the origin: then, according to Remark \ref{rem:puntomax}, in the case $q=\infty$ there is nothing to prove. 
\par
Let us consider the case $q<\infty$. Thanks to the fact that $u$ has unit $L^q(\Omega)$ norm, in light of Remark \ref{rem:costanti} we only need to show that it is possible to choose $\overline{R}$ only depending $N,p,q$ and $\Omega$, but {\it not} on $u$ itself, such that
\[
\|u\|_{L^\infty(\Omega\setminus B_r)}\le \left(\frac{\lambda_{p}(\Omega)}{2\,C_7\,\lambda_{p,q}(\Omega)}\right)^\frac{1}{q-p}=:\mathcal{E}_\Omega,\qquad \text{for every}\ r\ge \overline{R}.
\] 
More precisely, by taking into account the symmetries of $u$, this is the same as
\[
\max_{i\in\{1,\dots,N\}} \left(\sup_{t\ge \overline{R}}u(t\,\mathbf{e}_i)\right)\le \mathcal{E}_\Omega.
\]
In order to show existence of such a radius $\overline{R}$, we need to analyze the ``behaviour at infinity'' of our set $\Omega$. To this aim, we recall the notation
\[
 \Omega_{i,t}=\{x\in\Omega\,:\,\langle x,\mathbf{e}_i\rangle =t\},\qquad\text{for}\ t\in\mathbb{R},
\]
and the definition \eqref{inradius_sect} of $r_{\Omega_{i,t}}$, for its (relative) inradius. Then, by recalling Lemma \ref{lm:luca}, for every fixed $i\in\{1,\dots,N\}$ we may have three different possibilities (see Figure \ref{fig:pinocchione}):
\begin{itemize}
\item there exists $t_i>0$ such that $r_{\Omega_{i,t}}=0$ for every $|t|\ge t_i$. In this case, our set $\Omega$ is bounded in direction $\mathbf{e}_i$;
\vskip.2cm
\item we have $r_{\Omega_{i,t}}>0$ for every $t\in\mathbb{R}$ and
\[
\lim_{|t|\to +\infty} r_{\Omega_{i,t}}=0.
\]
In this case, we say that $\Omega$ {\it shrinks at infinity in direction $\mathbf{e}_i$};
\vskip.2cm
\item finally, it may happen that
\[
\lim_{|t|\to +\infty} r_{\Omega_{i,t}}>0.
\]
In this case, we say that $\Omega$ {\it is tubular at infinity in direction $\mathbf{e}_i$}.
\end{itemize}
\begin{figure}
\includegraphics[scale=.35]{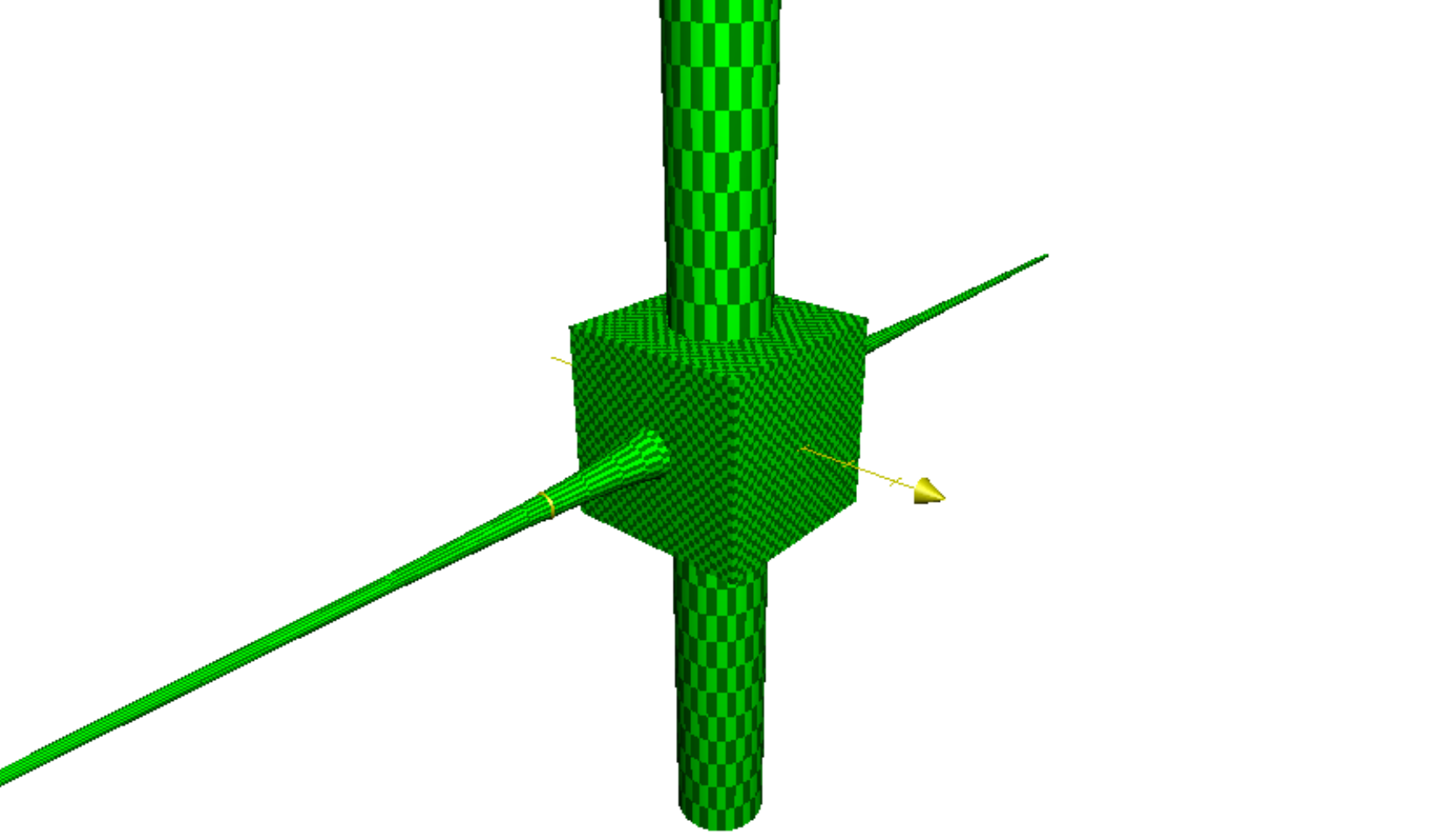}
\caption{An unbounded Steiner symmetric set in $\mathbb{R}^3$, exhibiting three different behaviours ``at infinity'', according to the coordinate directions. In particular, we have that $\Omega$ shrinks at infinity in direction $\mathbf{e}_1$; it is bounded in direction $\mathbf{e}_2$ and it is tubular at infinity in direction $\mathbf{e}_3$.}
\label{fig:pinocchione}
\end{figure}
\par
Accordingly, we now divide the rest of the proof in three cases.
\vskip.2cm\noindent
{\it Case 1: $\Omega$ is bounded in direction $\mathbf{e}_i$.} This is the simplest case. In this case, we obviously have $u(t\,\mathbf{e}_i)=0$ for every $t\ge t_i$ and thus
\[
\sup_{t\ge t_i}u(t\,\mathbf{e}_i)=0\le \mathcal{E}_\Omega;
\]
{\it Case 2: $\Omega$ shrinks at infinity in direction $\mathbf{e}_i$.}
Observe that, with a slight abuse, we can view the set $\Omega_{i,t}$ as a Steiner symmetric open set in $\mathbb{R}^{N-1}$, that is $\Omega_{i,t}\in\mathfrak{S}^{N-1}$. Moreover, thanks to Lemma \ref{lm:luca}, we know that for every $t\ge \tau_\Omega>0$, such a section has finite inradius.
In particular, in view of Lemma \ref{lemmainradSN}, Proposition \ref{prop:positivity} with $p=q$ and the definition of $r_{\Omega_{i,t}}$, for every $t\ge \tau_\Omega$ we have
\begin{equation}\label{eq:poincaresezioni}
\frac{1}{C}\,\left(\frac{1}{r_{\Omega_{i,t}}}\right)^p\int_{\Omega_{i,t}}|\psi|^p\,d\widehat{x}_i\le  \int_{\Omega_{i,t}}|\nabla \psi|^p\,d\widehat{x}_i,\qquad \text{for every}\ \psi\in C^\infty_0(\Omega_{i,t}),
\end{equation}   
for some $C=C(N,p)>0$. We used the notation $d\widehat x_i=dx_1\dots dx_{i-1}\,dx_{i+1}\dots d x_N$. 
%By \eqref{eq:tauomega}, this entails  $r_{\Omega, i}(x_i)<\infty$ for every $x_i>t$. 
Therefore, for every $\varphi\in C^{\infty}_0(\Omega)$, by using  Fubini-Tonelli's theorem and \eqref{eq:poincaresezioni}, we deduce the following Poincar\'e inequality
\[
\begin{split}
\int_{\{x\in \Omega\, :\, \langle x,\mathbf{e}_i\rangle>t\}}|\varphi|^p\,dx&=\int_t^\infty\left(\int_{\Omega_{i,x_i}}|\varphi|^p\,d\widehat x_i\right)\,dx_i\\
& \le C \int_t^{\infty} (r_{\Omega, i}(x_i))^p \left(\int_{\Omega_{i,x_i}} \left(\sum_{j\neq i} \left|\frac{\partial\varphi}{\partial x_j}\right|^2\right)^{p/2}\, d\widehat x_i\right) dx_i\\
&\le C (r_{\Omega, i}(t))^p\int_{\Omega}|\nabla \varphi|^p\, dx.
\end{split}
\]
Observe that in the last inequality we used \eqref{eq:rit}.
By density, this estimate extends to $W^{1,p}_0(\Omega)$, as well.
In particular, when applied to our extremal $u$, it gives the following decay for the $L^p$ norm in direction $\mathbf{e}_i$
\[
\int_{\{x\in \Omega\, :\, \langle x,\mathbf{e}_i\rangle>t\}}|u|^p\,dx\le C\,\lambda_{p,q}(\Omega)\, \Big(r_{\Omega, i}(t)\Big)^p.
\]
In terms of the $L^q$ norm, we easily deduce
\[
\|u\|_{L^{q}(\{x\in \Omega\, :\,\langle x,\mathbf{e}_i\rangle >t\})}\le \|u\|^{\frac{q-p}{q}}_{L^{\infty}(\Omega)} \left(\int_{\{x\in \Omega\, :\, \langle x,\mathbf{e}_i\rangle>t\}}|u|^p\,dx\right)^{\frac{1}{q}}\le \Big(C\,\lambda_{p,q}(\Omega)\Big)^{\frac{1}{q}}\, (r_{\Omega, i}(t))^{\frac{p}{q}}\, \|u\|^{\frac{q-p}{q}}_{L^{\infty}(\Omega)}.
\]
In turn, by Lemma \ref{lm:Linfty} and using the fact that $\|u\|_{L^q(\Omega)}=1$, we obtain
\[
\|u\|_{L^{q}(\{x\in \Omega\, : x_i>t\})}\le  C_1^{\frac{q-p}{q}}\left(C \lambda_{p,q}(\Omega)^{\frac{N(q-p)}{p\,q-(q-p)\,N}+1}\right)^{\frac{1}{q}}\,\Big(r_{\Omega, i}(t)\Big)^{\frac{p}{q}}.
\]
In order to reach the conclusion, we may evoke the $L^\infty-L^q$ estimate \eqref{stimafuorisemispazio} of Remark \ref{rem:fuorisemispazio} (again in conjunction with Lemma \ref{lm:Linfty}), to get
\[
\begin{split}
u((t+1)\,\mathbf{e}_i)=\|u\|_{L^{\infty}(\{x\in\Omega\, : x_i>t+1\})}\le C_8\, (r_{\Omega, i}(t))^{\frac{p}{q}},
\end{split}
\]
for some $C_8=C_8(N,p,q,\lambda_{p,q}(\Omega))>0$.
Our standing assumptions on $r_{\Omega, i}(t)$ provides the desired conclusion.
\vskip.2cm\noindent
{\it Case 3: $\Omega$ is tubular at infinity in direction $\mathbf{e}_i$.} This is the most delicate case. We will adapt and generalize a trick taken from \cite[Lemma 4.9]{BBT}, the latter being concerned with the case $p=N=2$ and $\Omega\subseteq\mathbb{R}^2$ an infinite strip.
Let us set
\[
R_i=\lim_{|t|\to +\infty} r_{\Omega_{i,t}}.
\]
Thanks to properties (S1) and (S2) in Definition \ref{S}, we have that $\Omega$ contains the infinite tube $\mathcal{T}_i(R_i)$, where we set
\[
\mathcal{T}_i(r)=\Big\{x=(x_1,\dots,x_N)\in\mathbb{R}^N\, :\, |x-x_i\,\mathbf{e}_i|<r\Big\},\qquad \text{for}\ r>0.
\]
From standard Elliptic Regularity, we know that $u$ is locally Lipschitz continuous on $\Omega$. More precisely, by \cite[Theorem 1.2]{DM} we have the following local estimate
\begin{equation}
\label{stimaDM}
\|\nabla u\|_{L^\infty(B_{r/2}(x_0))}\le C\,\left(\fint_{B_r(x_0)} |\nabla u|^p\,dx\right)^\frac{1}{p}+C\,\|\mathbf{P}(\cdot,r)\|^\frac{1}{p-1}_{L^\infty(B_r(x_0))},
\end{equation}
for every $B_r(x_0)\Subset\Omega$, where $C=C(N,p)>0$ and $\mathbf{P}(x,r)$ is a non-linear potential depending on the right-hand side of the Lane-Emden equation. In our situation, it reads as
\[
\mathbf{P}(x,r)=\int_0^r \left(\frac{\left\|\lambda_{p,q}(\Omega)\,u^{q-1}\right\|^2_{L^2(B_\varrho(x))}}{\varrho^{N-2}}\right)^\frac{1}{2}\,\frac{d\varrho}{\varrho}.
\]
By using Lemma \ref{lm:Linfty}, the latter can be easily estimated as follows
\[
\begin{split}
\mathbf{P}(x,r)&\le \lambda_{p,q}(\Omega)\, \|u\|^{q-1}_{L^\infty(\Omega)}\,\int_0^r \left(\frac{\omega_N\, \varrho^N}{\varrho^{N-2}}\right)^\frac{1}{2}\,\frac{d\varrho}{\varrho}\\
&=\lambda_{p,q}(\Omega)\, \|u\|^{q-1}_{L^\infty(\Omega)}\,\sqrt{\omega_N}\, r\le C_1^{q-1}\,\sqrt{\omega_N}\,\left(\lambda_{p,q}(\Omega)\right)^{1+\frac{N\,(q-1)}{p\,q-(q-p)\,N}}\,r.
\end{split}
\]
On the other hand, by the minimality of $u$ we have
\[
\fint_{B_r(x_0)} |\nabla u|^p\,dx\le \frac{1}{|B_r(x_0)|}\,\int_\Omega |\nabla u|^p\,dx=\frac{\lambda_{p,q}(\Omega)}{|B_r(x_0)|}.
\]
Thus, from \eqref{stimaDM} we get in particular
\[
\|\nabla u\|_{L^\infty(B_{r/2}(x_0))}\le \frac{C_8}{r^\frac{N}{p}}+C_8\,r^\frac{1}{p-1},
\]
for every $B_r(x_0)\Subset\Omega$, with $C_8>0$ depending only on $N,p,q$ and $\lambda_{p,q}(\Omega)$. We use this estimate with $r=R_i/2$ and $x_0=t\,\mathbf{e}_i$, for an arbitrary $t\in\mathbb{R}$. With the notation above, we obtain
\[
\|\nabla u\|_{L^\infty(\mathcal{T}_i(R_i/4))}=\sup_{t\in\mathbb{R}} \|\nabla u\|_{L^\infty(B_{R_i/4}(t\,\mathbf{e}_i))}\le C_8\,\left(\left(\frac{2}{R_i}\right)^\frac{N}{p}+\left(\frac{R_i}{2}\right)^\frac{1}{p-1}\right)=:C_9.
\]
This gives a uniform Lipschitz estimate on $u$ in the tube $\mathcal{T}_i(R_i/4)$. In particular, by using this fact we have
\begin{equation}
\label{bassolip}
u(x)\ge u(x_i\,\mathbf{e}_i)-C_9\,|x-x_i\,\mathbf{e}_i|,\qquad \text{for}\ x\in \mathcal{T}_i(R_i/4).
\end{equation}
We define the  positive quantity 
\[
\varphi(t)=\min\left\{\frac{R_i}{4},\, \frac{u(t\,\mathbf{e}_i)}{2\,C_9}\right\},
\]
then, from \eqref{bassolip},  we get
\[
u(x)\ge \frac{u(x_i\,\mathbf{e}_i)}{2},\qquad \text{for every}\ x\ \text{such that}\ |x-x_i\,\mathbf{e}_i|\le \varphi(x_i).
\]
By using this fact, the normalization on the $L^q(\Omega)$ norm and Fubini-Tonelli's Theorem, we get for every $t>0$
\[
\begin{split}
1=\int_{\mathbb{R}^N} u^q\,dx&\ge \int_0^t\left(\int_{\{\widehat x_i\in\mathbb{R}^{N-1}:\, |\widehat{x}_i|\le \varphi(x_i)\}} u^q\,d\widehat{x}_i\right)\,dx_i\\
&\ge \frac{\omega_{N-1}}{2^q}\,\int_0^t \varphi(x_i)^{N-1}\,u(x_i\,\mathbf{e}_i)^q\,dx_i\ge \frac{\omega_{N-1}}{2^q}\,t\,\varphi(t)^{N-1}\,u(t\,\mathbf{e}_i)^q.
\end{split}
\]
In the last inequality we also used that both $t\mapsto \varphi(t)$ and $t\mapsto u(t\,\mathbf{e}_i)$ are non-increasing. Thus, if we define the increasingly diverging function 
\[
\Phi(\tau)=\left(\min\left\{\frac{R_i}{4},\, \frac{\tau}{2\,C_9}\right\}\right)^{N-1}\,\tau^q,\qquad \text{for}\ \tau\ge 0,
\]
we have obtained the following decay estimate
\[
\Phi(u(t\,\mathbf{e}_i))\le \frac{2^q}{\omega_{N-1}}\,\frac{1}{t}\qquad \text{that is}\qquad u(t\,\mathbf{e}_i)\le \Phi^{-1}\left(\frac{2^q}{\omega_{N-1}}\,\frac{1}{t}\right).
\]
By observing that 
\[
\lim_{L\to 0^+} \Phi^{-1}(L)=0,
\]
and that the function $\Phi$ only depends on $R_i, N, p,q$ and $\lambda_{p,q}(\Omega)$, we get the desired conclusion in this case, as well.
\end{proof}
\begin{rem}
\label{rem:nolip}
With reference to the last part of the previous proof, we observe that in general we can not expect our extremals to be {\it globally} Lipschitz continuous. This is already false in the case $p=N=2$: for example, by considering the following Steiner symmetric set (see Figure \ref{fig:scalini})
\[
E=\{(x,y)\in\mathbb{R}^2\, :\, |y|< \psi(x)\},\qquad \text{with}\ \psi(x)=\sum_{i=0}^\infty \frac{1}{i+1}\,1_{[i,i+1)}(x)+\sum_{-\infty}^{j=0} \frac{1}{|j|+1}\,1_{(j-1,j]}(x),
\]
we notice that it has countably many boundary points at which we can center a cone with opening $3/2\,\pi$, entirely contained in $E$. These are the points of the form 
\[
\left\{\left(i,\pm \frac{1}{|i|+1}\right)\right\}_{i\in\mathbb{Z}\setminus\{0\}}.
\]
By using a standard barrier argument\footnote{The lower barrier can be constructed by taking a positive homogeneous harmonic function, as in \cite[Lemma 6.7.1]{BraBook} for example.}, we can easily see that the gradient of a positive solution $u\in W^{1,2}_0(E)$ to
\[
-\Delta u=\lambda\,u^{q-1},\qquad \mbox{in}\ E,
\]
must blow-up when approaching these points.
\begin{figure}
\includegraphics[scale=.3]{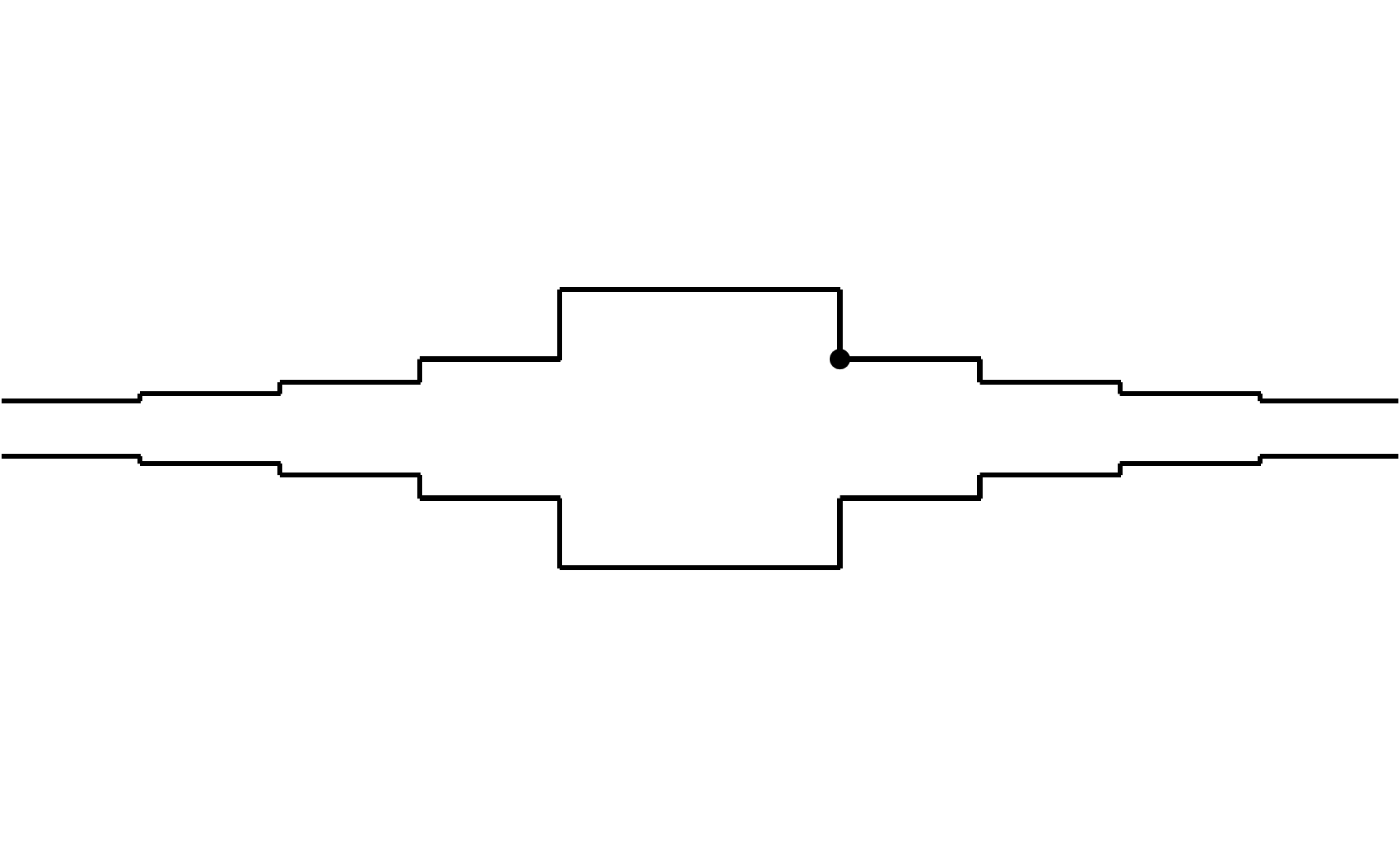}
\caption{The set $E$ in Remark \ref{rem:nolip}. Near the boundary point marked by a black dot, the gradient of the solution blows-up.}
\label{fig:scalini}
\end{figure}
\end{rem}

\appendix

\section{The endpoint case $q=\infty$}
\label{sec: app}

We start with the following simple result, about decay at infinity of Sobolev functions.
\begin{lm}
\label{lm:decay}
Let $p>N$ and let $\varphi\in W^{1,p}(\mathbb{R}^N)$. Then we have
\[
\lim_{R\to +\infty} \|\varphi\|_{L^\infty(\mathbb{R}^N\setminus B_R)}=0.
\]
In particular, there exists $x_0\in\mathbb{R}^N$ such that
\[
|\varphi(x_0)|=\|\varphi\|_{L^\infty(\mathbb{R}^N)}.
\]
\end{lm}
\begin{proof}
We first observe that 
\[
W^{1,p}(\mathbb{R}^N)\hookrightarrow C^{0,1-\frac{N}{p}}(\mathbb{R}^N),
\]
by Morrey's inequalities. Thus, the second part of the statement will easily follow, once we will establish the claimed uniform decay property.  
\par
For every $x,y\in \mathbb{R}^N$, by the first Morrey inequality, we have
\[
|\varphi(x)|\le |\varphi(x)-\varphi(y)|+|\varphi(y)|\le C\,\|\nabla \varphi\|_{L^p(\mathbb{R}^N)}\,|x-y|^{1-\frac{N}{p}}+|\varphi(y)|.
\]
We integrate the previous estimate with respect to $y\in B_r(x)$, for an arbitrary radius $r>0$. This gives
\[
\omega_N\,r^N\,|\varphi(x)|\le C\,\|\nabla \varphi\|_{L^p(\mathbb{R}^N)}\,\int_{B_r(x)} |x-y|^{1-\frac{N}{p}}\,dy+\int_{B_r(x)} |\varphi(y)|\,dy.
\]
We divide by $\omega_N\,r^N$ and use a change of variable in the first integral. We get
\[
|\varphi(x)|\le \frac{C\,r^{1-\frac{N}{p}}}{\omega_N}\,\|\nabla \varphi\|_{L^p(\mathbb{R}^N)}\,\int_{B_1(0)} |z|^{1-\frac{N}{p}}\,dz+\fint_{B_r(x)} |\varphi(y)|\,dy.
\]
In particular, by using Jensen's inequality in the first term, we get
\begin{equation}
\label{decay}
|\varphi(x)|\le \frac{C\,r^{1-\frac{N}{p}}}{\omega_N}\,\|\nabla \varphi\|_{L^p(\mathbb{R}^N)}\,\int_{B_1(0)} |z|^{1-\frac{N}{p}}\,dz+\left(\fint_{B_r(x)} |\varphi(y)|^p\,dy\right)^{\frac 1 {p}}.
\end{equation}
We now take $0<r<R$ and observe that for every $|x|>R$ we have
\[
B_r(x)\subseteq \mathbb{R}^N\setminus B_{R-r}. 
\]
Thus, we obtain
\[
\lim_{R\to +\infty} \left(\sup_{|x|>R}\fint_{B_r(x)} |\varphi|^p\,dy\right)\le \lim_{R\to+\infty} \int_{\mathbb{R}^N\setminus B_{R-r}} |\varphi|^p\,dy=0,
\]
thanks to the fact that $\varphi\in L^p(\mathbb{R}^N)$. By using this result in \eqref{decay}, we get
\[
\lim_{R\to+\infty} \left(\sup_{|x|>R} |\varphi(x)|\right)\le \frac{C\,r^{1-\frac{N}{p}}}{\omega_N}\,\|\nabla \varphi\|_{L^p(\mathbb{R}^N)}\,\int_{B_1(0)} |z|^{1-\frac{N}{p}}\,dz.
\]
By arbitrariness of $r>0$, we get the desired conclusion.
\end{proof}
The next result is concerned with the Euler-Lagrange equation for $\lambda_{p,\infty}(\Omega)$.
\begin{lm}
\label{lm:maxpoint}
Let $p>N$ and let $\Omega\subseteq\mathbb{R}^N$ be an open set such that the following problem
\[
\lambda_{p,\infty}(\Omega)=\inf_{\varphi\in W^{1,p}_0(\Omega)} \left\{\int_\Omega |\nabla \varphi|^p\,dx\, :\, \|\varphi\|_{L^\infty(\Omega)}=1\right\},
\]
admits a solution $u\in W^{1,p}_0(\Omega)$. Then we have
\[
\int_\Omega \langle |\nabla u|^{p-2}\,\nabla u,\nabla \varphi\rangle\,dx=\lambda_{p,\infty}(\Omega)\,|u(x_0)|^{p-2}\,u(x_0)\,\varphi(x_0),\qquad \mbox{ for every } \varphi\in W^{1,p}_0(\Omega),
\]
where $x_0\in\Omega$ is such that 
\begin{equation}\label{maxpoint}
|u(x_0)|=\|u\|_{L^\infty(\Omega)}=1.
\end{equation}
Moreover, such a point is unique.
\end{lm}
\begin{proof}
In the case of an open bounded set, the result is already contained in \cite[Theorem 2.5]{EP} and \cite[Lemma 2.1]{HL}. We will give here a slightly different argument, which permits to handle a more general class of open sets.
\par
We start by observing that the existence of a point $x_0\in \Omega$ satisfying \eqref{maxpoint} follows from Lemma \ref{lm:decay}. Also, by definition of $\lambda_{p,\infty}(\Omega)$, we have
\[
\frac{1}{p}\,\int_\Omega |\nabla \varphi|^p\,dx-\frac{\lambda_{p,\infty}(\Omega)}{p}\,\|\varphi\|^p_{L^\infty(\Omega)}\ge 0,\qquad \mbox{ for every } \varphi \in W^{1,p}_0(\Omega).
\]
In particular, by observing that $|\varphi(x_0)|\le \|\varphi\|_{L^\infty(\Omega)}$, we also have 
\[
\mathcal{F}(\varphi):=\frac{1}{p}\,\int_\Omega |\nabla \varphi|^p\,dx-\frac{\lambda_{p,\infty}(\Omega)}{p}\,|\varphi(x_0)|^p\ge 0,\qquad \mbox{ for every } \varphi \in W^{1,p}_0(\Omega).
\]
Moreover, we have $\mathcal{F}(u)=0$, thanks to the minimality of $u$ for $\lambda_{p,\infty}(\Omega)$ and to the property of $x_0$. In other words, the function $u$ is a minimizer for $\mathcal{F}$ over $W^{1,p}_0(\Omega)$. Thus, by computing the Euler-Lagrange equation for this functional, we end up with the claimed equation.
\par
Finally, let us suppose that there exists a second point $x_1\in \Omega$ such that 
\[
|u(x_1)|=\|u\|_{L^\infty(\Omega)}=1.
\]
Thus, the function $u$ must satisfy the same equation as before, with $x_1$ in place of $x_0$. In particular, if we take $\varphi\in W^{1,p}_0(\Omega)$ such that $\varphi(x_0)=0$ and $\varphi(x_1)\not=0$, we would get
\[
\lambda_{p,\infty}(\Omega)\,|u(x_1)|^{p-2}\,u(x_1)\,\varphi(x_1)=\int_\Omega \langle |\nabla u|^{p-2}\,\nabla u,\nabla \varphi\rangle\,dx=\lambda_{p,\infty}(\Omega)\,|u(x_0)|^{p-2}\,u(x_0)\,\varphi(x_0).
\]
Thanks to the choice of $\varphi$, this gives
\[
|u(x_1)|^{p-2}\,u(x_1)\,\varphi(x_1)=0 \qquad \mbox{ that is }\qquad u(x_1)=0,
\]
which would give a contradiction.
\end{proof}
%In the case of an open bounded set, the proof of the following result  has been  already given in 
The next result extends to general open sets, not necessarily bounded, the result of
\cite[Corollary 2.2]{HL}. %by using the Strong Maximum principle. 
Our proof uses a simpler argument. 
\begin{prop} 
\label{prop:francesca}
Let $p>N$ and let $\Omega\subseteq\mathbb{R}^N$ be an open set. For $\lambda>0$ and  $x_0\in \Omega$, let
$u\in W^{1,p}_0(\Omega)$ be a non-trivial weak solution of
\[
-\Delta_p u=\lambda\,|u(x_0)|^{p-2}\,u(x_0)\,\delta_{x_0},\qquad \text{in}\ \Omega.
\]
Then $u$ must have constant sign.
\end{prop}
\begin{proof} 
%First of all we show that $x_0$ is a maximum point of $|u|$.
%Indeed, by testing \eqref{eq:LE2}  with $\varphi=u$, we get
%\[\int_{\Omega} |\nabla u|^{p} dx=  \lambda_{p,\infty}(\Omega)\,  |u(x_0)|^p \leq  \lambda_{p,\infty}(\Omega)\,  \|u\|_{L^{\infty}(\Omega)}^p \leq \int_{\Omega} |\nabla u|^{p} dx. \]
%Hence the function $v=u/ \|u\|_{L^\infty (\Omega)}$ is a minimizer for $\lambda_{p,\infty}(\Omega)$ and, thanks to Lemma \ref{lm:maxpoint}, we have that  there exists $x_1\in \Omega$ such that \[
%|v(x_1)|=\|v\|_{L^\infty(\Omega)}=1
%\]
%and \[
%\int_\Omega \langle |\nabla v|^{p-2}\,\nabla v,\nabla \varphi\rangle\,dx=\lambda_{p,\infty}(\Omega)\,|v(x_1)|^{p-2}\,v(x_1)\,\varphi(x_1),\qquad \mbox{ for every } \varphi\in W^{1,p}_0(\Omega).
%\]
%If $x_0=x_1$, we conclude that  $x_0$ is a maximum point for $|u|$. By contradiction, assume that $x_0\not=x_1$. Then  if we take $\varphi\in W^{1,p}_0(\Omega)$ such that $\varphi(x_1)=0$ and $\varphi(x_0)\not=0$, we would get
%\[
%\lambda_{p,\infty}(\Omega)\,|u(x_1)|^{p-2}\,u(x_1)\,\| u\|_{L^{\infty}}^{1-p}\, \varphi(x_0)=\int_\Omega \langle |\nabla v|^{p-2}\,\nabla v,\nabla \varphi\rangle\,dx=\lambda_{p,\infty}(\Omega)\,|v(x_1)|^{p-2}\,v(x_1)\,\varphi(x_1)=0
%\]
%which implies the contradiction $$ \|u\|_{L^\infty(\Omega)} =|v(x_1)|  \|u\|_{L^\infty(\Omega)} =|u(x_1)|=0.$$
Without loss of generality, we assume that $u^+:=\max\{u,0\}\not=0$. It is sufficient to show that $u=u^+$. By testing the weak formulation of the equation with $\varphi=u_+\in W^{1,p}_0(\Omega)$, we get  
\[
0<\int_\Omega |\nabla u_+|^{p}\,dx=\lambda\,|u(x_0)|^{p-2}\, u(x_0)\, u_+(x_0).
\]
This implies that $u(x_0)=u_+(x_0)>0$.
Hence, by testing the weak formulation with $\varphi=u_-$, we obtain this time
\[
\int_\Omega  |\nabla u_- |^{p}\,dx=\lambda\,|u(x_0)|^{p-2}\,u(x_0)\,u_-(x_0)=0.
\]
Since $u_-\in W^{1,p}_0(\Omega)$, the previous identity shows that $u_-=0$.
 \end{proof}

\end{document}